\documentclass[12pt, a4paper]{amsart}
\usepackage{bm}
\usepackage{mathrsfs}
\usepackage{amsmath,amssymb,amsthm,booktabs,comment,longtable}
\usepackage{enumerate}
\usepackage{graphicx}
\usepackage{bbm}
\usepackage{mathtools}
\usepackage{tikz-cd}
\usepackage[all]{xy}
\theoremstyle{definition}
\newtheorem{theorem}{Theorem}[section]

\newtheorem{corollary}[theorem]{Corollary}
\newtheorem{proposition}[theorem]{Proposition}
\newtheorem{definition}[theorem]{Definition}

\newtheorem{example}[theorem]{Example}
\newtheorem{remark}[theorem]{Remark}

\newtheorem{setting}[theorem]{Setting}

\newcommand{\defarrow}{\mathrel{\stackrel{\mathrm{def}}{\longleftrightarrow}}}
\newcommand{\Low}{\mathop{\mathrm{Low}}\nolimits}
\newcommand{\open}{\mathop{\mathrm{open}}\nolimits}

\newcommand{\Ad}{\mathop{\mathrm{Ad}}\nolimits}

\newcommand{\id}{\mathop{\mathrm{id}}\nolimits}

\newcommand{\Mor}{\mathop{\mathrm{Mor}}\nolimits}
\newcommand{\Ob}{\mathop{\mathrm{Ob}}\nolimits}

\newcommand{\Z}{\mathbb{Z}}

\newcommand{\R}{\mathbb{R}}

\newcommand{\Coarse}{\mathop{\mathbf{Coarse}}\nolimits}

\newcommand{\Rel}{\mathrm{Rel}}
\newcommand{\map}{\mathrm{Map}}
\newcommand{\TotRel}{\mathrm{TotRel}}
\newcommand{\CtrMap}{\mathrm{CtrMap}}
\newcommand{\CtrTotRel}{\mathrm{CtrTotRel}}

\newcommand{\CtrRel}{\mathrm{CtrRel}}

\newcommand{\close}{\mathrm{close}}

\newcommand{\Dom}{\mathrm{Dom}}

\newcommand{\Mono}{\mathrm{Mono}}
\newcommand{\AS}{\mathrm{ASymm}}
\newcommand{\PreOrd}{\mathbf{PreOrd}}
\newcommand{\ParOrd}{\mathbf{Par}}
\newcommand{\hPreOrd}{\mathbf{hPreOrd}}
\newcommand{\CompJoinPreOrd}{\mathbf{CJ}\text{-}\mathbf{PreOrd}}
\newcommand{\CompJoinParOrd}{\mathbf{CJ}\text{-}\mathbf{Par}}
\newcommand{\hCompJoinPreOrd}{\mathbf{CJ}\text{-}\mathbf{hPreOrd}}
\newcommand{\FinJoinPreOrd}{\mathbf{FJ}\text{-}\mathbf{PreOrd}}
\newcommand{\FinJoinParOrd}{\mathbf{FJ}\text{-}\mathbf{Par}}
\newcommand{\hFinJoinPreOrd}{\mathbf{FJ}\text{-}\mathbf{hPreOrd}}
\newcommand{\FCCCat}{\mathbf{FCC}\text{-}\mathbf{Cat}}
\newcommand{\FCCThinCat}{\mathbf{FCC}\text{-}\mathbf{ThinCat}}
\newcommand{\FCChCat}{\mathbf{FCC}\text{-}\mathbf{hCat}}
\newcommand{\Set}{{\mathbf{Set}}}
\newcommand{\Cat}{{\mathbf{Cat}}}
\newcommand{\ThinCat}{{\mathbf{ThinCat}}}
\newcommand{\hCat}{{\mathbf{hCat}}}
\newcommand{\hThinCat}{{\mathbf{hThinCat}}}
\newcommand{\FCChThinCat}{{\mathbf{FCC}\text{-}\mathbf{hThinCat}}}

\newcommand{\CSlice}{\mathrm{Slice}^\mathcal{C}}
\newcommand{\MSlice}{\mathcal{M}\text{-}\mathrm{Slice}^\mathcal{C}}
\newcommand{\Sub}{\operatorname{Sub}}

\newcommand{\Upper}{\mathrm{Upper}}

\newcommand{\adperp}{\mathrel{\perp_{\mathrm{ad}}}}

\newcounter{mycounter}

\begin{document}

\title[Some categorical remarks on coarse subspaces]{Some categorical remarks on coarse subspaces of coarse spaces}
\author[M.~Miyaji, H.~Nagaya, K.~Ogawa, T.~Okuda]{Muneto Miyaji, Hiroaki Nagaya, Kento Ogawa, Takayuki Okuda}
\subjclass[2020]{
Primary 51F30, %Lipschitz and coarse geometry of metric spaces
Secondary 
18A25, %Functor categories, comma categories.
18A32. %Factorization systems, substructures, quotient structures, congruences, amalgams 
} 
\keywords{Coarse space, Slice category, Orthogonal factorization system}
\thanks{
The first and second authors are supported by JST SPRING, Grant Number JP-MJSP2132. 
The fourth author is supported by JSPS Grants-in-Aid for Scientific Research JP20K03589, JP20K14310, and JP22H0112.
This work was supported by the Research Institute for Mathematical Sciences,
an International Joint Usage/Research Center located in Kyoto University.
}

\address[M.~Miyaji]{%
	Graduate School of Advanced Science and Engineering, Hiroshima University, 
    1-3-1 Kagamiyama, Higashi-Hiroshima City, Hiroshima, 739-8526, Japan.
        }
\email{miyajimuneto@hiroshima-u.ac.jp}
\address[H.~Nagaya]{%
	Graduate School of Advanced Science and Engineering, Hiroshima University, 
    1-3-1 Kagamiyama, Higashi-Hiroshima City, Hiroshima, 739-8526, Japan.
        }
\email{nagayahiroaki@hiroshima-u.ac.jp}
\address[K.~Ogawa]{%
	Graduate School of Advanced Science and Engineering, Hiroshima University, 
    1-3-1 Kagamiyama, Higashi-Hiroshima City, Hiroshima, 739-8526, Japan.
        }
\email{knt-ogawa@hiroshima-u.ac.jp}
\address[T.~Okuda]{%
	Graduate School of Advanced Science and Engineering, Hiroshima University, 
    1-3-1 Kagamiyama, Higashi-Hiroshima City, Hiroshima, 739-8526, Japan.
        }
\email{okudatak@hiroshima-u.ac.jp}

\maketitle

\begin{abstract}
In this paper, we provide a categorical framework for understanding coarse
subspaces of coarse spaces.
First, we introduce the notion of a controlled total relation between coarse
spaces and show that the category whose morphisms are closeness classes of
controlled total relations is isomorphic to the conventional category of
coarse spaces defined using closeness classes of controlled maps.
Next, we show that the assignment associating to each coarse space the
finite-join partially ordered set of its coarse subspaces is functorial, and
prove that this partially ordered set is naturally isomorphic to the poset of
subobjects in the category of coarse spaces.
Furthermore, we formulate asymptotic disjointness between coarse subspaces
and show that mono-morphisms preserve this relation.
These results provide a categorical interpretation of the framework of
coarse subspaces introduced by 
Leitner--Vigolo [Lecture Notes in Math.~(2023)] 
and characterize coarse subspaces as objects
intrinsic to the category of coarse spaces.
They also provide a foundation for a coarse-geometric interpretation of the
properness criterion established by Kobayashi
[Math.~Ann.~(1989); J.~Lie Theory (1996)]
and Benoist
[Ann.~of Math.~(1996)]
(cf.~Nagaya--Ogawa--Okuda
[Proc.~Japan Acad.~Ser.~A (2025)]).
\end{abstract}

\section{Introduction}

Coarse geometry was developed as a theory for studying the large-scale geometric structure of spaces.
The viewpoint of regarding finitely generated groups as metric spaces equipped with word metrics and studying their large-scale geometry up to quasi-isometry has its origins in results such as the Švarc--Milnor lemma
\cite{Milnor1968Curvature, Svarc1955}.
Gromov \cite{Gromov1987,Gromov1993} developed this viewpoint into a central framework of geometric group theory and had a decisive influence on the development of large-scale geometry and coarse geometry through asymptotic invariants such as growth and asymptotic dimension.
Within this broader development of large-scale geometric perspectives,
Roe \cite{Roe1993,Roe1996IndexTheory,Roe2003LectureCoarse} developed
coarse geometry against the background of index theory on noncompact
manifolds and further established the notion of a coarse space as a
framework for describing large-scale geometry beyond the setting of
metric spaces.
Since then, coarse geometry of metric spaces and, more generally, coarse
spaces has developed in close connection with various areas, including
geometric group theory, operator algebras, and topology
(\cite{ArzhantsevaGuentnerSpakula2012, BellDranishnikov2008, HigsonRoe1995, HigsonRoeYu1993, Nowak2007, Nowak2015, NowakYu2023, Yu1998, Yu2000}).
Moreover, recent work by the authors
\cite{Nagaya2026,NagayaOgawaOkuda2025,OgawaOkuda2025Kob60} has revealed
that coarse geometry is also closely related to the properness of
topological group actions and to the theory of discontinuous groups on
homogeneous spaces with noncompact isotropy subgroups.

In this paper, we focus on coarse subspaces of coarse spaces.
For a coarse space $X$, the coarse structure naturally induces an equivalence relation on the power set $\mathcal{P}(X)$, and each equivalence class is called a \emph{coarse subspace} of $X$, following Leitner--Vigolo \cite{LeitnerVigolo2023}.
Mart\'{i}nez--Vigolo \cite{MartinezVigolo2026roe} developed a framework for constructing Roe-type operator algebras over general coarse spaces by means of coarse geometric modules.
In this framework, for coarse spaces $X$ and $Y$, the coarse support of a bounded operator between coarse geometric modules is formulated as a coarse subspace of $Y\times X$.
In particular, although the support of an operator cannot in general be defined naturally in a topological sense over an arbitrary coarse space, its coarse support can be defined as a coarse subspace of $Y\times X$.
This framework makes it possible to formulate operator-theoretic conditions such as controlled propagation and quasi-locality in terms of coarse subspaces of general coarse spaces.
Thus, their work demonstrates that coarse subspaces are not merely coarse analogues of subsets, but rather fundamental geometric objects for describing Roe-type operator algebras over general coarse spaces.

In this paper, we aim to develop a categorical framework for coarse subspaces of coarse spaces.
The category whose objects are coarse spaces and whose morphisms are closeness classes of controlled maps is commonly used as a setting for coarse geometry
(cf.~\cite{DikranjanZava2017,Zava2019CoarseCat}).
We introduce the notion of a controlled total relation between coarse spaces and construct, without using controlled maps, a realization of a category isomorphic to the one described above (Theorem \ref{theorem:CtrMapClo_isom_CtrTotRelClo}).
Hereafter, we refer to either of these isomorphic categories simply as the ``category of coarse spaces.''
We further show that the assignment associating with each coarse space the partially ordered set of its coarse subspaces is functorial (Theorem \ref{theorem:etaPc}) and that coarse subspaces are equivalent to subobjects in the category of coarse spaces (Theorem \ref{theorem:PcSub}).

We also discuss the property of coarse subspaces being ``asymptotically disjoint''
(cf.~\cite{Dranishnikov2001,Protasov2003}).
In particular, we show that mono-morphisms in the category of coarse spaces preserve the relation of being asymptotically disjoint (Theorem \ref{theorem:properAD}).
This result was announced in \cite{NagayaOgawaOkuda2025} and constitutes an important ingredient in providing a coarse-geometric interpretation of the properness criteria for group actions on homogeneous spaces established by Kobayashi
\cite{Kobayashi89,Kobayashi96} and Benoist \cite{Benoist96}.

The organization of this paper is as follows.
Section \ref{section:category} collects the categorical conventions and notation used throughout the paper.
After reviewing the basic notions concerning coarse spaces, controlled maps, coarse equivalences, and coarse subspaces in Section \ref{section:prelim_coarsespaces}, we introduce controlled total relations in Section \ref{section:realizationCtrTotRel} and use them to obtain a realization of the category of coarse spaces.
The resulting framework is then applied in Section \ref{section:Fct_coarsesubset} to establish the functoriality of the assignment that associates with each coarse space $X$ the partially ordered set of its coarse subspaces.
Section \ref{section:subspace_subobject} clarifies the relationship between this partially ordered set and the partially ordered set of subobjects of $X$ in the category of coarse spaces.
We conclude in Section \ref{section:asymptoticallydisjoint} by formulating asymptotic disjointness for coarse subspaces and proving its invariance under mono-morphisms.

\section{Preliminaries on category theory}\label{section:category}

In this section, we collect the categorical terminology and conventions used
throughout the paper.

\subsection{Notations on category theory}

Since the arguments in this paper involve certain set-theoretically large
constructions, we first fix the following setting.

\begin{setting}\label{setting:UVC}
Throughout this paper, we work in ZFC together with the additional axiom:
For every set $X$, there exists a universe $\mathcal U$ containing $X$.
We fix universes $\mathcal U$ and $\mathcal{V}$ satisfying
$\mathbb N\in\mathcal U\in\mathcal{V}$
(see \cite[Expos\'e I]{SGA41} for the definition of a universe).
\end{setting}

Throughout this paper, a category $\mathcal{C}$ is said to be
\emph{$\mathcal{V}$-small} if
$\Ob(\mathcal{C})\in\mathcal{V}$ and
$\Mor(\mathcal{C})\in\mathcal{V}$.
Here $\Ob(\mathcal{C})$ and $\Mor(\mathcal{C})$ denote the classes of all
objects and all morphisms of $\mathcal{C}$, respectively.

A category $\mathcal{C}$ is said to be \emph{locally $\mathcal{V}$-small} if
$\mathcal{C}(X,Y)\in\mathcal{V}$
for every pair of objects $X,Y\in\mathcal{C}$.
Here $\mathcal{C}(X,Y)$ denotes the set of all morphisms from $X$ to $Y$ in
$\mathcal{C}$.
Furthermore, a category $\mathcal{C}$ is said to be \emph{thin} if
$\#\mathcal{C}(X,Y)\le 1$
for every pair of objects $X,Y\in\mathcal{C}$.

For each pair $(X,Y)$ of objects of $\mathcal{C}$, we write an element $f$ of $\mathcal{C}(X,Y)$ as $f : X \rightarrow_{\mathcal{C}} Y$, or simply as $f : X \rightarrow Y$. We also occasionally write $s(f)$ and $t(f)$ for the source $X$ and the target $Y$ of $f$, respectively.

Let $\Cat$ denote the category whose objects are $\mathcal{V}$-small
categories and whose morphisms are functors.
We write $\ThinCat$ for the full subcategory of $\Cat$ consisting of thin
categories.

We consider the congruence on $\Cat$ [resp.~$\ThinCat$] induced by natural
isomorphisms and denote the corresponding quotient category by
$\hCat$ [resp.~$\hThinCat$].
The category $\hThinCat$ is naturally a full subcategory of $\hCat$.

We also write $\FCCCat$ [resp.~$\FCCThinCat$] for the subcategory of
$\Cat$ [resp.~$\ThinCat$] whose objects are finitely cocomplete
$\mathcal{V}$-small categories
[resp.~finitely cocomplete $\mathcal{V}$-small thin categories]
and whose morphisms preserve finite colimits.
The corresponding quotient categories by natural isomorphisms will be
denoted by $\FCChCat$ and $\FCChThinCat$, respectively.

All categories introduced above are locally $\mathcal{V}$-small.
The inclusions and quotient functors are summarized in the following diagram:
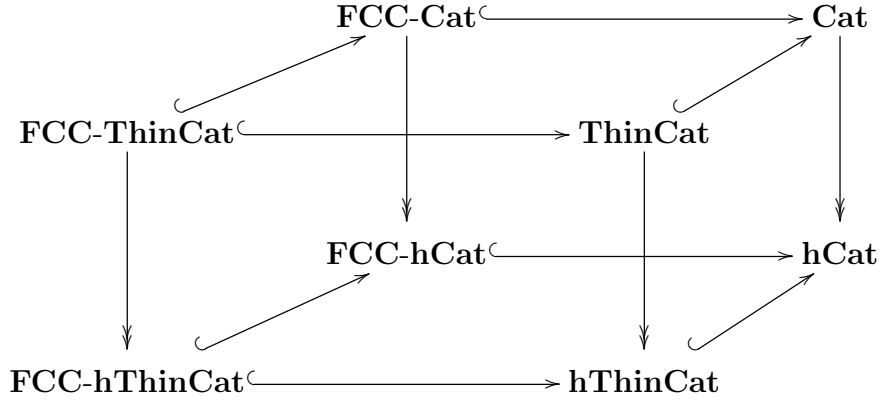
\[
\xymatrix@C=2em@R=2.4em{
&
\FCCCat \ar@{^{(}->}[rr]
         \ar@{->>}[dd]
&&
\Cat \ar@{->>}[dd]
\\
\FCCThinCat \ar@{^{(}->}[ur]
             \ar@{^{(}->}[rr]
             \ar@{->>}[dd]
&&
\ThinCat \ar@{^{(}->}[ur]
            \ar@{->>}[dd]
&
\\
&
\rule{0pt}{2.5ex}\FCChCat \ar@{^{(}->}[rr]
&&
\rule{0pt}{2.5ex}\hCat
\\
\rule{0pt}{2.5ex}\FCChThinCat \ar@{^{(}->}[ur]
                               \ar@{^{(}->}[rr]
&&
\rule{0pt}{2.5ex}\hThinCat \ar@{^{(}->}[ur]
&
}
\]

\subsection{Some categories of pre-ordered sets}\label{subsection:pre-order_category}

Several categories related to pre-ordered sets will appear throughout this
paper. In this subsection, we collect the necessary terminology.

Let $P$ be a $\mathcal{V}$-set (i.e.~$P \in \mathcal{V}$). A binary relation $\prec$ on $P$ is called a
\emph{pre-order} if it satisfies the following conditions:
\begin{enumerate}
    \item $x \prec x$ for each $x\in P$;
    \item $x \prec y$ and $y \prec z$ imply $x \prec z$ for each
    $x,y,z\in P$.
\end{enumerate}
A $\mathcal{V}$-set equipped with a pre-order, denoted by
$P=(P,\prec)$, is called a \emph{pre-ordered $\mathcal{V}$-set}, or simply a
\emph{pre-ordered set}.

Let $P=(P,\prec)$ and $Q=(Q,\prec)$ be pre-ordered sets.
A map $f:P\to Q$ is said to be \emph{order-preserving} if
$f(x)\prec f(y)$ whenever $x,y\in P$ satisfy $x\prec y$.
We write $\PreOrd(P,Q)$ for the set of all order-preserving maps from $P$
to $Q$.
We also write $\PreOrd$ for the locally $\mathcal{V}$-small category whose
objects are pre-ordered sets and whose morphisms are order-preserving maps.

\begin{definition}\label{definition:simonpre}
For a pre-ordered set $(P,\prec)$, we define an equivalence relation $\sim$
on $P$ by
\[
p_1\sim p_2
\defarrow
p_1\prec p_2 \text{ and } p_2\prec p_1.
\]
\end{definition}

For pre-ordered sets $P=(P,\prec)$ and $Q=(Q,\prec)$, we define an
equivalence relation $\sim$ on $\PreOrd(P,Q)$ by
\[
f\sim g
\defarrow
f(x)\sim g(x)\text{ for each }x\in P
\]
for $f,g\in\PreOrd(P,Q)$.    
This defines a congruence on the locally $\mathcal{V}$-small category
$\PreOrd$.
We denote the corresponding quotient category by $\hPreOrd$.
For each $f\in\PreOrd(P,Q)$, we write $[f]$ for the equivalence class
represented by $f$.
We say that $f \in \PreOrd(P,Q)$ is an equivalence if $[f]$ is an isomorphism in $\hPreOrd$.

Let us summarize the relationship between $\mathcal{V}$-small thin categories
and pre-ordered sets.

For a $\mathcal{V}$-small thin category $\mathcal{C}$, define a binary relation
$\prec$ on $\Ob(\mathcal{C})$ by
\[
x\prec y
\defarrow
\text{there exists a morphism from }x\text{ to }y\text{ in }\mathcal{C}.
\]
Then $(\Ob(\mathcal{C}),\prec)$ is a pre-ordered set. Henceforth, by a slight abuse of notation, we shall say, for example, that ``$\mathcal{C} = (\Ob(\mathcal{C}),\prec)$ is a pre-ordered set.''

Let $\mathcal{C}_1$ and $\mathcal{C}_2$ be $\mathcal{V}$-small thin categories,
and let $F:\mathcal{C}_1\to\mathcal{C}_2$ be a functor.
The restriction of $F$ to $\Ob(\mathcal{C}_1)$ gives an order-preserving map
between pre-ordered sets $\mathcal{C}_1$ and $\mathcal{C}_2$.

As is well known, the correspondences above 
give an equivalence of categories
$\hThinCat\to\hPreOrd$.
Accordingly, throughout this paper, we shall treat
$\mathcal{V}$-small thin categories and pre-ordered sets as equivalent
notions.

The well-known characterization that a functor is an equivalence if and only
if it is fully faithful and essentially surjective
(see, for example, \cite[Theorem 1.5.9]{Riehl2016category})
translates, under the above equivalence, into the following statement.

\begin{proposition}\label{proposition:preord_hisom}
Let $f\in\PreOrd(P,Q)$.
Then $[f]$ is an isomorphism in $\hPreOrd$ if and only if the following
conditions hold:
\begin{enumerate}
    \item For every $x,y\in P$, $f(x)\prec f(y)$ implies $x\prec y$.
    \item For every $q\in Q$, there exists $x\in P$ such that
    $f(x)\sim q$.
\end{enumerate}
\end{proposition}

We next summarize the relationship between partially ordered sets and
pre-ordered sets.

A pre-ordered set $(P,\prec)$ is said to be a \emph{partially ordered set}
if it satisfies the following condition:
\begin{itemize}
    \item $x\prec y$ and $y\prec x$ imply $x=y$ for each $x,y\in P$.
\end{itemize}
In this case, the relation $\prec$ is called a \emph{partial order}.
We write $\ParOrd$ for the full subcategory of $\PreOrd$ consisting of
partially ordered sets.

If $Q$ is a partially ordered set, then the equivalence relation $\sim$ on
$Q$ is trivial (i.e.~$x\sim y$ if and only if $x=y$).
Moreover, for every pre-ordered set $P$, the equivalence relation $\sim$ on
$\PreOrd(P,Q)$ is also trivial (i.e.~$f\sim g$ if and only if $f=g$).
Therefore, $\ParOrd$ may also be regarded as a full subcategory of
$\hPreOrd$.

For a pre-ordered set $P=(P,\prec)$, we write
$\AS(P):=P/\sim$
for the quotient set by the equivalence relation $\sim$ defined above, and call $\AS(P)$ the anti-symmetrization of $P$.
For each $x\in P$, let $[x]\in\AS(P)$ denote the $\sim$-equivalence class
represented by $x$.
The set $\AS(P)$ naturally inherits a partial order; namely,
\[
[x]\prec [y]
\defarrow
x\prec y.
\]
This is independent of the choice of representatives and defines a partial
order on $\AS(P)$.
The quotient map $\pi:P\to\AS(P)$ is order-preserving, and $[\pi]$ defines
an isomorphism in $\hPreOrd$.

Furthermore, for each $[f]\in\hPreOrd(P,Q)$, define
\[
\AS([f]):\AS(P)\to\AS(Q),\quad [x]\mapsto [f(x)].
\]
Then this construction defines a fully faithful functor
$\AS:\hPreOrd\to\ParOrd$, which is left adjoint to the inclusion functor
$\ParOrd\hookrightarrow\hPreOrd$.
In particular, $\ParOrd$ is a reflective subcategory of $\hPreOrd$
(see \cite[Definition 4.6.12]{Riehl2016category} for the terminology of
reflective subcategories).

In what follows, let $P=(P,\prec)$ be a pre-ordered set and $S$ a subset of
$P$. We define the set of upper bounds $\Upper(S)$ and the set of least
upper bounds $\bigvee S$ by
\begin{align*}
    \Upper(S)
    &:=
    \{x\in P\mid s\prec x\text{ for each }s\in S\},\\
    \bigvee S
    &:=
    \{x\in\Upper(S)\mid x\prec y\text{ for each }y\in\Upper(S)\}.
\end{align*}
Although the elements of $\bigvee S$ need not be unique, they determine a
single $\sim$-equivalence class in $P$.
In particular, if $P$ is a partially ordered set, then
$\#(\bigvee S)\le 1$.
Thus, by abuse of notation, we shall often regard $\bigvee S$ as an element
of $P$ whenever $\bigvee S\neq\emptyset$.

In this paper, a pre-ordered set $P$ is said to be \emph{complete join}
[resp.~\emph{finite join}] if every subset [resp.~every finite subset] of
$P$ has a non-empty set of least upper bounds.
When a pre-ordered set is regarded as a thin $\mathcal{V}$-small category,
the condition of being complete join [resp.~finite join] corresponds to the
condition of being cocomplete [resp.~finitely cocomplete].

Let $P$ and $Q$ be complete join [resp.~finite join] pre-ordered sets.
An order-preserving map $f:P\to Q$ is said to preserve arbitrary joins
[resp.~finite joins] if
\[
f(\bigvee S) \subset \bigvee f(S)
\]
for every subset [resp.~every finite subset] $S\subset P$.
We write $\CompJoinPreOrd$ [resp.~$\FinJoinPreOrd$] for the category whose
objects are complete join [resp.~finite join] pre-ordered sets and whose
morphisms are order-preserving maps preserving arbitrary joins
[resp.~finite joins].

These categories satisfy
\[
\CompJoinPreOrd \hookrightarrow
\FinJoinPreOrd \hookrightarrow
\PreOrd.
\]
The congruence $\sim$ on $\PreOrd$ defined above induces congruences on
$\CompJoinPreOrd$ and $\FinJoinPreOrd$.
We denote the corresponding quotient categories by
$\hCompJoinPreOrd$ and $\hFinJoinPreOrd$, respectively.
The inclusions induce
\[
\hCompJoinPreOrd \hookrightarrow
\hFinJoinPreOrd \hookrightarrow
\hPreOrd.
\]
We shall use the following elementary observation later.

\begin{proposition}\label{proposition:hpreisom_FJ}
Let $P$ and $Q$ be finite-join partially ordered sets, and let
$\Phi \in \PreOrd(P,Q)$. If $[\Phi] \in \hPreOrd(P,Q)$ is an isomorphism,
then $\Phi \in \FinJoinPreOrd(P,Q)$.
\end{proposition}

We write $\CompJoinParOrd$ [resp.~$\FinJoinParOrd$] for the full
subcategory of $\CompJoinPreOrd$ [resp.~$\FinJoinPreOrd$] consisting of
complete join [resp.~finite join] partially ordered sets.
These may naturally be regarded as reflective subcategories of
$\hCompJoinPreOrd$ and $\hFinJoinPreOrd$, respectively, and the left
adjoints to the inclusion functors are given by the restriction of the
functor $\AS$ defined above.
The relationships among these categories are summarized in the following
diagram:
\[
\xymatrix@C=5em@R=4em{
\hCompJoinPreOrd \ar@{^{(}->}[r]
&
\hFinJoinPreOrd \ar@{^{(}->}[r]
&
\hPreOrd
\\
\rule{0pt}{3ex}\CompJoinParOrd \ar@{^{(}->}[r]
                               \ar@{^{(}->}[u]
&
\rule{0pt}{3ex}\FinJoinParOrd \ar@{^{(}->}[r]
                              \ar@{^{(}->}[u]
&
\rule{0pt}{3ex}\ParOrd \ar@{^{(}->}[u]
}
\]

\subsection{Open Maps between Pre-ordered Sets}

We record the terminology concerning open maps between pre-ordered sets.

Let $P = (P,\leq)$ and $Q = (Q,\leq)$ be pre-ordered sets.

\begin{definition}\label{definition:open}
We use the following terminology:
\begin{itemize}
    \item A subset $L$ of $P$ is called \emph{lower} if, whenever $x \in L$ and $a \leq x$, one has $a \in L$.
    \item An order-preserving map $f : P \rightarrow Q$ is called \emph{open} if, for every lower subset $L$ of $P$, the image $f(L)$ is lower in $Q$.
\end{itemize}
\end{definition}

\begin{remark}
The term ``open'' for maps between pre-ordered sets originates from the Alexandrov topology associated with a pre-ordered set. In this paper, by the Alexandrov topology on a pre-ordered set, we mean the topology whose open sets are precisely the lower subsets. For the Alexandrov topology, see, for example, \cite{Johnstone1986stonespaces}.
\end{remark}

The following follows immediately from the definition:

\begin{proposition}
If $f : P \rightarrow Q$ is open, then $\AS([f]) : \AS(P) \rightarrow \AS(Q)$ is also open.
\end{proposition}

We write $\PreOrd_{\open}$ for the wide subcategory of $\PreOrd$ whose morphisms are open maps. Similarly, $\FinJoinParOrd_{\open}$ denotes the wide subcategory of $\FinJoinParOrd$ whose morphisms are open maps.

\begin{remark}
The property of being open for order-preserving maps between pre-ordered sets is not preserved under $\sim$ (see the preceding subsection for the notation). Thus, care is needed when defining, for example, ``open morphisms'' in $\hPreOrd$.
\end{remark}

For each subset $\Omega$ of $P$, put
\[
\Low(\Omega) := \{ x \in P \mid x \leq \omega \text{ for every } \omega \in \Omega \}.
\]

The following proposition will be used later:

\begin{proposition}\label{proposition:LowOmega}
If $f : P \rightarrow Q$ is order-preserving, then
\[
f(\Low(\Omega)) \subset \Low(f(\Omega)).
\]
Moreover, if $f$ is injective and open, and $\Omega \neq \emptyset$, then
\[
f(\Low(\Omega)) = \Low(f(\Omega)).
\]
\end{proposition}

\begin{proof}
The first assertion follows immediately from the definition. We prove the second assertion. Since $f$ is open, together with the first assertion, for each $\omega \in \Omega$, we have
\[
f(\Low(\{ \omega \})) = \Low(\{ f(\omega) \}).
\]
Moreover, since $f$ is injective and $\Omega \neq \emptyset$,
\begin{align*}
f(\Low(\Omega))
    &= f\left(\bigcap_{\omega \in \Omega} \Low(\{ \omega \})\right) \\
    &= \bigcap_{\omega \in \Omega} f(\Low(\{ \omega \})) \\
    &= \bigcap_{\omega \in \Omega} \Low(\{ f(\omega) \}) \\
    &= \Low(f(\Omega)).
\end{align*}
\end{proof}

\subsection{Slice categories and subobjects}

Let $\mathcal{C}$ be a $\mathcal{V}$-small category. For each object $X$ of $\mathcal{C}$, let $\mathcal{C}/X$ denote the slice category of $\mathcal{C}$ over $X$; that is, $\Ob(\mathcal{C}/X) := \{ h \in \Mor(\mathcal{C}) \mid t(h) = X \}$, and, for each $h_1,h_2 \in \Ob(\mathcal{C}/X)$,
\[
(\mathcal{C}/X)(h_1,h_2) := \{ \phi \in \mathcal{C}(s(h_1),s(h_2)) \mid h_2 \circ \phi = h_1 \}.
\]
Note that $\mathcal{C}/X$ is again a $\mathcal{V}$-small category. If $h_2$ is a mono-morphism in $\mathcal{C}$, then $\#(\mathcal{C}/X)(h_1,h_2) \leq 1$. In particular, the slice category of a category all of whose morphisms are mono-morphisms is thin.

Let $\Mono^\mathcal{C}$ be the wide subcategory of $\mathcal{C}$ consisting of all mono-morphisms in $\mathcal{C}$. For each object $X$, the slice category $\Mono^\mathcal{C}/X$ is thin. Regarding this thin category as a pre-ordered set in the sense of Section \ref{subsection:pre-order_category}, we write $\Sub(X) := \AS(\Mono^\mathcal{C}/X)$ for its anti-symmetrization, viewed as a partially ordered set. The elements of $\Sub(X)$ correspond to the isomorphism classes of objects in the slice category $\Mono^\mathcal{C}/X$. Such an element is called a subobject of $X$ in $\mathcal{C}$ (cf.~\cite[Definition 4.7.8]{Riehl2016category}).

\begin{example}\label{example:SetSub}
Let $\Set$ denote the category of $\mathcal{U}$-sets and maps. For each object $X$ of $\Set$, $\Sub(X)$ is naturally identified with $\mathcal{P}(X)$. This identification is given by
\[
\Sub(X) \rightarrow \mathcal{P}(X), ~ [h : Z \rightarrow X] \mapsto h(Z).
\]
\end{example}

\subsection{Orthogonal factorization systems}

In this subsection, we recall the notion of an orthogonal factorization system that will be used throughout the paper.
For basic terminology and results concerning orthogonal factorization systems,
we refer the reader to \cite{EchiLazaarAbdallahi2015OFS} and \cite[Section~2]{FreydKelly1972}.

Throughout this subsection, we fix a category $\mathcal{C}$.

\begin{definition}[Lifting property]
Let $e:A\to B$ and $m:X\to Y$ be morphisms in $\mathcal{C}$.
We say that $(e,m)$ has the \emph{lifting property} if, for every commutative diagram
\[
\xymatrix{
A \ar[r]^{u} \ar[d]_{e} & X \ar[d]^{m} \\
B \ar[r]_{v} & Y,
}
\]
there exists a unique morphism $d:B\to X$ such that the following diagram commutes:
\[
\xymatrix{
A \ar[r]^{u} \ar[d]_{e} & X \ar[d]^{m} \\
B \ar[r]_{v} \ar@{-->}[ur]^{d} & Y.
}
\]
In this case, we write $e\perp m$.
\end{definition}

\begin{definition}\label{definition:perpmor}
Let $\mathcal{S}\subset\Mor(\mathcal{C})$ be a class of morphisms.
We define classes of morphisms $\mathcal{S}^{\perp}$ and ${}^{\perp}\mathcal{S}$ by
\[
\begin{aligned}
\mathcal{S}^{\perp}
&=
\{
m\in\Mor(\mathcal{C})
\mid
e\perp m
\text{ for every }
e\in\mathcal{S}
\},
\\
{}^{\perp}\mathcal{S}
&=
\{
e\in\Mor(\mathcal{C})
\mid
e\perp m
\text{ for every }
m\in\mathcal{S}
\}.
\end{aligned}
\]
\end{definition}

The following proposition follows immediately from the definition.

\begin{proposition}\label{proposition:perpsubcategory}
In the setting of Definition \ref{definition:perpmor},
the classes $\mathcal{S}^{\perp}$ and ${}^{\perp}\mathcal{S}$ form wide
subcategories of $\mathcal{C}$.
\end{proposition}

We now recall the definition of an orthogonal factorization system.

\begin{definition}[Orthogonal factorization system]\label{definition:OFS}
Let $\mathcal{E},\mathcal{M}\subset\Mor(\mathcal{C})$ be classes of
morphisms in $\mathcal{C}$.
The pair $(\mathcal{E},\mathcal{M})$ is called an
\emph{orthogonal factorization system} on $\mathcal{C}$ if the following
conditions are satisfied:
\begin{enumerate}
\item
Every morphism $f$ in $\mathcal{C}$ admits a factorization
\[
f=m\circ e
\]
with $e\in\mathcal{E}$ and $m\in\mathcal{M}$.
Such a pair $(e,m)$ is called an
$(\mathcal{E},\mathcal{M})$-factorization of $f$.

\item
\[
\mathcal{E}^{\perp}=\mathcal{M},
\qquad
{}^{\perp}\mathcal{M}=\mathcal{E}.
\]
\end{enumerate}

An orthogonal factorization system $(\mathcal{E},\mathcal{M})$ is said to
be \emph{proper} if every morphism in $\mathcal{E}$ is an epi-morphism and
every morphism in $\mathcal{M}$ is a mono-morphism.
\end{definition}

By Proposition \ref{proposition:perpsubcategory}, if $(\mathcal{E},\mathcal{M})$ is an orthogonal factorization system on $\mathcal{C}$, then both $\mathcal{E}$ and $\mathcal{M}$ form wide subcategories of $\mathcal{C}$.

Moreover, the $(\mathcal{E},\mathcal{M})$-factorization of each morphism $f$ is unique in the sense of the following proposition, which follows immediately from the lifting property:

\begin{proposition}\label{proposition:uniqueness_of_factrization}
For a morphism $h$ in $\mathcal{C}$, let $(e,m)$ and $(e',m')$ be $(\mathcal{E},\mathcal{M})$-factorizations of $h$.
Then there uniquely exists an isomorphism $d$ with $d \circ e = e'$ and $m' \circ d = m$.
\end{proposition}

\begin{example}\label{Example:SetEpiMonoOFS}
For the category $\Set$ defined in Example \ref{example:SetSub}, 
let $\mathcal{E}$ [resp.~$\mathcal{M}$] be the class of all surjections (epi-morphisms) [resp.~injections (mono-morphisms)] in $\Set$.
Then $(\mathcal{E},\mathcal{M})$ gives a proper orthogonal factorization system on $\Set$.
\end{example}

\subsection{Slice categories and orthogonal factorization systems}\label{subsection:Mslice}

In this subsection, we collect terminology concerning slice categories and
several properties of slice categories associated with orthogonal
factorization systems that will be needed later.
Although the results presented in this subsection appear to be well known,
we were unable to find a convenient reference containing precisely the
statements required here. For this reason, proofs will be given in the
appendix (see Section \ref{section:app:proofMslice}).

Throughout this subsection, let $\mathcal{C}$ be a $\mathcal{V}$-small
category.
For each morphism $f:X\to Y$, define a functor
$f_\sharp:\mathcal{C}/X\to\mathcal{C}/Y$ by
\begin{align*}
\Ob(\mathcal{C}/X)
&\rightarrow
\Ob(\mathcal{C}/Y),
&
h
&\mapsto
f\circ h,
\\
\Mor(\mathcal{C}/X)
&\rightarrow
\Mor(\mathcal{C}/Y),
&
(\phi:h_1\rightarrow_{\mathcal{C}/X}h_2)
&\mapsto
(\phi:(f\circ h_1)\rightarrow_{\mathcal{C}/Y}(f\circ h_2)).
\end{align*}

The correspondences
$X\mapsto\mathcal{C}/X$ and 
$f\mapsto f_\sharp$
define a functor from $\mathcal{C}$ to $\Cat$.
In particular, the correspondences $X\mapsto\mathcal{C}/X$ and 
$f\mapsto[f_\sharp]$
define a functor from $\mathcal{C}$ to $\hCat$, which will be denoted by
$\CSlice$.

The following proposition summarizes basic cocompleteness properties of
slice categories.

\begin{proposition}\label{proposition:CXFCC}
Assume that $\mathcal{C}$ is finitely cocomplete.
Then, for each $X\in\Ob(\mathcal{C})$, the slice category
$\mathcal{C}/X$ is also finitely cocomplete.
Moreover, for each morphism
$f:X_1\rightarrow_{\mathcal{C}}X_2$,
the functor
$f_\sharp:\mathcal{C}/X_1\rightarrow\mathcal{C}/X_2$
preserves finite colimits.
In particular, the functor $\CSlice$ may be regarded as a functor from
$\mathcal{C}$ to $\FCChCat$.
\end{proposition}

We next explain how an orthogonal factorization system on $\mathcal{C}$
induces a corresponding slice-category construction.

Throughout the remainder of this subsection, let
$(\mathcal{E},\mathcal{M})$
be an orthogonal factorization system on $\mathcal{C}$ (see Definition \ref{definition:OFS} for the definition of orthogonal factorization systems), and regard
$\mathcal{M}$ as a wide subcategory of $\mathcal{C}$.
For each object $X$ of $\mathcal{C}$ (and hence also of $\mathcal{M}$),
the notation $\mathcal{M}/X$ will always denote the slice category of
$\mathcal{M}$ over $X$.
Note that $\mathcal{M}/X$ is naturally a subcategory of $\mathcal{C}/X$.

\begin{proposition}\label{proposition:MXCXreflective}
The category $\mathcal{M}/X$ is a reflective subcategory of
$\mathcal{C}/X$.
Moreover, let
$\pi_X:\mathcal{C}/X\rightarrow_{\Cat}\mathcal{M}/X$
be a left adjoint to the inclusion functor, 
and we write 
$(e_h,m_h)$ for the $(\mathcal{E},\mathcal{M})$-factorization of each object $h$ of $\mathcal{C}/X$.
Then $\pi_X(h)$ and $m_h$ are isomorphic as objects of $\mathcal{M}/X$.
\end{proposition}

Using Proposition \ref{proposition:MXCXreflective}, we may associate to
every morphism in $\mathcal{C}$ a functor between the corresponding slice
categories of $\mathcal{M}$.

Let $X$ and $Y$ be objects of $\mathcal{C}$, and let
$\pi_X$ [resp.~$\pi_Y$] be a left adjoint to the inclusion functor $\iota_X : \mathcal{M}/X \rightarrow \mathcal{C}/X$ [resp.~$\iota_Y : \mathcal{M}/Y \rightarrow \mathcal{C}/Y$].
For each morphism
$f:X\rightarrow_{\mathcal{C}}Y$,
put
\[
f_\sharp^{\mathcal{M}}
:=
\pi_Y \circ f_\sharp \circ \iota_X.
\]
Note that $[\pi_Y]$ in $\hCat$ does not depend on the choice of $\pi_Y$, by the uniqueness of left adjoint (see \cite[Proposition 4.3.1]{Riehl2016category}), and neither does $[f_\sharp^{\mathcal{M}}]$.

\begin{proposition}\label{proposition:functfM}
The correspondences
$X\mapsto\mathcal{M}/X$
and
$f\mapsto[f_\sharp^{\mathcal{M}}]$
define a functor
\[
\MSlice:\mathcal{C}\to\hCat.
\]    
\end{proposition}

Combining the previous results, we obtain the following cocompleteness 
properties of the functor $\MSlice$.

\begin{theorem}\label{theorem:fMFCC}
Assume that $\mathcal{C}$ is finitely cocomplete.
\begin{enumerate}
\item \label{item:fMFCC:fMsharpFCCpreserve}
For each object $X$, the category $\mathcal{M}/X$ is finitely cocomplete.
Moreover, for each morphism
$f:X\rightarrow_{\mathcal{C}}Y$,
the functor
$[f_\sharp^{\mathcal{M}}]:
\mathcal{M}/X \rightarrow_{\hCat}\mathcal{M}/Y$
preserves finite colimits.
In particular, $\MSlice$ may be regarded as a functor from
$\mathcal{C}$ to $\FCChCat$.

\item \label{item:fMFCC:Mono}
Assume further that every morphism in $\mathcal{M}$ is a mono-morphism in $\mathcal{C}$. 
Then, for each object $X$, the category $\mathcal{M}/X$ is a thin category.
Consequently, $\MSlice$ defines a functor
\[
\mathcal{C}
\longrightarrow
\FCChThinCat
\simeq
\hFinJoinPreOrd
\simeq
\FinJoinParOrd.
\]

Henceforth, we regard $\MSlice$ as a functor from $\mathcal{C}$ to
$\FinJoinParOrd$.
Fix a morphism
$f:X\rightarrow_{\mathcal{C}}Y$,
and consider the finite-join-preserving map
\[
\AS([f_\sharp^{\mathcal{M}}]):
\AS(\mathcal{M}/X)
\rightarrow
\AS(\mathcal{M}/Y).
\]
For each
$h:Z\rightarrow_{\mathcal{M}}X$,
let
$(e_{f\circ h},m_{f\circ h})$
be an $(\mathcal{E},\mathcal{M})$-factorization of
$f\circ h$.
Then
\[
\AS([f_\sharp^{\mathcal{M}}])([h])
=
[m_{f\circ h}].
\]
\end{enumerate}
\end{theorem}

\begin{example}
Consider the setting of Example \ref{Example:SetEpiMonoOFS}. Note that $\mathcal{M} = \Mono^{\Set}$. For each object $X$ of $\Set$, one has $\AS(\mathcal{M}/X) = \Sub(X) \simeq \mathcal{P}(X)$. For each map $f : X \rightarrow Y$, the map induced by $f^\mathcal{M}_\sharp$ corresponds, under this identification, to the direct image map
\[
f^\mathcal{P} : \mathcal{P}(X) \rightarrow \mathcal{P}(Y), ~ S \mapsto f(S).
\]
Consequently, the functor $\MSlice : \Set \rightarrow \FinJoinParOrd$ is naturally isomorphic to the functor given by $X \mapsto \mathcal{P}(X)$ and $f \mapsto f^{\mathcal{P}}$. Moreover, the latter functor may be regarded as a functor from $\Set$ to $\CompJoinParOrd_{\open}$.
\end{example}

\section{Preliminaries for coarse spaces}\label{section:prelim_coarsespaces}

The notion of a coarse space was introduced by J.~Roe
\cite{Roe1993,Roe2003LectureCoarse}.
In this section, we formulate coarse spaces in terms of the category of
relations.
To this end, we first collect some terminology concerning relations and then
review the category of coarse spaces, one of the fundamental categories in
coarse geometry.

Throughout this section, the term ``set'' will always mean a
$\mathcal{U}$-set (i.e.~an element of $\mathcal{U}$; see Setting
\ref{setting:UVC}).

\subsection{Category of relations between sets}\label{section:categoryofrel}

In this subsection, we collect some terminology concerning the category of sets and relations.

For sets $X$ and $Y$, a subset $R$ of $Y\times X$ is called a relation from $X$ to $Y$, and we write
$R:X\rightarrow_{\Rel}Y$.
The set of all relations from $X$ to $Y$ is denoted by
$\Rel(X,Y):=\mathcal{P}(Y\times X)$.
When $X=Y$, we simply write
$\Rel(X):=\Rel(X,X)=\mathcal{P}(X\times X)$.

We define the composition, identity relation, dagger, and inclusion of relations as follows.

\begin{description}
\item[Composition]
For relations
$R_1:X\rightarrow_{\Rel}Y$ and
$R_2:Y\rightarrow_{\Rel}Z$,
their composite
$R_2\circ R_1:X\rightarrow_{\Rel}Z$
is defined by
\[
R_2\circ R_1
:=
\{
(z,x)\in Z\times X
\mid
(z,y)\in R_2
\text{ and }
(y,x)\in R_1
\text{ for some }
y\in Y
\}.
\]

\item[Identity relation]
For a set $X$, we define
\[
1_X^{\Rel}
:=
\{
(x,x)\in X\times X
\mid
x\in X
\}.
\]

\item[Dagger]
For a relation
$R:X\rightarrow_{\Rel}Y$,
its \emph{dagger}
$R^\dagger:Y\rightarrow_{\Rel}X$
is defined by
\[
R^\dagger
:=
\{
(x,y)\in X\times Y
\mid
(y,x)\in R
\}.
\]

\item[Inclusion]
For $R_1,R_2\in\Rel(X,Y)$,
we write
$R_1\subset R_2$
if $(y,x)\in R_1$ implies $(y,x)\in R_2$.
\end{description}

We denote the resulting category of sets and relations by $\Set_{\Rel}$.

\begin{proposition}
Let
$R,R_1,R_2:X\rightarrow_{\Rel}Y$,
$S,S_1,S_2:Y\rightarrow_{\Rel}Z$,
and let
$\{R_\lambda\}_{\lambda\in\Lambda}\subset\Rel(X,Y)$ and
$\{S_\lambda\}_{\lambda\in\Lambda}\subset\Rel(Y,Z)$
be families of relations.
Then the following assertions hold:
\begin{enumerate}
\item
If $R_1\subset R_2$, then
$S\circ R_1\subset S\circ R_2$.

\item
If $S_1\subset S_2$, then
$S_1\circ R\subset S_2\circ R$.

\item
\[
\Bigl(
\bigcup_{\lambda\in\Lambda}
S_\lambda
\Bigr)
\circ R
=
\bigcup_{\lambda\in\Lambda}
(S_\lambda\circ R).
\]

\item
\[
S\circ
\Bigl(
\bigcup_{\lambda\in\Lambda}
R_\lambda
\Bigr)
=
\bigcup_{\lambda\in\Lambda}
(S\circ R_\lambda).
\]

\item
\[
(R^\dagger)^\dagger=R.
\]

\item
\[
(S\circ R)^\dagger
=
R^\dagger\circ S^\dagger.
\]

\item
\[
\Bigl(
\bigcup_{\lambda\in\Lambda}
R_\lambda
\Bigr)^\dagger
=
\bigcup_{\lambda\in\Lambda}
R_\lambda^\dagger.
\]
\end{enumerate}
\end{proposition}

For a relation $R:X\rightarrow_{\Rel}Y$ and a subset $S\subset X$, we define
the image of $S$ under $R$ by
\[
R(S)
:=
\{
y\in Y
\mid
(y,s)\in R
\text{ for some }
s\in S
\}.
\]
This defines an arbitrary-union-preserving map
\[
R^{\mathcal{P}}
:
\mathcal{P}(X)
\rightarrow
\mathcal{P}(Y),
\quad
S\mapsto R(S).
\]

\begin{proposition}\label{proposition:eta^P_functorial}
The correspondence
\[
X\mapsto (\mathcal{P}(X), \subset),
\qquad
R\mapsto R^{\mathcal{P}}
\]
defines a fully faithful functor
\[
\Set_{\Rel}
\rightarrow
\CompJoinParOrd.
\]
Moreover, for every family
$\{R_\lambda\}_{\lambda\in\Lambda}\subset\Rel(X,Y)$,
one has
\[
\Bigl(
\bigcup_{\lambda\in\Lambda}
R_\lambda
\Bigr)(S)
=
\bigcup_{\lambda\in\Lambda}
(R_\lambda(S)).
\]
\end{proposition}

The proposition below will be applied in a later section:

\begin{proposition}\label{proposition:U1U2R}
Let $U_1,U_2 \subset X$ and $R \in \Rel(X)$.
Then the following holds:
\[
U_1 \cap R(U_2) \subset R(R^\dagger (U_1) \cap U_2) 
\]    
\end{proposition}

\begin{proof}[Proof of Proposition \ref{proposition:U1U2R}]
Take any $x \in U_1 \cap R(U_2)$. Our goal is to show that $x \in R(R^\dagger(U_1) \cap U_2)$. Choose $u_2 \in U_2$ such that $(x,u_2) \in R$. Then $(u_2,x) \in R^\dagger$ and $x \in U_1$, so $u_2 \in R^\dagger(U_1)$. Hence $x \in R(R^\dagger(U_1) \cap U_2)$.
\end{proof}

The usual category of sets and maps will be denoted by $\Set_{\map}$.
It is a wide subcategory of $\Set_{\Rel}$ via the identification
\[
f
=
\{
(f(x),x)\in Y\times X
\mid
x\in X
\}
\in\Rel(X,Y)
\]
for each map $f:X\to Y$.
We note that $\id_X=1_X^{\Rel}$ for each set $X$.

For a map $f:X\to Y$, regarded as relations
\[
f\in\Set_{\Rel}(X,Y),
\qquad
f^\dagger\in\Set_{\Rel}(Y,X),
\]
the maps
\[
f^{\mathcal{P}}
:
\mathcal{P}(X)\to\mathcal{P}(Y),
\quad
S\mapsto f(S),
\]
and
\[
(f^\dagger)^{\mathcal{P}}
:
\mathcal{P}(Y)\to\mathcal{P}(X),
\quad
T\mapsto f^\dagger(T),
\]
are defined.
These coincide with the direct image map and the inverse image map along
$f$, respectively.

We also introduce the following notation.

\begin{definition}
For a relation
$R:X\rightarrow_{\Rel}Y$,
we define
\[
\Ad_R
:
\Rel(X)
\rightarrow
\Rel(Y),
\quad
E
\mapsto
R\circ E\circ R^\dagger.
\]
Then the correspondence
\[
X\mapsto\Rel(X),
\qquad
R\mapsto\Ad_R
\]
defines a functor
\[
\Set_{\Rel}
\rightarrow
\CompJoinParOrd.
\]
Moreover, for $E,E_1,E_2\in\Rel(X)$,
\begin{enumerate}
\item
$R\subset R'$ implies
$\Ad_R(E)\subset\Ad_{R'}(E)$.

\item
$(\Ad_R(E))^\dagger=\Ad_R(E^\dagger)$.

\item
\[
\Ad_R(E_1)\circ\Ad_R(E_2)
=
\Ad_R\bigl(E_1\circ(R^\dagger\circ R)\circ E_2\bigr).
\]
\end{enumerate}
\end{definition}

For a map $f:X\to Y$, we note that
\begin{align*}
\Ad_f(E)
&=
\{
(f(x_1),f(x_2))
\in
Y\times Y
\mid
(x_1,x_2)\in E
\},
\\
\Ad_{f^\dagger}(F)
&=
\{
(x_1,x_2)
\in
X\times X
\mid
(f(x_1),f(x_2))
\in F
\}.
\end{align*}

\subsection{Coarse spaces, controlled maps and the category of coarse spaces}

In this subsection, we recall the notions of coarse spaces, controlled maps,
and the category of coarse spaces.

We begin with the definition of a coarse structure.

\begin{definition}[cf.~Roe \cite{Roe2003LectureCoarse}]
Let $X$ be a set.
A subset
$\mathcal{E}\subset\Rel(X)=\mathcal{P}(X\times X)$
is called a coarse structure on $X$ if the following conditions hold:
\begin{enumerate}
\item
$1_X^{\Rel}\in\mathcal{E}$.

\item
$\mathcal{E}$ is closed under composition.

\item
$\mathcal{E}$ is closed under the dagger.

\item
$\mathcal{E}$ is closed under finite unions.

\item
$\mathcal{E}$ is lower in $\Rel(X)$, i.e.,
$E_0\subset E\in\mathcal{E}$ implies $E_0\in\mathcal{E}$.
\end{enumerate}

A pair $(X,\mathcal{E})$ consisting of a set $X$ and a coarse structure
$\mathcal{E}$ on $X$ is called a coarse space.
\end{definition}

Metric spaces provide a fundamental source of coarse spaces.

\begin{example}\label{example:metric_coarse}
Let $(X,d)$ be a metric space.
For each $r\ge0$, we put
\[
E_r^d
:=
\{
(x_1,x_2)\in X\times X
\mid
d(x_1,x_2)\le r
\}.
\]

Then
\[
\mathcal{E}_X^d
:=
\{
E\in\Rel(X)
\mid
\text{there exists }r\ge0
\text{ such that }
E\subset E_r^d
\}
\]
forms a coarse structure on $X$.
\end{example}

Another important class of examples arises from group actions.

\begin{example}\label{example:groupaction_coarse}
Let $G$ be a locally compact Hausdorff group, and let 
$\rho:G\times X\to X$ 
be a continuous action on a locally compact Hausdorff space $X$.

For each compact subset $C\subset G$, we put
\[
E_C^\rho
:=
\{
(x_1,x_2)\in X\times X
\mid
\text{there exists }c\in C
\text{ such that }
\rho(c,x_2)=x_1
\}.
\]

Then
\[
\mathcal{E}_X^\rho
:=
\{
E\in\Rel(X)
\mid
\text{there exists a compact subset }
C\subset G
\text{ such that }
E\subset E_C^\rho
\}
\]
forms a coarse structure on $X$.

We shall discuss several properties of the coarse structure $\mathcal{E}^{\mathrm{LR}}_G$ on $G$, defined as follows, in later sections. Consider the $(G \times G)$-action $\tau$ on $G$ given by left and right multiplication. Then $\mathcal{E}^{\mathrm{LR}}_G := \mathcal{E}^{\tau}_G$ is a coarse structure on $G$.
\end{example}

The following proposition describes the generation of coarse structures.

\begin{proposition}
Let $X$ be a set.
\begin{enumerate}
\item
For an arbitrary family
$\{\mathcal{E}_{\lambda}\}_{\lambda\in\Lambda}$
of coarse structures on $X$, the intersection
$\bigcap_{\lambda\in\Lambda}\mathcal{E}_\lambda$
in $\Rel(X)$ is again a coarse structure on $X$.

\item
For each subset $S\subset\Rel(X)$, there exists a unique smallest coarse
structure $\langle S\rangle$ on $X$ containing $S$.
\end{enumerate}

Such a coarse structure $\langle S\rangle$ is called the coarse structure
on $X$ generated by $S$.
\end{proposition}

We next introduce controlled maps between coarse spaces.

\begin{definition}[cf.~\cite{BunkeEngel2020}]\label{definition:ctrmap}
Let
$X=(X,\mathcal{E}_X)$
and
$Y=(Y,\mathcal{E}_Y)$
be coarse spaces.
A map
$f:X\to Y$
is said to be \emph{controlled} (or \emph{bornologous}) if
$\Ad_f(\mathcal{E}_X)
\subset
\mathcal{E}_Y$.
\end{definition}

Controlled maps form the morphisms of a category.

\begin{theorem}
Coarse spaces and controlled maps form a category with respect to the usual
composition and identity maps.
In this paper, this category will be denoted by
$\Coarse_{\CtrMap}$.
\end{theorem}

For later use, we introduce the notation
\[
\CtrMap(X,Y)
:=
\{
f:X\to Y
\mid
f \text{ is controlled}
\}.
\]

We now introduce the notion of closeness between controlled maps.

\begin{definition}\label{definition:close_ctrmap}
Let
$f_1,f_2\in\CtrMap(X,Y)$.
We say that $f_1$ and $f_2$ are \emph{close}, and write
$f_1\sim f_2$, if
\[
\{
(f_1(x),f_2(x))
\in Y\times Y
\mid
x\in X
\}
\in
\mathcal{E}_Y.
\]
\end{definition}

The following theorem shows that closeness is compatible with the
categorical structure of controlled maps.

\begin{theorem}[cf.~\cite{DikranjanZavaCatCoarseGp2020}]
The relation ``close'' defines a congruence on the category
$\Coarse_{\CtrMap}$.
\end{theorem}

Hence one may form the corresponding quotient category.

\begin{definition}[cf.~\cite{LeitnerVigolo2023}]
The quotient category of
$\Coarse_{\CtrMap}$
by the congruence ``close'' is called the
\emph{category of coarse spaces} and is denoted by
\[
\Coarse_{\CtrMap}/\close.
\]
\end{definition}

For a controlled map
$f:X\to Y$,
its equivalence class in
$\Coarse_{\CtrMap}/\close$
will be denoted by
$[f]_{\CtrMap}$,
or simply by $[f]$.

Using the category of coarse spaces, we can now formulate the notion of
coarse equivalence.

\begin{definition}\label{definition:coarse_eq}
Let $X$ and $Y$ be coarse spaces.
A controlled map
$f:X\to Y$
is called a \emph{coarse equivalence} if
$[f]_{\CtrMap}$
is an isomorphism in the category of coarse spaces
$\Coarse_{\CtrMap}/\close$.
\end{definition}

% \begin{remark}
% It is worth comparing Definition~\ref{definition:coarse_eq} with the
% classical definition of coarse equivalence due to Roe.
% In Roe's book \cite{Roe2003LectureCoarse}, the primary objects of study are
% not controlled maps (called \emph{bornologous maps} there) but the stronger
% notion of \emph{coarse maps}, namely controlled maps satisfying an
% additional properness condition.
% Roe considers the quotient category of coarse maps modulo closeness, and an
% isomorphism in that category is called a coarse equivalence.

% On the other hand, every coarse equivalence in the sense of
% Definition~\ref{definition:coarse_eq} is automatically proper
% (see, for example, \cite[Theorem 5.2]{DikranjanZava2017} or
% \cite[Definition 1.2(ix)]{DikranjanZavaCatCoarseGp2020}).
% Consequently, Roe's notion of coarse equivalence coincides with the one
% adopted in this paper.
% \end{remark}

\begin{example}
A quasi-Lipschitz map between metric spaces is a controlled map with respect to the coarse structures introduced in Example \ref{example:metric_coarse}. Moreover, a quasi-isometry between metric spaces is a coarse equivalence. In particular, if the metric spaces are quasi-geodesic, then being coarsely equivalent is equivalent to being quasi-isometric (cf.~Roe \cite[Lemma 1.10]{Roe2003LectureCoarse}).
\end{example}

\begin{example}\label{example:Cartan}
Let $G$ be a linear reductive Lie group, and consider the coarse structure $\mathcal{E}^{\mathrm{LR}}_G$ on $G$ defined in 
Example \ref{example:groupaction_coarse}.
Fix a maximal compact subgroup $K$ of $G$. Then $M := G/K$ admits a $G$-invariant Riemannian metric $g$, and $(M,g)$ is a Riemannian symmetric space of non-compact type (cf.~\cite{HelgasonDiff}). Let $d$ denote the distance function on $M$ induced by $g$. We define $\Omega := K \backslash M = K \backslash G/K$, equipped with the distance $d_\Omega$ between $K$-orbits induced by $d$.

Consider the canonical map
\[
\mu : G \rightarrow \Omega = K \backslash G/K, \qquad g \mapsto KgK.
\]
We note that this map is essentially the same as the Cartan projection. The map $\mu$ defines a coarse equivalence from $(G,\mathcal{E}^{\mathrm{LR}}_G)$ to $(\Omega,\mathcal{E}^{d_\Omega})$ (see \cite[Example 5.5]{NagayaOgawaOkuda2025};  the detailed proof will be reported elsewhere).
\end{example}

\subsection{Some properties of the category of coarse spaces}

We conclude this section by recalling several categorical properties of the
category of coarse spaces $\Coarse_{\CtrMap}/\close$ that will be used later.

The existence of limits and colimits in the category of coarse spaces has
been studied by several authors. In particular, the following results are
known.

\begin{theorem}[Leitner--Vigolo \cite{LeitnerVigolo2023}]\label{theorem:Coarsecocomp}
The category of coarse spaces is finitely cocomplete.
\end{theorem}

\begin{theorem}[Zava \cite{Zava2019CoarseCat}]
The category of coarse spaces admits arbitrary products.
On the other hand, it does not necessarily admit equalizers.
\end{theorem}

We next recall characterizations of epi-morphisms and mono-morphisms in the
category of coarse spaces.

\begin{theorem}[Dikranjan--Zava {\cite{DikranjanZava2017}}]\label{theorem:CtrMap_epimono}
Let $f:X\rightarrow Y$ be a controlled map.
\begin{enumerate}
\item
The morphism $[f]_{\CtrMap}$ is an epi-morphism in the category of coarse
spaces if and only if there exists $E\in\mathcal{E}_Y$ such that
$E(f(X))=Y$.

\item
The morphism $[f]_{\CtrMap}$ is a mono-morphism in the category of coarse
spaces if and only if
$\Ad_{f^\dagger}(\mathcal{E}_Y)\subset\mathcal{E}_X$.
\end{enumerate}
\end{theorem}

Combining this characterization with the following theorem yields a useful
criterion for recognizing coarse equivalences.

\begin{theorem}[Dikranjan--Zava {\cite{DikranjanZava2017}}]\label{theorem:coarse_balanced}
The category of coarse spaces is balanced; that is, a morphism is an
isomorphism if and only if it is both an epi-morphism and a mono-morphism.
\end{theorem}

Finally, we recall a recent result of Tang showing that the category of
coarse spaces admits a natural orthogonal factorization system.

\begin{theorem}[Tang {\cite[Theorem 1.8]{Tang2024categorical}}]\label{theorem:EM_coarse}
Let $\mathcal{E}$ [resp.~$\mathcal{M}$] be the class of epi-morphisms
[resp.~mono-morphisms] in the category of coarse spaces.
Then $(\mathcal{E},\mathcal{M})$ forms an orthogonal factorization system on
the category of coarse spaces.
\end{theorem}

\subsection{Coarse subspaces}\label{subsection:coarsesubset}

We begin by introducing the notion of a \emph{coarse subspace} of a coarse space.

% \begin{definition}[cf.~Leitner--Vigolo \cite{LeitnerVigolo2023}]\label{definition:similar}
% Let $X=(X,\mathcal{E}_X)$ be a coarse space. Define a pre-order on the power set $\mathcal{P}(X)$ by
% \[
% A \prec B \defarrow \text{there exists } E\in\mathcal{E}_X \text{ such that } A\subset E(B).
% \]
% In this paper, we write $(\mathcal{P}_c(X),\prec)$ for the anti-symmetrization of the pre-order $(\mathcal{P}(X),\prec)$ (see Section \ref{subsection:pre-order_category} for the definition of anti-symmetrization). The elements of $\mathcal{P}_c(X)$ are called \emph{coarse subspaces} of $X$.
% \end{definition}

\begin{definition}[cf.~Leitner--Vigolo \cite{LeitnerVigolo2023}]\label{definition:similar}
    Let $X=(X,\mathcal{E}_X)$ be a coarse space. Define a pre-order on the power set $\mathcal{P}(X)$ by
    \[
    A \prec B
    \defarrow
    \text{there exists } E\in\mathcal{E}_X \text{ such that } A\subset E(B).
    \]
    Recall that, for a pre-order, the relation $\sim$ was defined in Definition~\ref{definition:simonpre} by
    \[
    A\sim B
    \defarrow
    A\prec B \text{ and } B\prec A.
    \]
    Further, the anti-symmetrization of a pre-order is defined by using this relation in Section \ref{subsection:pre-order_category}.
    In this paper, we define $(\mathcal{P}_c(X),\prec)$ to be the anti-symmetrization of the pre-order $(\mathcal{P}(X),\prec)$, and the elements of $\mathcal{P}_c(X)$ are called \emph{coarse subspaces} of $X$.
\end{definition}

\begin{remark}
In the literature, a coarse space obtained by equipping a subset of $X$ with the induced coarse structure is sometimes also called a coarse subspace (see for example, \cite[Definition 2.6]{ChenWang2004Ideal}). The compatibility of this terminology with the above definition will be justified by the discussion in Section \ref{section:subspace_subobject}.
\end{remark}

\begin{remark}
Kalantari--Honari \cite{KalantariHonari2016} call the equivalence relation on $\mathcal{P}(X)$ induced by the pre-order $\prec$ a \emph{resemblance}. When $S_1\sim S_2$, they say that $S_1$ and $S_2$ are \emph{alike}.
\end{remark}

\begin{example}
Let $X$ be a metric space equipped with the coarse structure of Example \ref{example:metric_coarse}. For subsets $S_1,S_2\in\mathcal{P}(X)$, the relation $S_1\sim S_2$ holds if and only if the Hausdorff distance between $S_1$ and $S_2$ is finite.
\end{example}

\begin{example}
Let $G$ be a locally compact Hausdorff group, and let $\mathcal{E}^{\mathrm{LR}}_G$ be the coarse structure defined as in Example \ref{example:groupaction_coarse}. Then, for subsets $S_1,S_2\subset G$, the relation $S_1\sim S_2$ in the sense of Definition \ref{definition:similar} holds if and only if $S_1\sim S_2$ in the sense of Kobayashi \cite{Kobayashi96}.
\end{example}

We also record the following set-theoretic observation.

\begin{proposition}
As declared at the beginning of this section, every set $X$ is assumed to satisfy $X\in\mathcal{U}$ (see Setting \ref{setting:UVC}). Under this convention, for any coarse space $X=(X,\mathcal{E})$, both $\mathcal{P}(X)$ and $\mathcal{P}_c(X)$ belong to $\mathcal{U}$. In particular, $\mathcal{P}(X),\mathcal{P}_c(X)\in\mathcal{V}$.
\end{proposition}

\begin{proof}
Since $X\in\mathcal{U}$ and $\mathcal{U}$ is a universe, we have $\mathcal{P}(X),\mathcal{P}(\mathcal{P}(\mathcal{P}(X)))\in\mathcal{U}$. By definition, $\mathcal{P}_c(X)$ is an element of $\mathcal{P}(\mathcal{P}(\mathcal{P}(X)))$, and hence $\mathcal{P}_c(X)\in\mathcal{U}$. The final assertion follows from the facts that $\mathcal{U}\in\mathcal{V}$ and that $\mathcal{V}$ is also a universe.
\end{proof}

\section{A realization of the category of coarse spaces based on controlled total relations}\label{section:realizationCtrTotRel}

In this section, we introduce the notion of a \emph{controlled total relation},
which may be regarded as a generalization of a controlled map.
We then show that the category of coarse spaces admits a realization as a
quotient category of a category of controlled total relations.

The ideas presented in this section are essentially contained in
Mart\'{i}nez--Vigolo \cite{MartinezVigolo2026roe}.
Our treatment differs in that we reformulate those ideas from a more
explicitly categorical viewpoint.

As in the previous section, the term ``set'' will always mean a
$\mathcal{U}$-set (i.e.~an element of $\mathcal{U}$; see Setting
\ref{setting:UVC}).

\subsection{Total relations}

We begin by recalling the notion of a total relation (cf.~\cite{BerghammerFurusawaGuttmannHofner2020Rel}) between sets.

\begin{definition}
Let $X$ and $Y$ be sets.
\begin{enumerate}
    \item For each relation $\eta \in \Rel(X,Y)$, we define $\Dom(\eta) \subset X$ by 
    \[
    \Dom(\eta) := \{ x \in X \mid \text{ there exists } y \in Y \text{ such that }(y,x) \in \eta \}.
    \]
    \item A relation $\eta\in\Rel(X,Y)$ is said to be \emph{total} if $\Dom(\eta) = X$.
\end{enumerate}
\end{definition}

The following properties are immediate from the definition.

\begin{proposition}\label{proposition:total_basic}
\begin{enumerate}
\item
Totality is an upper condition on $\Rel(X,Y)$.
That is, for $\xi,\eta\in\Rel(X,Y)$ with $\xi\subset\eta$,
if $\xi$ is total, then $\eta$ is total.

\item \label{item:total_basic:totalleftrem}
Let $\eta\in\Rel(X,Y)$ and $R\in\Rel(Y,Z)$.
If $R\circ\eta$ is total, then $\eta$ is total.

\item
Total relations form a wide subcategory of $\Set_{\Rel}$, which will be
denoted by $\Set_{\TotRel}$.

\item
$\Set_{\map}$ is a wide subcategory of $\Set_{\TotRel}$.
\end{enumerate}
\end{proposition}

We also record several equivalent characterizations of totality.

\begin{proposition}\label{proposition:total_condition}
For a relation $\eta\in\Rel(X,Y)$, the following conditions are equivalent:
\begin{enumerate}
\item\label{item:total_condition:total}
$\eta$ is total.

\item\label{item:total_condition:non-empty}
$\eta(S)\neq\emptyset$ for each non-empty subset $S$ of $X$.

\item\label{item:total_condition:alg}
\[
1_X^{\Rel}
\subset
\Ad_{\eta^\dagger}(1_Y^{\Rel})
(=\eta^\dagger\circ\eta).
\]
\end{enumerate}
\end{proposition}

\begin{proof}
The equivalence
\eqref{item:total_condition:total}
$\Leftrightarrow$
\eqref{item:total_condition:non-empty}
is immediate.

We prove
\eqref{item:total_condition:total}
$\Leftrightarrow$
\eqref{item:total_condition:alg}.
Observe that
\begin{align*}
\Ad_{\eta^\dagger}(1_Y^\Rel)
&=
\eta^\dagger\circ1_Y^\Rel\circ\eta
\\
&=
\eta^\dagger\circ\eta
\\
&=
\{
(x_1,x_2)\in X\times X
\mid
\text{there exists }y\in Y
\text{ such that }
(x_1,y)\in\eta^\dagger,
\ (y,x_2)\in\eta
\}
\\
&=
\{
(x_1,x_2)\in X\times X
\mid
\text{there exists }y\in Y
\text{ such that }
(y,x_1),(y,x_2)\in\eta
\}.
\end{align*}
Therefore,
$1_X^\Rel\subset\Ad_{\eta^\dagger}(1_Y^\Rel)$
if and only if, for each $x\in X$, there exists $y\in Y$ such that
$(y,x)\in\eta$.
This proves
\eqref{item:total_condition:total}
$\Leftrightarrow$
\eqref{item:total_condition:alg}.
\end{proof}

The following three propositions will be applied later.

\begin{proposition}\label{proposition:SDometatotal}
For each $\eta \in \Rel(X,Y)$ and $S \subset \Dom(\eta)$, 
$S \subset (\eta^\dagger \circ \eta)(S)$.    
\end{proposition}

\begin{proposition}\label{proposition:total_adjoint}
Let $\eta\in\Rel(X,Y)$ and $E \in \Rel(X)$.
Then for each subset $S$ of $\Dom(\eta) \subset X$, 
\[
(\eta \circ E)(S) \subset ((\Ad_\eta E)\circ \eta)(S)
\]
holds.
In particular, if $\eta$ is total, then 
\[
\eta\circ E \subset (\Ad_\eta E)\circ \eta
\]
in $\Rel(X,Y)$ holds.
\end{proposition}

\begin{proposition}\label{proposition:iotaxi_Omega}
Let $X$ and $Z$ be sets, and let $\Omega$ be a subset of $X$.
Let $\iota_\Omega:\Omega\to X$ be the inclusion map, regarded as an element of
$\Rel(\Omega,X)$.
For $\xi\in\Rel(Z,X)$, put
$\xi_\Omega:=\iota_\Omega^\dagger\circ\xi\in\Rel(Z,\Omega)$.
Then 
\[
\xi_\Omega=(\Omega\times Z)\cap\xi
\]
holds.
In particular, for every subset $Z_0\subset Z$, one has
$\xi_{\Omega}(Z_0)=\xi(Z_0)\cap \Omega$.
Moreover, if $\xi(Z)\subset \Omega$, then
\[
\xi=\iota_\Omega \circ\xi_\Omega.
\]
In this case, if $\xi$ is total, then $\xi_\Omega$ is also total.
\end{proposition}

Let us give proofs of Propositions \ref{proposition:SDometatotal}, 
\ref{proposition:total_adjoint} 
and \ref{proposition:iotaxi_Omega}:

\begin{proof}[Proof of Proposition \ref{proposition:SDometatotal}]
Let $\iota : \Dom(\eta) \rightarrow X$ denote the inclusion map.
Then $\eta \circ \iota \in \Rel(\Dom(\eta),Y)$ is total.
Thus by Proposition \ref{proposition:total_condition}, we have \[
1^{\Rel}_{\Dom(\eta)} \subset (\eta \circ \iota)^\dagger \circ (\eta \circ \iota).
\]
Thus 
\begin{align*}
S 
    &= 1^{\Rel}_{\Dom(\eta)}(S) \\
    &\subset ((\eta \circ \iota)^\dagger \circ (\eta \circ \iota))(S) \\
    &= \iota^\dagger((\eta^\dagger \circ \eta)(S)) \\
    &= \iota^{-1}((\eta^\dagger \circ \eta)(S)) \\
    &= (\eta^\dagger \circ \eta)(S) \cap \Dom(\eta) \\
    &\subset (\eta^\dagger \circ \eta)(S).
\end{align*}
\end{proof}

\begin{proof}[Proof of Proposition \ref{proposition:total_adjoint}]
Take $S \subset \Dom(\eta)$.
Then by Proposition \ref{proposition:SDometatotal}, 
we have 
\begin{align*}
(\eta \circ E)(S)
    &\subset (\eta \circ E)((\eta^\dagger \circ \eta)S) \\
    &= (\Ad_{\eta}(E)) \circ \eta)(S).
\end{align*}
\end{proof}

\begin{proof}[Proof of Proposition \ref{proposition:iotaxi_Omega}]
First note that
\[
(x,\omega)\in\iota_\Omega
\iff
x\in \Omega \text{ and } x=\omega.
\]
Hence, for each $(x,z)\in X\times Z$,
\begin{align*}
(x&,z)\in\xi_\Omega \\
&\iff
(x,z)\in\iota_\Omega^\dagger\circ\xi \\
&\iff
x\in \Omega,\text{ and there exists }x'\in X
\text{ such that }(x,x')\in\iota_\Omega^\dagger
\text{ and }(x',z)\in\xi \\
&\iff
x\in \Omega,\text{ and }(x,z)\in\xi.
\end{align*}
This proves $\xi_\Omega=(\Omega\times Z)\cap\xi$.
The identity $\xi_\Omega(Z_0)=\xi(Z_0)\cap\Omega$ follows immediately.

Assume now that $\xi(Z)\subset\Omega$.
Then
\begin{align*}
(x,z)\in\iota_\Omega\circ\xi_\Omega
&\iff
x\in\Omega,\text{ and }(x,z)\in\xi_\Omega \\
&\iff
x\in\Omega,\text{ and }(x,z)\in\xi \\
&\iff
(x,z)\in\xi,
\end{align*}
where the last equivalence follows from $\xi(Z)\subset\Omega$.
Thus $\xi=\iota_\Omega\circ\xi_\Omega$.

Finally, suppose in addition that $\xi$ is total.
Since $\xi=\iota_\Omega\circ\xi_\Omega$ is total, Proposition
\ref{proposition:total_basic} \eqref{item:total_basic:totalleftrem} implies that $\xi_\Omega$ is total.
\end{proof}

\subsection{Controlled relations}

The category of coarse spaces and relations [resp.~total relations] will be
denoted by $\Coarse_{\Rel}$ [resp.~$\Coarse_{\TotRel}$].

\begin{definition}[cf.~{\cite[Definition 3.17]{MartinezVigolo2026roe}}]\label{definition:controlledRel}
Let $(X,\mathcal{E}_X)$ and $(Y,\mathcal{E}_Y)$ be coarse spaces.
In this paper, a relation $\eta\in\Rel(X,Y)$
(i.e.~a subset of $Y\times X$) is said to be \emph{controlled} if
$\Ad_\eta(\mathcal{E}_X)\subset\mathcal{E}_Y$.
\end{definition}

We write $\CtrRel(X,Y)$ [resp.~$\CtrTotRel(X,Y)$] for the set of all
controlled relations [resp.~controlled total relations] between the coarse
spaces $X$ and $Y$.

The following theorem holds.

\begin{theorem}
Coarse spaces and controlled relations form a wide subcategory of
$\Coarse_{\Rel}$, which will be denoted by $\Coarse_{\CtrRel}$.
Moreover, controlled total relations form a wide subcategory of
$\Coarse_{\CtrRel}$, which will be denoted by $\Coarse_{\CtrTotRel}$.
\end{theorem}

\begin{proof}
It suffices to verify the following two assertions in $\Coarse_{\Rel}$:
\begin{enumerate}
\item
controlled relations and total relations are closed under composition;

\item
for each coarse space $X$, the relation $1_X^{\Rel}\in\Rel(X)$ is both
controlled and total.
\end{enumerate}
The closure of controlled relations under composition follows from the
functoriality of $\Ad$, namely
$\Ad_{\eta\circ\xi}=\Ad_\eta\circ\Ad_\xi$.
It is clear that $1_X^{\Rel}$ is a controlled relation from $X$ to $X$,
since $\Ad_{1_X^{\Rel}}$ is the identity map.
The closure of total relations under composition, as well as the totality of
$1_X^{\Rel}$, follows from the fact that $\Set_{\TotRel}$ is a subcategory
of $\Set_{\Rel}$.
\end{proof}

The following proposition follows immediately from the definitions.

\begin{proposition}\label{proposition:controlled_rel_map}
Let $f:X\to Y$ be a map, regarded as an element of $\Rel(X,Y)$.
Then $f$ is a controlled map if and only if $f$ is controlled in the sense of
Definition \ref{definition:controlledRel}.
In particular, $\Coarse_{\CtrMap}$ is a wide subcategory of
$\Coarse_{\CtrTotRel}$.
\end{proposition}

The following propositions will be used later.

\begin{proposition}\label{proposition:ent_cont}
For each coarse space $X=(X,\mathcal{E}_X)$, one has 
$\mathcal{E}_X
\subset
\CtrRel(X)$.
\end{proposition}

\begin{proposition}\label{proposition:invCstr}
Let $X$ be a set and let $Y=(Y,\mathcal{E}_Y)$ be a coarse space.
Fix a total relation $\eta\in\Rel(X,Y)$.
Suppose that
$\Ad_\eta(1_X^\Rel)
\in
\mathcal{E}_Y$.
Then
\[
(\Ad_\eta)^{-1}(\mathcal{E}_Y) := \{ E \in \Rel(X) \mid \Ad_{\eta}(E) \in \mathcal{E}_Y \} 
\subset
\Rel(X)
\]
is a coarse structure on $X$.
Moreover, $(\Ad_\eta)^{-1}(\mathcal{E}_Y)$ is the largest coarse structure
on $X$ for which $\eta$ is controlled.
\end{proposition}

\begin{proposition}\label{proposition:generator_controlled}
Let $X$ be a set and let $Y=(Y,\mathcal{E}_Y)$ be a coarse space.
Fix a subset $\mathcal{E}_0\subset\Rel(X)$ and a total relation
$\eta\in\Rel(X,Y)$.
Suppose that
$\Ad_\eta(1_X^\Rel)
\in
\mathcal{E}_Y$.
Then the following conditions are equivalent:
\begin{enumerate}
\item\label{proposition:generator_controlled_controlled}
The relation $\eta$ is controlled from
$(X,\langle\mathcal{E}_0\rangle)$
to
$(Y,\mathcal{E}_Y)$.
\item\label{proposition:generator_controlled_inclusion}
$\Ad_\eta(\mathcal{E}_0)
\subset
\mathcal{E}_Y$.
\end{enumerate}
\end{proposition}

\begin{proof}[Proof of Proposition \ref{proposition:ent_cont}]
For each $E,F \in \mathcal{E}_{X}$, 
we have 
\[
\Ad_{E}(F) = E \circ F \circ E^\dagger \in \mathcal{E}_X.
\]
This proves $\Ad_{E}(\mathcal{E}_X) \subset \mathcal{E}_X$.
\end{proof}

\begin{proof}[Proof of Proposition \ref{proposition:invCstr}]
Put $\mathcal{E}_X := (\Ad_\eta)^{-1}(\mathcal{E}_Y)\subset\Rel(X)$.
It suffices to verify the following:
\begin{enumerate}
    \item $1^\Rel_X \in \mathcal{E}_X$,
    \item $\mathcal{E}_X$ is lower in $\Rel(X)$,
    \item $\mathcal{E}_X$ is closed under composition, finite unions and the dagger operation.
\end{enumerate}

The first assertion follows immediately from the assumption.

For the second assertion, let $E_0 \subset E$ in $\Rel(X)$ with $E \in \mathcal{E}_X$.
Since $\Ad_\eta(E_0) \subset \Ad_\eta(E)$ and $\Ad_\eta(E) \in \mathcal{E}_Y$, the downward closedness of $\mathcal{E}_Y$ implies $\Ad_\eta(E_0) \in \mathcal{E}_Y$.
Hence $E_0 \in \mathcal{E}_X$.

Since $\Ad_\eta$ preserves finite unions and the dagger operation, and since $\mathcal{E}_Y$ is closed under these operations, it follows that $\mathcal{E}_X$ is also closed under finite unions and the dagger operation.
It remains to prove closure under composition.
Let $E_1,E_2\in\mathcal{E}_X$.
Since $\eta$ is total, Proposition~\ref{proposition:total_condition} yields
$1^\Rel_X\subset\eta^\dagger\circ\eta$.
Therefore
\[
\Ad_\eta(E_1\circ E_2)
=
\eta\circ E_1\circ E_2\circ\eta^\dagger
\subset
\eta\circ E_1\circ\eta^\dagger\circ\eta\circ E_2\circ\eta^\dagger
=
\Ad_\eta(E_1)\circ\Ad_\eta(E_2).
\]
Since $\Ad_\eta(E_1),\Ad_\eta(E_2)\in\mathcal{E}_Y$ and $\mathcal{E}_Y$ is closed under composition and taking subrelations, it follows that $\Ad_\eta(E_1\circ E_2)\in\mathcal{E}_Y$.
Hence $E_1\circ E_2\in\mathcal{E}_X$. 
Therefore $\mathcal{E}_X$ is a coarse structure on $X$.
\end{proof}

\begin{proof}[Proof of Proposition \ref{proposition:generator_controlled}]
It directly follows from the definition of controlled maps that \eqref{proposition:generator_controlled_controlled} implies \eqref{proposition:generator_controlled_inclusion}.
Hence, let us prove the converse claim.
Suppose that $\Ad_{\eta}(\mathcal{E}_0) \subset \mathcal{E}_Y$. It is enough to show that $\langle \mathcal{E}_0 \rangle \subset (\Ad_\eta)^{-1}(\mathcal{E}_Y)$.
Since $\Ad_\eta(1^\Rel_X) \in \mathcal{E}_Y$, 
by Proposition \ref{proposition:invCstr},   
$(\Ad_\eta)^{-1}(\mathcal{E}_Y)$ is a coarse structure on $X$.
By the assumption $\Ad_{\eta}(\mathcal{E}_0) \subset \mathcal{E}_Y$, we have $\mathcal{E}_0 \subset (\Ad_\eta)^{-1}(\mathcal{E}_Y)$.
Thus by the definition of $\langle \mathcal{E}_0 \rangle$, 
we also have $\langle \mathcal{E}_0 \rangle \subset (\Ad_\eta)^{-1}(\mathcal{E}_Y)$.
\end{proof}

\subsection{Coarse pre-orders on relations}

We next introduce a natural pre-order on the set of relations between two
coarse spaces.

\begin{definition}\label{definition:close_ctrrel}
Let $X=(X,\mathcal{E}_X)$ and $Y=(Y,\mathcal{E}_Y)$ be coarse spaces.
For $\eta_1,\eta_2\in\Rel(X,Y)$, we write
$\eta_1\prec_{\mathcal{E}_X,\mathcal{E}_Y}\eta_2$,
or simply
$\eta_1\prec_{X,Y}\eta_2$ or $\eta_1\prec\eta_2$,
if there exist
$E\in\mathcal{E}_X$ and $F\in\mathcal{E}_Y$ such that
$\eta_1\subset F\circ\eta_2\circ E$.
Then $\prec$ defines a pre-order on $\Rel(X,Y)$, called the
\emph{coarse pre-order} on $\Rel(X,Y)$.

Furthermore, two relations
$\eta_1,\eta_2\in\Rel(X,Y)$
are said to be \emph{close}, and we write
$\eta_1\sim_{\mathcal{E}_X,\mathcal{E}_Y}\eta_2$,
or simply
$\eta_1\sim_{X,Y}\eta_2$ or $\eta_1\sim\eta_2$,
if
$\eta_1\prec\eta_2$
and
$\eta_2\prec\eta_1$.
\end{definition}

\begin{proposition}\label{proposition:CtrLowerprec}
$\CtrRel(X,Y)$ is lower in the pre-ordered set
$(\Rel(X,Y),\prec)$.
In particular, $\CtrRel(X,Y)$ is closed under the equivalence relation
$\sim$ on $\Rel(X,Y)$.
\end{proposition}

\begin{proof}
Let $\eta_0 \prec \eta\in\CtrRel(X,Y)$.
Take $E \in \mathcal{E}_X$ and $F \in \mathcal{E}_Y$ with $\eta_0 \subset F \circ \eta \circ E$.
For each $G \in \mathcal{E}_X$, we have
\begin{align*}
 \Ad_{\eta_0}(G)
    &\subset \Ad_{F \circ \eta \circ E}(G) \\
    &= \Ad_{F}(\Ad_{\eta}(\Ad_{E}(G))).
\end{align*}
Since $\eta$, $E$, $F$ are all controlled (see also Proposition \ref{proposition:ent_cont}), 
we have $\Ad_{\eta_0}(G) \subset \Ad_{F}(\Ad_{\eta}(\Ad_{E}(G))) \in \mathcal{E}_Y$.
Since $\mathcal{E}_Y$ is lower in $(\Rel(Y),\subset)$,
it follows that
$\Ad_{\eta_0}(G)\in\mathcal{E}_Y$.
Therefore $\eta_0$ is controlled.
\end{proof}

\begin{remark}
One might expect that the equivalence relation $\sim$ on morphisms defines a
congruence on $\Coarse_{\Rel}$ or on $\Coarse_{\CtrRel}$.
However, this is not the case. 
Indeed, let $d$ be the standard metric on $\R$, and consider the coarse
spaces
\[
X=Y=Z=(\R,\mathcal{E}_{\R}^{d}).
\]
Define relations $\eta_1,\eta_2\in\Rel(X,Y)$ by
\begin{align*}
\eta_1 &:= \{(t,t)\mid t\in\R\},\\
\eta_2 &:= \{(t,t)\mid t\in\Z,\ t \text{ is even}\},
\end{align*}
and define $\xi\in\Rel(Y,Z)$ by
\[
\xi:=\{(s,s)\mid s\in\Z,\ s \text{ is odd}\}.
\]
These relations are controlled, and one has
$\eta_1\sim\eta_2$.
However,
\[
\eta_1\circ\xi
=
\{(k,k)\mid k\in\Z,\ k \text{ is odd}\},
\qquad
\eta_2\circ\xi
=
\emptyset.
\]
In particular,
$\eta_1\circ\xi\not\sim\eta_2\circ\xi$.
Thus $\sim$ fails to satisfy the congruence condition in both
$\Coarse_{\Rel}$ and $\Coarse_{\CtrRel}$.

In the next subsection, we show that this issue disappears after restricting
to the wide subcategory $\Coarse_{\CtrTotRel}$ of $\Coarse_{\CtrRel}$.
More precisely, $\sim$ becomes a congruence on
$\Coarse_{\CtrTotRel}$, and the resulting quotient category is equivalent to
the category of coarse spaces.

More generally, it is natural to ask whether controlled relations that are
not total can also be used to represent morphisms in the category of coarse
spaces. In this direction, Mart\'{i}nez--Vigolo
\cite{MartinezVigolo2026roe} consider morphisms arising from certain
``densely defined'' controlled relations. A more systematic understanding of
such constructions remains an interesting topic for future investigation.
\end{remark}

\subsection{The category of coarse spaces realized by controlled total relations}

In this subsection, we restrict our attention to the category
$\Coarse_{\CtrTotRel}$.

Our first goal is to show that the closeness relation introduced in
Definition \ref{definition:close_ctrrel} becomes a congruence on
$\Coarse_{\CtrTotRel}$.

\begin{theorem}\label{theorem:CtrTot_cong}
The closeness relation ``$\sim$'' defines a congruence on
$\Coarse_{\CtrTotRel}$.
\end{theorem}

To prove this theorem, we first establish the following characterization of
closeness between controlled total relations.

\begin{proposition}\label{proposition:close_on_totalrel}
Let $X=(X,\mathcal{E}_X)$ and $Y=(Y,\mathcal{E}_Y)$ be coarse spaces.
Fix $\eta_1,\eta_2\in\CtrTotRel(X,Y)$.
Then the following conditions are equivalent:
\begin{enumerate}
\item\label{item:close_on_totalrel:equivalent}
$\eta_1\sim\eta_2$.

\item\label{item:close_on_totalrel:eta1-below-eta2}
$\eta_1\prec\eta_2$.

\item\label{item:close_on_totalrel:eta1-contained}
There exists $E_Y\in\mathcal{E}_Y$ such that
$\eta_1\subset E_Y\circ\eta_2$.

\item\label{item:close_on_totalrel:eta2-below-eta1}
$\eta_2\prec\eta_1$.

\item\label{item:close_on_totalrel:eta2-contained}
There exists $E_Y\in\mathcal{E}_Y$ such that
$\eta_2\subset E_Y\circ\eta_1$.

\item\label{item:close_on_totalrel:composite12}
$\eta_1\circ\eta_2^\dagger\in\mathcal{E}_Y$.

\item\label{item:close_on_totalrel:composite21}
$\eta_2\circ\eta_1^\dagger\in\mathcal{E}_Y$.

\item\label{item:close_on_totalrel:witness-relation}
The relation
\[
\{
(y_1,y_2)\in Y\times Y
\mid
\text{there exists }x\in X
\text{ such that }
(y_1,x)\in\eta_1
\text{ and }
(y_2,x)\in\eta_2
\}
\]
belongs to $\mathcal{E}_Y$.
\end{enumerate}
\end{proposition}

\begin{proof}
The equivalence of
\eqref{item:close_on_totalrel:composite12}
and
\eqref{item:close_on_totalrel:composite21}
follows from the identity
$(\eta_1\circ\eta_2^\dagger)^\dagger
=
\eta_2\circ\eta_1^\dagger$, 
together with the fact that $\mathcal{E}_Y$ is closed under the dagger.
Moreover,
\eqref{item:close_on_totalrel:witness-relation}
is simply a reformulation of
\eqref{item:close_on_totalrel:composite12}
and
\eqref{item:close_on_totalrel:composite21}.

We show that
\eqref{item:close_on_totalrel:eta1-contained}
$\Rightarrow$
\eqref{item:close_on_totalrel:eta1-below-eta2}
$\Rightarrow$
\eqref{item:close_on_totalrel:composite12}
$\Rightarrow$
\eqref{item:close_on_totalrel:eta1-contained}.
The implication
\eqref{item:close_on_totalrel:eta1-contained}
$\Rightarrow$
\eqref{item:close_on_totalrel:eta1-below-eta2}
follows from the definition of $\eta_1\prec\eta_2$ and the fact that
$1_X^{\Rel}\in\mathcal{E}_X$.

Assume
\eqref{item:close_on_totalrel:eta1-below-eta2}.
Then there exist $E_X\in\mathcal{E}_X$ and $E_Y\in\mathcal{E}_Y$ such that
$\eta_1\subset E_Y\circ\eta_2\circ E_X$.
Since $\eta_2$ is controlled, $\Ad_{\eta_2}(E_X)\in\mathcal{E}_Y$.
Hence
\begin{align*}
\eta_1\circ\eta_2^\dagger
&\subset
E_Y\circ\eta_2\circ E_X\circ\eta_2^\dagger \\
&=
E_Y\circ\Ad_{\eta_2}(E_X)
\in\mathcal{E}_Y.
\end{align*}
Thus
\eqref{item:close_on_totalrel:composite12}
holds.

Assume
\eqref{item:close_on_totalrel:composite12}.
Put $E_Y:=\eta_1\circ\eta_2^\dagger\in\mathcal{E}_Y$.
Since $\eta_2$ is total, Proposition \ref{proposition:total_condition}
implies $1_X^{\Rel}\subset\eta_2^\dagger\circ\eta_2$.
Therefore
\begin{align*}
\eta_1
&=
\eta_1\circ1_X^{\Rel} \\
&\subset
\eta_1\circ\eta_2^\dagger\circ\eta_2 \\
&=
E_Y\circ\eta_2.
\end{align*}
This proves
\eqref{item:close_on_totalrel:eta1-contained}.

By the same argument with $\eta_1$ and $\eta_2$ interchanged,
\eqref{item:close_on_totalrel:eta2-contained},
\eqref{item:close_on_totalrel:eta2-below-eta1},
and
\eqref{item:close_on_totalrel:composite21}
are equivalent.
Hence
\eqref{item:close_on_totalrel:eta1-below-eta2}
--
\eqref{item:close_on_totalrel:witness-relation}
are all equivalent.

Finally,
\eqref{item:close_on_totalrel:equivalent}
means precisely that both
\eqref{item:close_on_totalrel:eta1-below-eta2}
and
\eqref{item:close_on_totalrel:eta2-below-eta1}
hold.
Since these two conditions are equivalent, this is equivalent to any one of
them.
Thus all conditions
\eqref{item:close_on_totalrel:equivalent}
--
\eqref{item:close_on_totalrel:witness-relation}
are equivalent.
\end{proof}

To prove Theorem \ref{theorem:CtrTot_cong}, it suffices to prove the following proposition.

\begin{proposition}\label{proposition:composition_close_total}
Let $\eta_1,\eta_2\in\CtrTotRel(X,Y)$ and
$\xi_1,\xi_2\in\CtrTotRel(Y,Z)$.
Suppose that $\eta_1\prec\eta_2$ and $\xi_1\prec\xi_2$.
Then
$\xi_1\circ\eta_1\prec\xi_2\circ\eta_2$.
\end{proposition}

\begin{proof}
By Proposition \ref{proposition:close_on_totalrel}, it is enough to show that
$(\xi_1\circ\eta_1)\circ(\xi_2\circ\eta_2)^\dagger\in\mathcal{E}_Z$.
Since $\eta_1\prec\eta_2$ and $\xi_1\prec\xi_2$, Proposition
\ref{proposition:close_on_totalrel} implies
$\eta_1\sim\eta_2$ and $\xi_1\sim\xi_2$.
Hence
$\eta_1\circ\eta_2^\dagger\in\mathcal{E}_Y$ and
$\xi_1\circ\xi_2^\dagger\in\mathcal{E}_Z$.
Moreover, since $\xi_1$ is controlled, we have
$\Ad_{\xi_1}(\eta_1\circ\eta_2^\dagger)\in\mathcal{E}_Z$.

Now Proposition \ref{proposition:total_adjoint}, applied to the total
relation $\xi_1$, gives
\[
\xi_1\circ(\eta_1\circ\eta_2^\dagger)
\subset
\Ad_{\xi_1}(\eta_1\circ\eta_2^\dagger)\circ\xi_1.
\]
Therefore
\begin{align*}
(\xi_1\circ\eta_1)\circ(\xi_2\circ\eta_2)^\dagger
&=
\xi_1\circ(\eta_1\circ\eta_2^\dagger)\circ\xi_2^\dagger \\
&\subset
\Ad_{\xi_1}(\eta_1\circ\eta_2^\dagger)
\circ
(\xi_1\circ\xi_2^\dagger).
\end{align*}
The right-hand side belongs to $\mathcal{E}_Z$, and $\mathcal{E}_Z$ is lower
in $\Rel(Z)$.
Thus
$(\xi_1\circ\eta_1)\circ(\xi_2\circ\eta_2)^\dagger\in\mathcal{E}_Z$.
This proves
$\xi_1\circ\eta_1\prec\xi_2\circ\eta_2$.
\end{proof}

By Theorem \ref{theorem:CtrTot_cong}, the quotient category of
$\Coarse_{\CtrTotRel}$ by the congruence ``$\sim$'' is well-defined.
We denote this category by
$\Coarse_{\CtrTotRel}/{\close}$.
For each
$\eta\in\CtrTotRel(X,Y)$,
its equivalence class in
$\Coarse_{\CtrTotRel}/{\close}$
will be denoted by
$[\eta]_{\CtrTotRel}$,
or simply by $[\eta]$ when no confusion can arise.
We say that a controlled total relation $\eta$ is a coarse equivalence if $[\eta]$ is an isomorphism in $\Coarse_{\CtrTotRel}/{\close}$.

The following theorem shows that
$\Coarse_{\CtrTotRel}/{\close}$
provides an alternative realization of the category of coarse spaces.

\begin{theorem}\label{theorem:CtrMapClo_isom_CtrTotRelClo}
The categories
$\Coarse_{\CtrMap}/{\close}$
and
$\Coarse_{\CtrTotRel}/{\close}$
are isomorphic via the functor defined by
\[
X\mapsto X,
\qquad
[f]_{\CtrMap}\mapsto [f]_{\CtrTotRel}.
\]
\end{theorem}

We first establish the following proposition.

\begin{proposition}
Let $f_1,f_2:X\to Y$ be controlled maps in the sense of
Definition \ref{definition:ctrmap}.
Then the following conditions are equivalent:
\begin{enumerate}
\item
$f_1\sim f_2$ in $\CtrMap(X,Y)$ in the sense of
Definition \ref{definition:close_ctrmap}.

\item
$f_1\sim f_2$ in $\CtrTotRel(X,Y)$ in the sense of
Definition \ref{definition:close_ctrrel}.
\end{enumerate}
In particular, the map
\[
\CtrMap(X,Y)/{\sim}
\longrightarrow
\CtrTotRel(X,Y)/{\sim},
\qquad
[f]_{\CtrMap}
\longmapsto
[f]_{\CtrTotRel}
\]
is well-defined and injective.
\end{proposition}

\begin{proof}
By Definition \ref{definition:close_ctrmap},
the condition $f_1\sim f_2$ means that
\[
\{(f_1(x),f_2(x))\mid x\in X\}
\]
belongs to $\mathcal{E}_Y$.
Since
\begin{align*}
\{(f_1(x),f_2(x))\mid x\in X\}
&= \\
\{
(y_1,y_2)\in Y\times Y
&\mid
\text{there exists }x\in X
\text{ such that }
f_1(x)=y_1,\,
f_2(x)=y_2
\},
\end{align*}
Proposition
\ref{proposition:close_on_totalrel}
(\ref{item:close_on_totalrel:equivalent})
$\Leftrightarrow$
(\ref{item:close_on_totalrel:witness-relation})
shows that this is equivalent to
$f_1\sim f_2$
in the sense of Definition
\ref{definition:close_ctrrel}.
\end{proof}

The above proposition gives a faithful functor
\[
\mathcal{F}
:
\Coarse_{\CtrMap}/{\close}
\longrightarrow
\Coarse_{\CtrTotRel}/{\close}
\]
which is the identity on objects and sends a morphism
$[f]_{\CtrMap}$
to
$[f]_{\CtrTotRel}$.

To prove that $\mathcal{F}$ is an isomorphism of categories, it suffices to
establish the following proposition.

\begin{proposition}[cf.~{\cite[Lemma 3.31]{MartinezVigolo2026roe}}]
\label{proposition:CtrMapCtrTotRel_surj}
Let $X$ and $Y$ be coarse spaces.
Then for every
$\eta\in\CtrTotRel(X,Y)$,
there exists a controlled map
$f\in\CtrMap(X,Y)$
such that
$f\subset\eta$
and
$[f]_{\CtrTotRel}
=
[\eta]_{\CtrTotRel}$.
In particular, the map
\[
\CtrMap(X,Y)/{\sim}
\longrightarrow
\CtrTotRel(X,Y)/{\sim},
\qquad
[f]_{\CtrMap}
\longmapsto
[f]_{\CtrTotRel}
\]
is surjective, and hence bijective.
\end{proposition}

\begin{proof}
Since $\eta$ is total, the Axiom of Choice allows us to choose, for each
$x\in X$, an element $y_x\in Y$ such that
$(y_x,x)\in\eta$.
Define a map
$f:X\to Y$
by
$f(x):=y_x$.

By construction, $f\subset\eta$, and hence
$f\prec\eta$.
Since $\eta$ is controlled,
Propositions
\ref{proposition:controlled_rel_map}
and
\ref{proposition:CtrLowerprec}
imply that $f$ is also controlled.
Furthermore,
Proposition \ref{proposition:close_on_totalrel} yields
$f\sim\eta$.
Therefore
$[f]_{\CtrTotRel}
=
[\eta]_{\CtrTotRel}$.
\end{proof}

\subsection{Controlled total relations which give epi- or mono-morphisms}

Throughout this subsection, we fix coarse spaces
$X=(X,\mathcal{E}_X)$ and $Y=(Y,\mathcal{E}_Y)$,
together with a controlled total relation
$\eta \in \CtrTotRel(X,Y)$.

We now give a characterization of when the close class $[\eta]$ is an epi-morphism in the category of coarse spaces.

\begin{theorem}\label{thm:epi-characterization}
Under the above setting, the following conditions on $\eta$ are equivalent:
\begin{enumerate}
    \item \label{item:epi-characterization-epi}
    The close class $[\eta]$ is an epi-morphism in
    $\Coarse_{\CtrTotRel}/\close$.
    
    \item \label{item:epi-characterization-dense}
    There exists $F \in \mathcal{E}_Y$ such that
    $F(\eta(X))=Y$.
    
    \item \label{item:epi-characterization-total-adjoint}
    There exists $\xi\in[\eta]$ such that the relation
    $\xi^\dagger \in \Rel(Y,X)$ is total.
\end{enumerate}
\end{theorem}

\begin{proof}
We first prove that
\eqref{item:epi-characterization-epi}
and
\eqref{item:epi-characterization-dense}
are equivalent. 
By Propositions \ref{proposition:close_on_totalrel} and
\ref{proposition:CtrMapCtrTotRel_surj},
one can choose a controlled map
$f \in \CtrMap(X,Y)$
with $f \subset \eta$ and
$F_0 \in \mathcal{E}_Y$
with $\eta \subset F_0 \circ f$.
By Theorems \ref{theorem:CtrMap_epimono} and \ref{theorem:CtrMapClo_isom_CtrTotRelClo},
condition
\eqref{item:epi-characterization-epi}
is equivalent to the existence of
$F \in \mathcal{E}_Y$
such that
$F(f(X))=Y$.
In particular, since
$f(X)\subset \eta(X)$,
it follows that
$F(\eta(X))=Y$.
Hence
\eqref{item:epi-characterization-dense}
holds.
Conversely, assume
\eqref{item:epi-characterization-dense},
and choose
$F \in \mathcal{E}_Y$
such that
$F(\eta(X))=Y$.
Then
$(F\circ F_0)(f(X))=Y$.
Since
$F\circ F_0\in\mathcal{E}_Y$,
the above argument shows that
condition
\eqref{item:epi-characterization-epi}
holds.

Next we prove that
\eqref{item:epi-characterization-dense}
and
\eqref{item:epi-characterization-total-adjoint}
are equivalent.
Assume
\eqref{item:epi-characterization-dense},
and fix
$F\in\mathcal{E}_Y$
such that
$F(\eta(X))=Y$.
Put
$\widetilde{F}:=F\cup 1_Y^{\Rel}$.
Then
$F\subset\widetilde{F}$,
$\widetilde{F}\in\mathcal{E}_Y$,
and
$\widetilde{F}$ is total.
Define
$\xi:=\widetilde{F}\circ\eta$.
Then
$\xi\in\CtrTotRel(X,Y)$
and
$\xi(X)=Y$.
In particular,
$\xi^\dagger$
is total.
Moreover,
Proposition \ref{proposition:close_on_totalrel}
implies that
$[\xi]=[\eta]$.
This proves
\eqref{item:epi-characterization-total-adjoint}.
Conversely, assume
\eqref{item:epi-characterization-total-adjoint},
and choose
$\xi\in[\eta]$
such that
$\xi^\dagger$
is total.
Note that
$\xi(X)=Y$.
By Proposition \ref{proposition:close_on_totalrel},
there exists
$F\in\mathcal{E}_Y$
such that 
$\xi\subset F\circ\eta$.
Hence
$F(\eta(X))\supset \xi(X)=Y$.
Therefore
$F(\eta(X))=Y$.
This proves
\eqref{item:epi-characterization-dense}.
\end{proof}

We next give a characterization of when the close class $[\eta]$ is a mono-morphism in the category of coarse spaces.

\begin{theorem}\label{thm:mono-characterization}
Under the above setting, the following conditions on $\eta$ are equivalent:
\begin{enumerate}
    \item\label{item:mono-characterization-mono}
    The close class $[\eta]$ is a mono-morphism in
    $\Coarse_{\CtrTotRel}/\close$.
    \item \label{item:mono-characterization-controlled-adjoint}
    The relation $\eta^\dagger \in \Rel(Y,X)$ is controlled.
\end{enumerate}
\end{theorem}

\begin{proof}
By Proposition \ref{proposition:CtrMapCtrTotRel_surj}, 
one can choose a controlled map $f \in \CtrMap(X,Y)$ with $f \subset \eta$ and $[f] = [\eta]$.
Suppose 
\eqref{item:mono-characterization-controlled-adjoint}.
Since $\eta^\dagger$ is controlled, 
by Proposition \ref{proposition:CtrLowerprec}, 
$f^\dagger \subset \eta^\dagger$ is also controlled.
Thus \eqref{item:mono-characterization-mono} holds by 
Theorems \ref{theorem:CtrMap_epimono} and \ref{theorem:CtrMapClo_isom_CtrTotRelClo}.
Conversely, assume \eqref{item:mono-characterization-mono}.
Then by Theorems \ref{theorem:CtrMap_epimono} and \ref{theorem:CtrMapClo_isom_CtrTotRelClo}, $f^\dagger$ is controlled, 
and hence by Proposition \ref{proposition:CtrLowerprec}, $\eta^\dagger \prec f^\dagger$ is also controlled.
\end{proof}

As a reformulation of Theorem \ref{theorem:coarse_balanced},
we obtain the following characterization of isomorphisms in
$\Coarse_{\CtrTotRel}/\close$.

\begin{theorem}\label{theorem:CtrTotRelClo_isom}
Under the above setting, the following conditions on $\eta$ are equivalent:
\begin{enumerate}
    \item $\eta$ is a coarse equivalence, that is, 
    $[\eta]$ is an isomorphism.    
    \item $\eta^\dagger$ is controlled and there exists
    $F \in \mathcal{E}_Y$ such that
    $F(\eta(X)) = Y$.
    
    \item \label{item:isom:dagger:CtrTotRelClo}
    There exists $\xi \in [\eta]$ such that
    $\xi^\dagger \in \CtrTotRel(Y,X)$.
\end{enumerate}

For such a situation, if $\xi \in [\eta]$ satisfies
Condition \eqref{item:isom:dagger:CtrTotRelClo}, then
$[\xi^\dagger]$ is the inverse of the morphism
$[\eta]=[\xi]$ in the category
$\Coarse_{\CtrTotRel}/\close$.
\end{theorem}

\begin{proof}
The equivalence of the three conditions follows immediately from
Theorems \ref{theorem:coarse_balanced},
\ref{theorem:CtrMapClo_isom_CtrTotRelClo},
\ref{thm:epi-characterization},
and \ref{thm:mono-characterization}.

It remains to prove the final assertion.
Suppose that
$\eta \in \CtrTotRel(X,Y)$
and
$\eta^\dagger \in \CtrTotRel(Y,X)$.
We shall show that
$[\eta]$
and
$[\eta^\dagger]$
are inverse to each other.
Since both $\eta$ and $\eta^\dagger$ are total, 
we have $1_X^{\Rel} \subset \eta^\dagger\circ \eta$ and $1_Y^{\Rel} \subset \eta\circ \eta^\dagger$.
In particular, $1_X^{\Rel} \prec \eta^\dagger\circ \eta$ and $1_Y^{\Rel} \prec \eta\circ \eta^\dagger$ hold.
Thus, by Proposition
\ref{proposition:close_on_totalrel},
\[
1_X^{\Rel}\sim\eta^\dagger \circ \eta,
\qquad
1_Y^{\Rel}\sim \eta \circ \eta^\dagger.
\]

Hence
\[
[\eta^\dagger]\circ[\eta]
=
[1_X^{\Rel}],
\qquad
[\eta]\circ[\eta^\dagger]
=
[1_Y^{\Rel}],
\]
and therefore
$[\eta]$
and
$[\eta^\dagger]$
are inverse morphisms.
\end{proof}

\section{Functoriality of coarse subspaces}\label{section:Fct_coarsesubset}

In Section \ref{subsection:coarsesubset}, for each coarse space
$X=(X,\mathcal E_X)$,
we defined the pre-ordered set
$(\mathcal{P}(X),\prec)$
and its anti-symmetrization
$(\mathcal{P}_c(X),\prec)$.
In this section, we show that the assignments
$X\mapsto(\mathcal{P}(X),\prec)$
and
$X\mapsto(\mathcal{P}_c(X),\prec)$
are functorial.

We begin with the following proposition.

\begin{proposition}\label{proposition:PcXfiniteJoin}
Let
$\mathcal{S}=\{S_\lambda\}_{\lambda\in\Lambda}$
be a finite family of subsets of $X$.
Then the equivalence class
$[\bigcup\mathcal{S}]$
coincides with the least upper bound
$\bigvee\mathcal{S}$ 
in the pre-ordered set $(\mathcal{P}(X),\prec)$.
In particular,
$(\mathcal{P}(X),\prec)$
is a finite-join pre-ordered set, and
$(\mathcal{P}_c(X),\prec)$
is a finite-join partially ordered set.
\end{proposition}

\begin{proof}
It suffices to show that
$\bigcup\mathcal{S}$
represents the least upper bound of
$\mathcal{S}$.
Since
$S_\lambda\subset\bigcup\mathcal{S}$
for every
$\lambda\in\Lambda$,
we have
$S_\lambda\prec\bigcup\mathcal{S}$.
Hence
$\bigcup\mathcal{S}\in\Upper(\mathcal{S})$.

Now let
$A\in\Upper(\mathcal{S})$.
It remains to prove that
$\bigcup\mathcal{S}\prec A$.
By the definition of
$\prec$,
for each
$\lambda\in\Lambda$
there exists
$E_\lambda\in\mathcal E_X$
such that
$S_\lambda\subset E_\lambda(A)$.
Since
$\Lambda$
is finite,
\[
E:=\bigcup_{\lambda\in\Lambda}E_\lambda
\in\mathcal E_X.
\]
By construction,
$\bigcup\mathcal{S}\subset E(A)$.
Hence
$\bigcup\mathcal{S}\prec A$.
\end{proof}

Next we prove the following proposition.

\begin{proposition}\label{proposition:Pc_functorial}
Let
$X$ and $Y$ be coarse spaces, and let
$\eta\in\CtrTotRel(X,Y)$.

\begin{enumerate}
\item
The map
\[
\eta^{\mathcal{P}}:
\mathcal{P}(X)\to\mathcal{P}(Y),
\qquad
S\mapsto\eta(S),
\]
preserves the pre-order
$\prec$
and finite joins, and is open (see Definition \ref{definition:open} for the definition of open maps between pre-ordered sets).
\item
If
$\xi\in\CtrTotRel(X,Y)$
satisfies
$\eta\sim\xi$,
then
$\eta^{\mathcal{P}}\sim\xi^{\mathcal{P}}$ in the sense of Definition \ref{definition:simonpre}.
\end{enumerate}
\end{proposition}

\begin{proof}
We first show that
$\eta^{\mathcal{P}}$
preserves
$\prec$.
Fix
$S_1,S_2\in\mathcal{P}(X)$
with
$S_1\prec S_2$.
We prove that
$\eta(S_1)\prec\eta(S_2)$.
Choose
$E\in\mathcal E_X$
such that
$S_1\subset E(S_2)$.
Since
$\eta$
is controlled,
$F:=\Ad_\eta(E)\in\mathcal E_Y$.
Moreover,
$\eta$
is total, and hence
$\eta^\dagger\circ\eta\supset1_X^{\Rel}$
by Proposition \ref{proposition:total_condition}.
Therefore
\begin{align*}
F(\eta(S_2))
    &= (\eta\circ E\circ\eta^\dagger\circ\eta)(S_2) \\
    &\supset \eta(E(S_2)) \\
    &\supset \eta(S_1).
\end{align*}
Hence
$\eta(S_1)\prec\eta(S_2)$.

Next we show that
$\eta^{\mathcal{P}}$
preserves finite joins.
Let
$\mathcal{S}=\{S_\lambda\}_{\lambda\in\Lambda}$
be a finite family of subsets of $X$.
Since
$\bigvee\mathcal{S}=[\bigcup\mathcal{S}]$
by Proposition \ref{proposition:PcXfiniteJoin},
it suffices to show that
\[
\eta\Bigl(\bigcup_{\lambda}S_\lambda\Bigr)
\sim
\bigcup_{\lambda}\eta(S_\lambda).
\]
In fact,
$\eta^{\mathcal{P}}$
preserves arbitrary unions, and therefore
\[
\eta\Bigl(\bigcup_{\lambda}S_\lambda\Bigr)
=
\bigcup_{\lambda}\eta(S_\lambda).
\]
Thus
$\eta^{\mathcal{P}}$
preserves finite joins.

We next prove that $\eta^\mathcal{P}$ is open. Let $\mathcal{L}$ be a lower set of $(\mathcal{P}(X),\prec)$. It suffices to show that $\eta^{\mathcal{P}}(\mathcal{L})$ is a lower set of $(\mathcal{P}(Y),\prec)$. Let $S \in \mathcal{L}$ and $A \subset Y$ satisfy $A \prec \eta(S)$. We seek an $S_0 \in \mathcal{L}$ such that $\eta(S_0) \sim A$. Since $A \prec \eta(S)$, there exists an $F \in \mathcal{E}_Y$ such that $A \subset F(\eta(S))$. Put $A' := F^\dagger A\cap \eta(S)$. Then $A \sim A'$. Indeed, $A' \subset F^\dagger A$, and, by Proposition \ref{proposition:U1U2R},
\[
A = A \cap F(\eta(S)) \subset F(F^\dagger A \cap \eta(S)) = F A'.
\]
Put $S_0 :=\eta^\dagger A'\cap S$. Since $S_0 \prec S$ and $\mathcal{L}$ is lower, we have $S_0 \in \mathcal{L}$. Furthermore, by Proposition \ref{proposition:U1U2R} again,
\begin{align*}
A' 
    &= A'\cap \eta(S) \quad (\text{since } A' \subset \eta(S)) \\
    &\subset \eta(\eta^\dagger A'\cap  S) \\
    &= \eta(S_0)
\end{align*}
and
\begin{align*}
\eta(S_0) = \eta(S \cap \eta^\dagger A') \subset (\eta \circ \eta^\dagger) A' = (\Ad_{\eta}(1^\Rel_{X})) A'.
\end{align*}
These inclusions prove that $\eta(S_0) \sim A' \sim A$.

Finally, fix
$\xi\in\CtrTotRel(X,Y)$
with
$\eta\sim\xi$.
We prove that
$\eta^{\mathcal{P}}\sim\xi^{\mathcal{P}}$.
Let
$S\in\mathcal{P}(X)$.
It suffices to show that
$\eta(S)\sim\xi(S)$,
that is,
$\eta(S)\prec\xi(S)$
and
$\xi(S)\prec\eta(S)$.
By Proposition \ref{proposition:close_on_totalrel},
$E:=\xi\circ\eta^\dagger
\in
\mathcal E_Y$.
Since
$\eta$
and
$\xi$
are total,
Proposition \ref{proposition:total_condition}
implies that
\[
\eta^\dagger\circ\eta\supset1_X^{\Rel},
\qquad
\xi^\dagger\circ\xi\supset1_X^{\Rel}.
\]
Therefore
\begin{align*}
E(\eta(S))
    &= (\xi\circ\eta^\dagger\circ\eta)(S) \\
    &\supset \xi(S), \\
E^\dagger(\xi(S))
    &= (\eta\circ\xi^\dagger\circ\xi)(S) \\
    &\supset \eta(S).
\end{align*}
Hence
$\eta(S)\prec\xi(S)$
and
$\xi(S)\prec\eta(S)$.
Therefore
$\eta(S)\sim\xi(S)$.
\end{proof}

The above proposition, together with the functoriality of the assignment
$\eta\mapsto\eta^{\mathcal{P}}$
established in Proposition \ref{proposition:eta^P_functorial},
immediately yields the following theorem.

\begin{theorem}\label{theorem:etaPc}
The assignment
\[
X\mapsto(\mathcal{P}(X),\prec),
\qquad
[\eta]\mapsto[\eta^{\mathcal{P}}],
\]
defines a functor
\[
\Coarse_{\CtrTotRel}/\close
\longrightarrow
\hFinJoinPreOrd.
\]
Moreover, the assignment
\[
X\mapsto(\mathcal{P}_c(X),\prec),
\qquad
[\eta]\mapsto
[\eta]^{\mathcal{P}_c}
:=
\AS([\eta^{\mathcal{P}}]),
\]
defines a functor
\[
\mathcal{P}_c : 
\Coarse_{\CtrTotRel}/\close
\longrightarrow
\FinJoinParOrd_{\open}.
\]

Here
\[
[\eta]^{\mathcal{P}_c}
=
\AS([\eta^{\mathcal{P}}]):
\mathcal{P}_c(X)
\longrightarrow
\mathcal{P}_c(Y),
\qquad
[S]\longmapsto[\eta(S)].
\]
\end{theorem}

In particular, we obtain the following corollary.

\begin{corollary}\label{corollary:CE_Pcbij}
Let $\eta\in\CtrTotRel(X,Y)$ be a coarse equivalence. Then the map
\[
[\eta]^{\mathcal{P}_c}:
\mathcal{P}_c(X)
\longrightarrow
\mathcal{P}_c(Y),
\qquad
[S]\longmapsto[\eta(S)]
\]
is a finite-join-preserving bijection (and hence an open map).
\end{corollary}

The following refinement of this corollary also holds:

\begin{theorem}\label{theorem:etaPcspimono}
Let $\eta\in\CtrTotRel(X,Y)$. Suppose that $[\eta]$ is a mono-morphism [resp.~an epi-morphism] in $\Coarse_{\CtrTotRel}/{\close}$. Then, regarded as a functor between thin categories, $\eta^{\mathcal{P}} : (\mathcal{P}(X),\prec) \rightarrow (\mathcal{P}(Y),\prec)$ is fully faithful [resp.~essentially surjective]. In particular, the map $[\eta]^{\mathcal{P}_c} : \mathcal{P}_c(X) \rightarrow \mathcal{P}_c(Y)$ is injective [resp.~surjective].
\end{theorem}

\begin{proof}
Assume first that $[\eta]$ is a mono-morphism. We show that $\eta^\mathcal{P}$ is fully faithful. Since $(\mathcal{P}(X),\prec)$ is regarded as a thin category, faithfulness is automatic. To prove fullness, fix $S_1,S_2 \subset X$ and suppose that $\eta(S_1) \prec \eta(S_2)$. It suffices to show that $S_1 \prec S_2$. Choose $F \in \mathcal{E}_Y$ such that $\eta(S_1) \subset F(\eta(S_2))$. Since $\eta$ is total, Proposition \ref{proposition:total_condition} yields
\begin{align*}
S_1 
    &\subset (\eta^\dagger \eta)(S_1) \\
    &\subset (\eta^\dagger \circ F \circ \eta)(S_2) \\
    &= (\Ad_{\eta^\dagger}(F))(S_2).
\end{align*}
Since $[\eta]$ is a mono-morphism, $\eta^\dagger$ is controlled by Theorem \ref{thm:mono-characterization}. Hence $\Ad_{\eta^\dagger}(F) \in \mathcal{E}_X$, and therefore $S_1 \prec S_2$. Thus $\eta^{\mathcal{P}}$ is fully faithful. In particular, $\eta^{\mathcal{P}}$ preserves non-isomorphisms, and its anti-symmetrization $[\eta]^{\mathcal{P}_c}$ is injective.

Next, assume that $[\eta]$ is an epi-morphism. We show that $\eta^{\mathcal{P}}$ is essentially surjective. Let $S \subset Y$ be arbitrary. It suffices to find a subset $S_0 \subset X$ such that $\eta(S_0) \sim S$. Since $[\eta]$ is an epi-morphism, Theorem \ref{thm:epi-characterization} implies that $\eta(X) \sim Y$. In particular, $S \prec \eta(X)$. Since $\eta^{\mathcal{P}}$ is open by Proposition \ref{proposition:Pc_functorial}, applying this to the lower set $\mathcal{P}(X)$ yields a subset $S_0 \subset X$ such that $\eta(S_0) \sim S$. Thus, $\eta^{\mathcal{P}}$ is essentially surjective. In particular, its anti-symmetrization $[\eta]^{\mathcal{P}_c}$ is surjective.\end{proof}

\section{Coarse subspaces and subobjects}\label{section:subspace_subobject}

Throughout this section, 
for a morphism in
$\Coarse_{\CtrTotRel}/\close$,
we write
$[\eta] :
X \rightarrow_{\mathrm{coarse}} Y$
for
$\eta\in\CtrTotRel(X,Y)$.
Furthermore, the notation 
$[\eta] :
X \rightarrow_{\mathcal{E}} Y$ 
or
$[\eta] :
X \rightarrow_{\mathcal{M}} Y$
means that
$[\eta]\in\mathcal{E}$
or
$[\eta]\in\mathcal{M}$,
respectively.

Let $\mathcal{E}$ [resp.~$\mathcal{M}$] be the class of epi-morphisms
[resp.~mono-morphisms]  in the category of coarse spaces
$\Coarse_{\CtrTotRel}/\close$.
Since $(\mathcal{E},\mathcal{M})$ forms an orthogonal factorization system (by Theorem \ref{theorem:EM_coarse}) on the category of coarse spaces $\Coarse_{\CtrTotRel}/{\close} \cong \Coarse_{\CtrMap}/{\close}$ (see also Theorem \ref{theorem:CtrMapClo_isom_CtrTotRelClo}), Theorems \ref{theorem:fMFCC} and  \ref{theorem:Coarsecocomp} (and Proposition \ref{proposition:CoarseVsmall} below) yield a functor
\begin{align*}
\Sub :
\Coarse_{\CtrTotRel}/{\close}
    &\longrightarrow
\FinJoinParOrd, \\
X
    &\longmapsto
\AS(\mathcal{M}/X) =:\Sub(X), \\
([\eta] : X \rightarrow_{\mathrm{coarse}} Y)
    &\longmapsto 
\AS[[\eta]^{\mathcal{M}}_{\sharp}].
\end{align*}

The purpose of this section is to prove the following theorem:

\begin{theorem}\label{theorem:PcSub}
For each coarse space $X$, the finite-join partially ordered sets $\mathcal{P}_c(X)$ (see Section \ref{subsection:coarsesubset} for the notation) and $\Sub(X)$ are isomorphic. Moreover, the two functors $\mathcal{P}_c$ and $\Sub$ from $\Coarse_{\CtrTotRel}/{\close}$ to $\FinJoinParOrd$ are naturally isomorphic.
\end{theorem}

Furthermore, Theorems \ref{theorem:etaPc} and \ref{theorem:etaPcspimono} yield the following:

\begin{corollary}
\begin{enumerate}
    \item The functor $\Sub : \Coarse_{\CtrTotRel}/{\close} \rightarrow \FinJoinParOrd$ may be regarded as a functor to $\FinJoinParOrd_{\open}$; that is, for every $[\eta] : X \rightarrow_{\mathrm{coarse}} Y$, the map $\AS([[\eta]^\mathcal{M}_\sharp]) : \Sub(X) \rightarrow \Sub(Y)$ is open.
    \item If $[\eta] : X \rightarrow_{\mathrm{coarse}} Y$ is a mono-morphism [respectively, an epi-morphism], then $\AS([[\eta]^\mathcal{M}_\sharp]) : \Sub(X) \rightarrow \Sub(Y)$ is injective [respectively, surjective].
\end{enumerate}
\end{corollary}

Let us start to prove Theorem \ref{theorem:PcSub}.

We first establish the following proposition.

\begin{proposition}\label{proposition:CoarseVsmall}
As declared at the beginning of Sections \ref{section:prelim_coarsespaces} and
\ref{section:realizationCtrTotRel}, 
every set $X$ is assumed to satisfy $X\in\mathcal{U}$ (see Setting \ref{setting:UVC}). 
Under this convention,
the category $\Coarse_{\CtrTotRel}/{\close}$ is $\mathcal{V}$-small.
In particular, the results of Section \ref{subsection:Mslice} can be applied to
$\mathcal{C}=\Coarse_{\CtrTotRel}/{\close}$.
\end{proposition}

\begin{proof}
Let $\mathcal{C}:=\Coarse_{\CtrTotRel}/{\close}$.
By Setting \ref{setting:UVC}, we have $\mathcal{U}\in\mathcal{V}$, and hence
$\mathcal{P}(\mathcal{U})\in\mathcal{V}$.
We shall show that $\Ob(\mathcal{C})\subset\mathcal{U}$ and
$\Mor(\mathcal{C})\subset\mathcal{U}$.
It then follows that
$\Ob(\mathcal{C}),\Mor(\mathcal{C})\in\mathcal{P}(\mathcal{U})\in\mathcal{V}$,
and therefore $\mathcal{C}$ is $\mathcal{V}$-small.

Let $(X,\mathcal{E})$ be a coarse space.
Then $X\in\mathcal{U}$.
Since $\mathcal{U}$ is a universe, we have
$X\times X\in\mathcal{U}$ and
$\mathcal{P}(X\times X)\in\mathcal{U}$.
Since $\mathcal{E}\subset\mathcal{P}(X\times X)$, it follows that
$\mathcal{E}\in\mathcal{U}$.
Therefore $(X,\mathcal{E})\in\mathcal{U}$.
This proves that $\Ob(\mathcal{C})\subset\mathcal{U}$.

Next let
$X=(X,\mathcal{E}_X)$ and
$Y=(Y,\mathcal{E}_Y)$
be coarse spaces.
Then $X,Y\in\mathcal{U}$, and hence
$Y\times X\in\mathcal{U}$ and
$\mathcal{P}(Y\times X)\in\mathcal{U}$.
Since
$\CtrTotRel(X,Y)\subset\mathcal{P}(Y\times X)$,
we obtain
$\CtrTotRel(X,Y)\in\mathcal{U}$.
Let $\eta\in\CtrTotRel(X,Y)$.
Since $[\eta]\subset\CtrTotRel(X,Y)$, it follows that
$[\eta]\in\mathcal{U}$.
Since $X,Y,[\eta]\in\mathcal{U}$ and $\mathcal{U}$ is closed under finite tuples,
the morphism represented by $[\eta]$, regarded as the ordered triple
$(X,Y,[\eta])$, belongs to $\mathcal{U}$.
This proves that $\Mor(\mathcal{C})\subset\mathcal{U}$.
\end{proof}

In what follows, for a subset $S$ of a coarse space
$X=(X,\mathcal{E}_X)$,
let
$\iota_S:S\rightarrow X$
denote the inclusion map.
We always equip $S$ with the induced coarse structure
\begin{align*}
\mathcal{E}_S
    &:=
\Ad_{\iota_S}^{-1}(\mathcal{E}_X) \\
    &=
\{\,E\in\mathcal{E}_X \mid E\subset S\times S\,\} \\
    &=
\{\,E\cap(S\times S)\in\Rel(S)
\mid
E\in\mathcal{E}_X\subset\Rel(X)
=
\mathcal{P}(X\times X)\,\},
\end{align*}
(cf.~Proposition \ref{proposition:invCstr}).

\begin{proposition}\label{proposition:iota_mono}
Under the above setting,
$\iota_S$ is controlled, and
$[\iota_S]$
is a mono-morphism in the category of coarse spaces.
\end{proposition}

\begin{proof}
It is clear from the definition that
$\iota_S$
is total and controlled.
By Theorem \ref{thm:mono-characterization},
it suffices to show that
$\iota_S^\dagger$
is controlled.
Since
\[
\iota_S\circ\iota_S^\dagger
=
\{(s,s)\in X\times X\mid s\in S\}
\subset
1_X^{\Rel},
\]
for every
$E\in\mathcal{E}_X$
we have
\begin{align*}
\Ad_{\iota_S}(\Ad_{\iota_S^\dagger}(E))
    &=
\Ad_{\iota_S\circ\iota_S^\dagger}(E) \\
    &\subset
\Ad_{1_X^{\Rel}}(E) \\
    &=
E
\in
\mathcal{E}_X.
\end{align*}
Hence
$\Ad_{\iota_S^\dagger}(E)
\in
\Ad_{\iota_S}^{-1}(\mathcal{E}_X)
=
\mathcal{E}_S$.
Therefore
$\iota_S^\dagger$
is controlled.
\end{proof}

We prove the following proposition.

\begin{proposition}\label{proposition:Phi_tilde_isom}
We regard the finitely cocomplete thin category $\mathcal{M}/X = \Mono/X$ as a finite-join pre-ordered set in the sense of Section \ref{subsection:pre-order_category}.
Define a map from $(\mathcal{P}(X),\prec)$ to $\mathcal{M}/X$ by
\[
\widetilde{\Phi_X}:\mathcal{P}(X)\rightarrow\mathcal{M}/X,\quad
S\mapsto[\iota_S].
\]
Then $\widetilde{\Phi_X}$ preserves the pre-order and finite joins.
In particular, $\widetilde{\Phi_X}$ is an equivalence between finite join pre-ordered sets $\mathcal{P}(X)$ and $\mathcal{M}/X = \Mono/X$.
\end{proposition}

\begin{proof}
By Proposition \ref{proposition:hpreisom_FJ}, 
it is enough to show that $\widetilde{\Phi_X}$ preserves the pre-order and that $[\widetilde{\Phi_X}]$ is an isomorphism in $\hPreOrd$.

We first prove that $\widetilde{\Phi_X}$ preserves the pre-order. Fix $S_1,S_2\in\mathcal{P}(X)$ with $S_1\prec S_2$. We show that $[\iota_{S_1}]\prec[\iota_{S_2}]$. By the condition $S_1\prec S_2$, there exists $E\in\mathcal{E}_X$ such that $S_1\subset E(S_2)$. Define a relation $\eta\in\Rel(S_1,S_2)$ by
\[
\eta:=\iota_{S_2}^{\dagger}\circ E\circ\iota_{S_1}
=
(S_2\times S_1)\cap E.
\]
By the choice of $E$, the relation $\eta$ is total. Moreover, by Theorem \ref{thm:mono-characterization}, the relations $\iota_{S_i}$ and $\iota_{S_i}^{\dagger}$ are controlled for $i=1,2$, and since $E,E^{\dagger}\in\mathcal{E}_X\subset\CtrRel(X)$, both $\eta$ and $\eta^{\dagger}$ are controlled. Thus $\eta\in\CtrTotRel(S_1,S_2)$ and $[\eta]\in\mathcal{M}$. Furthermore,
\begin{align*}
\iota_{S_2}\circ\eta
    &= (\iota_{S_2}\circ\iota_{S_2}^{\dagger}\circ E)\circ\iota_{S_1} \\
    &= (\Ad_{\iota_{S_2}}(1_{S_2}^{\Rel})\circ E)\circ\iota_{S_1}.
\end{align*}
Since $\Ad_{\iota_{S_2}}(1_{S_2}^{\Rel})\circ E\in\mathcal{E}_X$, Proposition \ref{proposition:close_on_totalrel} implies that $[\iota_{S_2}]\circ[\eta]=[\iota_{S_1}]$. Hence $[\eta]$ gives a morphism from $[\iota_{S_1}]$ to $[\iota_{S_2}]$ in $\mathcal{M}/X$, and therefore $[\iota_{S_1}]\prec[\iota_{S_2}]$.

Next we prove that $[\widetilde{\Phi_X}]$ is an isomorphism in $\hPreOrd$. By Proposition \ref{proposition:preord_hisom}, it is enough to verify the following two conditions:
\begin{enumerate}
    \item If $S_1,S_2\in\mathcal{P}(X)$ satisfy $[\iota_{S_1}]\prec[\iota_{S_2}]$, then $S_1\prec S_2$.
    \item For every $[\xi]\in\mathcal{M}/X$, there exists $S\in\mathcal{P}(X)$ such that $[\iota_S]\sim[\xi]$.
\end{enumerate}

We first verify the first condition. Fix $S_1,S_2\in\mathcal{P}(X)$ with $[\iota_{S_1}]\prec[\iota_{S_2}]$. By the definition of the pre-order on $\mathcal{M}/X$, there exists $[\eta]:S_1\rightarrow_{\mathcal{M}}S_2$ such that $[\iota_{S_2}]\circ[\eta]=[\iota_{S_1}]$. Hence, by Proposition \ref{proposition:close_on_totalrel}, there exists $E\in\mathcal{E}_X$ such that
\[
\iota_{S_1}\subset E\circ\iota_{S_2}\circ\eta.
\]
For this $E$, we have $S_1\subset E(S_2)$. Indeed, since $\eta(S_1)\subset S_2$,
\begin{align*}
S_1
    &= \iota_{S_1}(S_1) \\
    &\subset (E\circ\iota_{S_2}\circ\eta)(S_1) \\
    &\subset (E\circ\iota_{S_2})(S_2) \\
    &= E(S_2).
\end{align*}
Thus $S_1\prec S_2$.

Finally, we verify the second condition. Fix $[\xi]:Z\rightarrow_{\mathcal{M}}X$, and put $S:=\xi(Z)$. We show that $[\iota_S]\sim[\xi]$. Define $\xi_S:=\iota_S^{\dagger}\circ\xi\in\Rel(Z,S)$. By Proposition \ref{proposition:iotaxi_Omega}, the relation $\xi_S$ is total, satisfies $\xi_S(Z)=\xi(Z)=S$, and satisfies $\iota_S\circ\xi_S=\xi$. Moreover, $\iota_S$, $\iota_S^{\dagger}$, $\xi$, and $\xi^{\dagger}$ are controlled by Theorem \ref{thm:mono-characterization} and Proposition \ref{proposition:iota_mono}. Hence $\xi_S=\iota_S^{\dagger}\circ\xi$ and $\xi_S^{\dagger}=\xi^{\dagger}\circ\iota_S$ are controlled. Thus $\xi_S\in\CtrTotRel(Z,S)$ and $[\xi_S]\in\mathcal{M}$. Since $\xi_S(Z)=S$, Theorem \ref{theorem:CtrTotRelClo_isom} implies that $[\xi_S]$ is an isomorphism from $Z$ to $S$ in the category of coarse spaces. Since $\iota_S\circ\xi_S=\xi$, we have $[\iota_S]\circ[\xi_S]=[\xi]$. This means that $[\iota_S]\sim[\xi]$.
\end{proof}

Recall that the maps
\[
\mathcal{P}(X)\rightarrow\mathcal{P}_c(X),\quad S\mapsto [S],
\]
and
\[
\mathcal{M}/X\rightarrow\AS(\mathcal{M}/X) = \Sub(X),\quad [\eta]\mapsto \AS([\eta]),
\]
are both isomorphisms in $\hFinJoinPreOrd$, and that $\FinJoinParOrd$ is a full subcategory of $\hFinJoinPreOrd$. Combining this with the proposition above, we immediately obtain the following corollary.

\begin{corollary}\label{corollary:PhiXisisom}
The map
\[
\Phi_X:\mathcal{P}_c(X)\rightarrow \Sub(X),\quad [S]\mapsto\AS([\iota_S])
\]
is an isomorphism in $\FinJoinParOrd$.
\end{corollary}

The first assertion of Theorem \ref{theorem:PcSub} follows from Corollary \ref{corollary:PhiXisisom}. The second assertion follows from the following proposition.

\begin{proposition}\label{proposition:PhigivesNatIsom}
The assignment $X\mapsto\Phi_X$ defines a natural isomorphism from the functor $\mathcal{P}_c$ to the functor $\Sub$.
\end{proposition}

\begin{proof}[Proof of Proposition \ref{proposition:PhigivesNatIsom}]
By Corollary \ref{corollary:PhiXisisom}, it suffices to show that the assignment $X\mapsto\Phi_X$ defines a natural transformation from $\mathcal{P}_c$ to $\Sub$.

Let $X$ and $Y$ be coarse spaces, and fix $\eta\in\CtrTotRel(X,Y)$. It suffices to prove the equality
\[
\Phi_Y\circ[\eta]^{\mathcal{P}_c}
=
\AS[[\eta]^{\mathcal{M}}_{\sharp}] \circ\Phi_X
\]
in $\FinJoinParOrd$.

Fix a subset $S$ of $X$. It is enough to show that
\[
(\Phi_Y\circ[\eta]^{\mathcal{P}_c})([S])
=
(\AS[[\eta]^{\mathcal{M}}_{\sharp}]\circ\Phi_X)([S]).
\]

The left-hand side is
\begin{align*}
(\Phi_Y\circ[\eta]^{\mathcal{P}_c})([S])
    &= \Phi_Y([\eta(S)]) \\
    &= \AS([\iota_{\eta(S)}]).
\end{align*}

The right-hand side is
\[
(\AS[[\eta]^{\mathcal{M}}_{\sharp}] \circ\Phi_X)([S])
=
\AS[[\eta]^{\mathcal{M}}_{\sharp}](\AS([\iota_S])).
\]
Put $L:=\eta(S)$. 
Then by Theorem \ref{theorem:fMFCC}, 
it suffices to show that there exists a morphism
\[
e:S\rightarrow_{\mathcal{E}}L = \eta(S)
\]
such that $(e,[\iota_L])$ gives an $(\mathcal{E},\mathcal{M})$-factorization of
$[\eta]\circ[\iota_S]$ in $\Coarse_{\CtrTotRel}/{\close}$.

Define
\[
\eta_L
:=
\iota_L^\dagger\circ\eta\circ\iota_S
\in
\Rel(S,L).
\]
We shall show that
$\eta_L\in\CtrTotRel(S,L)$,
that
$[\eta_L]$ is an epi-morphism,
and that
\[
[\iota_L]\circ[\eta_L]
=
[\eta]\circ[\iota_S].
\]
Once these facts are established, taking
$e:=[\eta_L]$
gives the desired factorization.

Since
\[
(\eta\circ\iota_S)(S)
=
\eta(S)
=
L,
\]
by applying Proposition \ref{proposition:iotaxi_Omega}
with $(\xi,Z,\Omega,\xi_{\Omega}) := (\eta \circ \iota_S,S,L,\eta_L)$, 
we have that $\eta_L(S)=L$, 
\[
\iota_L\circ\eta_L = \eta\circ\iota_S, 
\]
and 
$\eta_L$ is total.
Moreover,
$\iota_S$ and $\eta$ are both controlled, and
$\iota_L^\dagger$ is also controlled by
Theorem \ref{thm:mono-characterization}
and Proposition \ref{proposition:iota_mono}.
Hence
\[
\eta_L
=
\iota_L^\dagger\circ\eta\circ\iota_S
\]
is controlled.
Since
$\eta_L(S)=L$,
Condition~\eqref{item:epi-characterization-dense}
of Theorem \ref{thm:epi-characterization} is satisfied with
$F=1_L^{\Rel}$.
Therefore
$[\eta_L]$
is an epi-morphism.
This completes the proof.
\end{proof}

\section{Asymptotically disjoint pairs of coarse subspaces}\label{section:asymptoticallydisjoint}

In this section, we introduce the binary relation of being asymptotically disjoint between coarse subspaces of a coarse space and show that it is preserved under coarse equivalences. The results of this section provide a refinement of \cite[Section 4]{NagayaOgawaOkuda2025}.

\subsection{Bornological relations between coarse spaces}

We begin by introducing two notions of boundedness for subsets of coarse spaces:

\begin{definition}
Let $X$ be a coarse space. Throughout this paper, we define the subset $\mathcal{B}(X)$ of $\mathcal{P}(X)$ by
\[
\mathcal{B}(X) = \mathcal{B}(\mathcal{E}_X) := \{ E(B_0) \mid B_0 \text{ is a finite subset of } X, \text{ and } E \in \mathcal{E} \}.
\]
We also put
\[
\mathcal{B}^{\mathrm{conn}}(X) = \mathcal{B}^{\mathrm{conn}}(\mathcal{E}_X) := \{ E(\{ x \}) \mid x \in X, E \in \mathcal{E} \}.
\]
\end{definition}

Note that if $X$ is coarsely-connected (i.e.~for each $x_1,x_2 \in X$, there exists $E \in \mathcal{E}_X$ with $(x_1,x_2) \in E$: cf.~\cite[Definition 2.11]{Roe2003LectureCoarse}), 
then $\mathcal{B}(X) = \mathcal{B}^{\mathrm{conn}}(X)$.

The following proposition is immediate from the definition:

\begin{proposition}\label{propostion:stdborstr}
$\mathcal{B}^{\mathrm{conn}}(X)$ is lower in $(\mathcal{P}(X),\prec)$, covers $X$, and contains the empty set; that is,
\begin{itemize}
    \item For each $B \in \mathcal{B}^{\mathrm{conn}}(X)$ and $B_0 \prec B$, $B_0 \in \mathcal{B}^{\mathrm{conn}}(X)$. 
    \item $\bigcup \mathcal{B}^{\mathrm{conn}}(X) = X$.
    \item $\emptyset \in \mathcal{B}^{\mathrm{conn}}(X)$.
\end{itemize}
Furthermore, $\mathcal{B}(X)$ is lower in $(\mathcal{P}(X),\prec)$, covers $X$, and is closed under finite unions.
\end{proposition}

We also have the following:

\begin{proposition}[Roe~{\cite[Proposition 2.16]{Roe2003LectureCoarse}}]\label{proposition:Roebounded}
For a subset $B$ of $X$, $B \in \mathcal{B}^{\mathrm{conn}}(X)$ if and only if the relation $B \times B \in \Rel(X)$ belongs to $\mathcal{E}_X$.
\end{proposition}

\begin{remark}
A subset $B \in \mathcal{B}^{\mathrm{conn}}(X)$ is called \emph{bounded} by Roe~\cite[Proposition 2.16]{Roe2003LectureCoarse}. The subset $\mathcal{B}(X)$ of $\mathcal{P}(X)$ defined above coincides with what Bunke--Engel~\cite[Example 2.16]{BunkeEngel2020} call the minimal bornological structure compatible with the coarse structure $\mathcal{E}_X$.
\end{remark}

We introduce the notion of a bornological relation between coarse spaces in the following two senses.

\begin{definition}\label{definition:bornological}
Throughout this paper, a relation $R \in \Rel(X,Y)$ is said to be:
\begin{enumerate}
    \item $(\mathcal{B}^{\mathrm{conn}}(X),\mathcal{B}^{\mathrm{conn}}(Y))$-bornological if $R^{\mathcal{P}}(\mathcal{B}^{\mathrm{conn}}(X)) \subset \mathcal{B}^{\mathrm{conn}}(Y)$, where
    \[
    R^\mathcal{P} : \mathcal{P}(X) \rightarrow \mathcal{P}(Y), ~ S \mapsto R(S)
    \]
    denotes the map introduced in Section \ref{section:categoryofrel}.
    \item $(\mathcal{B}(X),\mathcal{B}(Y))$-bornological if $R^{\mathcal{P}}(\mathcal{B}(X)) \subset \mathcal{B}(Y)$.
\end{enumerate}
\end{definition}

\begin{remark}
For a map $f : X \rightarrow Y$, 
$f$ is a $(\mathcal{B}(X),\mathcal{B}(Y))$-bornological relation if and only if $f$ is a bornological map between bornological spaces $(X,\mathcal{B}(X))$ and $(Y,\mathcal{B}(Y))$ in the sense of \cite[Definition 2.8]{BunkeEngel2020}.
\end{remark}

\begin{proposition}\label{proposition:bornological_lower}
The two types of bornological properties defined in Definition \ref{definition:bornological} are both lower conditions in $(\Rel(X,Y),\prec)$.
\end{proposition}

\begin{proof}
Put $\mathcal{L} := \mathcal{B}^{\mathrm{conn}}$ or $\mathcal{L} := \mathcal{B}$.
Take any $(\mathcal{L}(X),\mathcal{L}(Y))$-bornological relation $\eta \in \Rel(X,Y)$, and any relation $R \in \Rel(X,Y)$ with $R \prec \eta$.
We shall prove that $R$ is also $(\mathcal{L}(X),\mathcal{L}(Y))$-bornological.
Fix any $B \in \mathcal{L}(X)$. Our goal is to show that $R(B) \in \mathcal{L}(Y)$.
Since $R \prec \eta$, 
there exists $E \in \mathcal{E}_X$, $F \in \mathcal{E}_Y$ such that $R \subset F \circ \eta \circ E$.
Since 
$\mathcal{L}(X)$ is lower in $(\mathcal{P}(X),\prec)$, 
$\eta^\mathcal{P}(\mathcal{L}(X)) \subset \mathcal{L}(Y)$, 
and 
$\mathcal{L}(Y)$ is lower in $(\mathcal{P}(Y),\prec)$, 
we have 
\begin{align*}
R(B) \subset (F \circ \eta \circ E)(B) \in \mathcal{L}(Y).
\end{align*}
Since  
$\mathcal{L}(Y)$ is lower in $(\mathcal{P}(Y),\prec)$, and hence lower in $(\mathcal{P}(Y),\subset)$, 
we have $R(B) \in \mathcal{L}(Y)$.
\end{proof}

The following implications relate controlled relations to the two notions of bornologicality introduced above:

\begin{proposition}\label{proposition:cont_prev_bornology}
The following implications hold:
\begin{enumerate}
    \item Each controlled relation from $X$ to $Y$ is $(\mathcal{B}^{\mathrm{conn}}(X),\mathcal{B}^{\mathrm{conn}}(Y))$-bornological.
    \item Each $(\mathcal{B}^{\mathrm{conn}}(X),\mathcal{B}^{\mathrm{conn}}(Y))$-bornological relation is $(\mathcal{B}(X),\mathcal{B}(Y))$-bornological.
\end{enumerate}
\end{proposition}

\begin{proof}
We first prove the first assertion. Let $\eta \in \CtrRel(X,Y)$ and $B \in \mathcal{B}^{\mathrm{conn}}(X)$. We show that $\eta(B) \in \mathcal{B}^{\mathrm{conn}}(Y)$. If $B \cap \Dom(\eta) = \emptyset$, then $\eta(B) = \emptyset \in \mathcal{B}^{\mathrm{conn}}(Y)$. Assume that $B \cap \Dom(\eta) \neq \emptyset$. Choose $x \in B \cap \Dom(\eta)$ and $y \in \eta(\{x\})$. Since $x \in \eta^\dagger(\{y\})$, we have
\[
\eta(\{x\}) \subset \eta(\eta^\dagger(\{y\})) = \Ad_\eta(1^\Rel_X)(\{y\}).
\]
Since $\eta$ is controlled, $\Ad_\eta(1^\Rel_X) \in \mathcal{E}_Y$, and hence $\eta(\{x\}) \in \mathcal{B}^{\mathrm{conn}}(Y)$. Moreover, Proposition \ref{proposition:Roebounded} yields $E_B := B \times B \in \mathcal{E}_X$. Since $B = E_B(\{x\})$ and $\Ad_\eta(E_B) \in \mathcal{E}_Y$, Proposition \ref{proposition:total_adjoint} gives
\begin{align*}
\eta(B)
&= \eta(E_B(\{x\})) \\
&\subset \Ad_\eta(E_B)(\eta(\{x\})).
\end{align*}
The right-hand side belongs to $\mathcal{B}^{\mathrm{conn}}(Y)$, and therefore so does $\eta(B)$.

For the second assertion, each element of $\mathcal{B}(X)$ is a finite union of elements of $\mathcal{B}^{\mathrm{conn}}(X)$, the map $\eta^{\mathcal{P}}$ preserves unions, and $\mathcal{B}(Y)$ is closed under finite unions. Hence every $(\mathcal{B}^{\mathrm{conn}}(X),\mathcal{B}^{\mathrm{conn}}(Y))$-bornological relation is $(\mathcal{B}(X),\mathcal{B}(Y))$-bornological.
\end{proof}

\subsection{Proper relations between coarse spaces}

We introduce the notion of a proper relation between coarse spaces in the following two senses.

\begin{definition}\label{definition:proper}
Throughout this paper, a relation $\eta \in \Rel(X,Y)$ is said to be $(\mathcal{B}^{\mathrm{conn}}(X),\mathcal{B}^{\mathrm{conn}}(Y))$-proper [resp.~$(\mathcal{B}(X),\mathcal{B}(Y))$-proper] if $\eta^\dagger$ is $(\mathcal{B}^{\mathrm{conn}}(Y),\mathcal{B}^{\mathrm{conn}}(X))$-bornological [resp.~$(\mathcal{B}(Y),\mathcal{B}(X))$-bornological].
\end{definition}

\begin{remark}
A map $f : X \rightarrow Y$ is $(\mathcal{B}^{\mathrm{conn}}(X),\mathcal{B}^{\mathrm{conn}}(Y))$-proper in the above sense if and only if it is proper in the sense of Roe~\cite[Definition 2.21]{Roe2003LectureCoarse}. Furthermore, $f$ is $(\mathcal{B}(X),\mathcal{B}(Y))$-proper if and only if it is proper as a map between the bornological spaces $(X,\mathcal{B}(X))$ and $(Y,\mathcal{B}(Y))$ in the sense of \cite[Definition 2.8]{BunkeEngel2020}.
\end{remark}

By Proposition \ref{proposition:bornological_lower}, we obtain the following:

\begin{proposition}\label{proposition:properlower}
The two notions of properness introduced in Definition \ref{definition:proper} are both lower conditions in $(\Rel(X,Y),\prec)$.
\end{proposition}

The controlledness of $\eta^\dagger$ and the two notions of properness introduced above are related by the following implications:

\begin{proposition}\label{proposition:CE_is_proper}
Let $\eta \in \Rel(X,Y)$.
\begin{enumerate}
    \item If $\eta^\dagger$ is controlled, then $\eta$ is $(\mathcal{B}^{\mathrm{conn}}(X),\mathcal{B}^{\mathrm{conn}}(Y))$-proper.
    \item If $\eta$ is $(\mathcal{B}^{\mathrm{conn}}(X),\mathcal{B}^{\mathrm{conn}}(Y))$-proper, then it is $(\mathcal{B}(X),\mathcal{B}(Y))$-proper.
\end{enumerate}
\end{proposition}

\begin{proof}
The first assertion follows by applying Proposition \ref{proposition:cont_prev_bornology} to $\eta^\dagger$. The second assertion follows immediately by the same proposition.
\end{proof}

The following proposition will be used in a later section:

\begin{proposition}\label{proposition:proper_prev_unbounded}
Let $\eta$ be a $(\mathcal{B}(X),\mathcal{B}(Y))$-proper relation from $X$ to $Y$. 
Then
% $(\eta^\mathcal{P})^{-1}(\mathcal{B}(Y)) \subset \mathcal{B}(X)$, that is, 
for each $U \subset \Dom(\eta)$,
\[
\eta(U) \in \mathcal{B}(Y) \Rightarrow 
U \in \mathcal{B}(X).
\]
In particular, if $\eta$ is total, 
\[
(\eta^\mathcal{P})^{-1}(\mathcal{B}(Y)) \subset \mathcal{B}(X)
\]
holds.
\end{proposition}

\begin{proof}
Fix any $U \subset \Dom(\eta)$ with $\eta(U) \in \mathcal{B}(Y)$. We shall prove that $U \in \mathcal{B}(X)$. Since $\eta$ is $(\mathcal{B}(X),\mathcal{B}(Y))$-proper, we have $\eta^\dagger(\eta(U)) \in \mathcal{B}(X)$. Since $U \subset \Dom(\eta)$, Proposition \ref{proposition:SDometatotal} gives
\[
U \subset (\eta^\dagger \circ \eta)(U) = \eta^\dagger(\eta(U)) \in \mathcal{B}(X).
\]
Since $\mathcal{B}(X)$ is lower in $(\mathcal{P}(X),\subset)$, we obtain $U \in \mathcal{B}(X)$.
\end{proof}

\subsection{Asymptotically disjoint pairs}

Throughout this subsection, let $X = (X,\mathcal{E})$ be a coarse space. Following Banakh--Protasov~\cite{BanakhProtasov2018}, we introduce the notion of asymptotic disjointness as a binary relation on $\mathcal{P}(X)$ and $\mathcal{P}_c(X)$:

\begin{definition}\label{definition:AsD}
Let $X = (X,\mathcal{E}_X)$ be a coarse space. A pair of subsets $(S_1,S_2)$ of $X$ is said to be \emph{asymptotically disjoint} in $X$ if $E_1(S_1) \cap E_2(S_2) \in \mathcal{B}(X)$ for all $E_1,E_2 \in \mathcal{E}_X$. In this paper, we write $S_1 \adperp S_2$ in $X$ if $(S_1,S_2)$ is asymptotically disjoint in $X$.
\end{definition}

The following proposition gives necessary and sufficient conditions for $S_1$ and $S_2$ to be asymptotically disjoint:

\begin{proposition}[cf.~{\cite[Corollary~4.2]{GrzegrzolkaSiegertnormal2019}}]
\label{proposition:ADequivcond}
Let $S_1,S_2 \subset X$. Then the following four conditions on $(S_1,S_2)$ are equivalent:
\begin{enumerate}
    \item\label{item:ADequivcond:AD} $S_1 \adperp S_2$ in $X$.
    \item\label{item:ADequivcond:ES1capS2} $E(S_1) \cap S_2 \in \mathcal{B}(X)$ holds for each $E \in \mathcal{E}_X$.
    \item\label{item:ADequivcond:S1capES2} $S_1 \cap E(S_2) \in \mathcal{B}(X)$ holds for each $E \in \mathcal{E}_X$.
    \item\label{item:ADequivcond:commonlowerbound} $\Low(S_1,S_2) \subset \mathcal{B}(X)$,
    where we write 
    \[
    \Low(S_1,S_2) := \{ B \subset X \mid B \prec S_1 \text{ and } B \prec S_2 \}.
    \]
\end{enumerate}
\end{proposition}

\begin{proof}
The implications \eqref{item:ADequivcond:AD} $\Rightarrow$ \eqref{item:ADequivcond:ES1capS2} and \eqref{item:ADequivcond:AD} $\Rightarrow$ \eqref{item:ADequivcond:S1capES2} are immediate. We prove \eqref{item:ADequivcond:ES1capS2} $\Rightarrow$ \eqref{item:ADequivcond:AD}. Assume that \eqref{item:ADequivcond:ES1capS2} holds. Fix $E_1,E_2 \in \mathcal{E}_X$. By Proposition \ref{proposition:U1U2R},
\[
E_1(S_1) \cap E_2(S_2) \subset E_2((E_2^\dagger \circ E_1)(S_1) \cap S_2),
\]
and hence
\[
E_1(S_1) \cap E_2(S_2) \prec (E_2^\dagger \circ E_1)(S_1) \cap S_2.
\]
Since \eqref{item:ADequivcond:ES1capS2} holds, $(E_2^\dagger \circ E_1)(S_1) \cap S_2 \in \mathcal{B}(X)$. Since $\mathcal{B}(X)$ is lower in $(\mathcal{P}(X),\prec)$ by Proposition \ref{propostion:stdborstr}, it follows that $E_1(S_1) \cap E_2(S_2) \in \mathcal{B}(X)$. Thus \eqref{item:ADequivcond:AD} holds. The implication \eqref{item:ADequivcond:S1capES2} $\Rightarrow$ \eqref{item:ADequivcond:AD} is proved similarly.

Finally, the equivalence between \eqref{item:ADequivcond:AD} and \eqref{item:ADequivcond:commonlowerbound} follows from the definition of the pre-order $\prec$ (see Definition \ref{definition:similar}) and the fact that $\mathcal{B}(X)$ is lower in $(\mathcal{P}(X),\subset)$.
\end{proof}

The following is an immediate consequence of Proposition \ref{proposition:ADequivcond}:

\begin{proposition}[cf.~{\cite[Lemma~4.10]{KalantariHonari2016}}]
The binary relation $\adperp$ on $\mathcal{P}(X)$ is invariant under the equivalence relation $\sim$ defined in Definition~\ref{definition:similar}. Namely, if $S_1 \sim S'_1$ and $S_2 \sim S'_2$, then $S_1 \adperp S_2$ if and only if $S'_1 \adperp S'_2$. In particular, the binary relation $\adperp$ on $\mathcal{P}_c(X)$ defined by 
\[
[S_1] \adperp [S_2] \text{ in } X \defarrow S_1 \adperp S_2 \text{ in } X
\]
is well-defined.
\end{proposition}

\begin{example}\label{example:Kobayashipitchfork}
Let $G$ be a locally compact Hausdorff group equipped with the coarse structure $\mathcal{E}^{\mathrm{LR}}_G$ defined in Example \ref{example:groupaction_coarse}. Then
\begin{align*}
\mathcal{B}(G) = \mathcal{B}^{\mathrm{conn}}(G) 
= \{ B \subset G \mid B \text{ is relatively compact in } G \}.
\end{align*}
Furthermore, $S_1 \adperp S_2$ in $(G,\mathcal{E}^{\mathrm{LR}}_G)$ if and only if $S_1 \pitchfork S_2$ in $G$ in the sense of Kobayashi~\cite{Kobayashi96}. In particular, for closed subgroups $H$ and $L$ of $G$, the following conditions are equivalent (cf.~Kobayashi~\cite{Kobayashi89,Kobayashi92Fuji,Kobayashi96}; see also \cite[Remark 3.1.3]{KobayashiYoshino05}):
\begin{enumerate}
    \item $H \adperp L$ in $(G,\mathcal{E}^{\mathrm{LR}}_G)$.
    \item The $H$-action on the homogeneous space $G/L$ is proper.
    \item The $L$-action on the homogeneous space $G/H$ is proper.
    \item The diagonal $G$-action on $G/L \times G/H$ is proper.
\end{enumerate}
\end{example}

\subsection{Asymptotic Disjointness and Proper Controlled Total Relations}

For each coarse space $X$, put
\[
\mathcal{B}_c(X) := \{ [B] \mid B \in \mathcal{B}(X) \} \subset \mathcal{P}_c(X).
\]
Then $\mathcal{B}_c(X)$ is lower in $(\mathcal{P}_c(X),\prec)$.

Let $X$ and $Y$ be coarse spaces. Fix $[\eta] : X \rightarrow_{\mathrm{coarse}} Y$, and consider the map $[\eta]^{\mathcal{P}_c} : \mathcal{P}_c(X) \rightarrow \mathcal{P}_c(Y)$ defined in Theorem \ref{theorem:etaPc}. Then Proposition \ref{proposition:cont_prev_bornology} yields
\[
[\eta]^{\mathcal{P}_c}(\mathcal{B}_c(X)) \subset \mathcal{B}_c(Y).
\]

We say that $[\eta]$ is $(\mathcal{B}_c(X),\mathcal{B}_c(Y))$-proper if a representative $\eta \in \CtrTotRel(X,Y)$ is $(\mathcal{B}(X),\mathcal{B}(Y))$-proper. By Proposition \ref{proposition:bornological_lower}, this condition is independent of the choice of representative.

By Proposition \ref{proposition:proper_prev_unbounded}, if $[\eta]$ is $(\mathcal{B}_c(X),\mathcal{B}_c(Y))$-proper, then
\[
([\eta]^{\mathcal{P}_c})^{-1}(\mathcal{B}_c(Y)) \subset \mathcal{B}_c(X).
\]

Moreover, as follows immediately from Theorem \ref{thm:mono-characterization} and Proposition \ref{proposition:CE_is_proper}, if $[\eta]$ is a mono-morphism in $\Coarse_{\CtrTotRel}/{\close}$, then it is $(\mathcal{B}_c(X),\mathcal{B}_c(Y))$-proper.

As we shall see below, non-asymptotic disjointness is preserved by $(\mathcal{B}_c(X),\mathcal{B}_c(Y))$-proper morphisms, whereas asymptotic disjointness is preserved by mono-morphisms:

\begin{theorem}\label{theorem:properAD}
The following statements hold for $\eta \in \CtrTotRel(X,Y)$:
\begin{enumerate}
    \item If $[\eta]$ is $(\mathcal{B}_c(X),\mathcal{B}_c(Y))$-proper, then $[\eta]^{\mathcal{P}_c}$ preserves non-asymptotic disjointness. That is, for $S_1,S_2 \subset X$,
    \[
    \eta(S_1) \adperp \eta(S_2) \text{ in } Y \Rightarrow S_1 \adperp S_2 \text{ in } X.
    \]
    \item If $[\eta]$ is a mono-morphism, then $[\eta]^{\mathcal{P}_c}$ preserves both asymptotic disjointness and non-asymptotic disjointness. That is, for $S_1,S_2 \subset X$,
    \[
    \eta(S_1) \adperp \eta(S_2) \text{ in } Y \Leftrightarrow S_1 \adperp S_2 \text{ in } X.
    \]
\end{enumerate}
\end{theorem}

\begin{proof}
Fix $S_1,S_2 \subset X$. By Proposition \ref{proposition:ADequivcond}, the condition that $S_1 \adperp S_2$ in $X$ is equivalent to
\[
\Low([S_1],[S_2]) \subset \mathcal{B}_c(X).
\]
Likewise, the condition that $\eta(S_1) \adperp \eta(S_2)$ in $Y$ is equivalent to
\[
\Low([\eta(S_1)],[\eta(S_2)]) \subset \mathcal{B}_c(Y).
\]
Here we put
\begin{align*}
\Low([S_1],[S_2]) &:= \{ [B] \in \mathcal{P}_c(X) \mid [B] \prec [S_1] \text{ and } [B] \prec [S_2] \} \subset \mathcal{P}_c(X), \\
\Low([\eta(S_1)],[\eta(S_2)]) &:= \{ [B] \in \mathcal{P}_c(Y) \mid [B] \prec [\eta(S_1)] \text{ and } [B] \prec [\eta(S_2)] \} \subset \mathcal{P}_c(Y).
\end{align*}
Since $[\eta]^{\mathcal{P}_c}$ preserves the partial orders, we also note that
\[
[\eta]^{\mathcal{P}_c}(\Low([S_1],[S_2])) \subset \Low([\eta(S_1)],[\eta(S_2)])
\]
holds.

Assume that $[\eta]$ is $(\mathcal{B}_c(X),\mathcal{B}_c(Y))$-proper and that $\eta(S_1) \adperp \eta(S_2)$, or equivalently,
\[
\Low([\eta(S_1)],[\eta(S_2)]) \subset \mathcal{B}_c(Y).
\]
Then
\[
[\eta]^{\mathcal{P}_c}(\Low([S_1],[S_2])) \subset \Low([\eta(S_1)],[\eta(S_2)]) \subset \mathcal{B}_c(Y).
\]
Since $[\eta]$ is $(\mathcal{B}_c(X),\mathcal{B}_c(Y))$-proper, we have
\[
([\eta]^{\mathcal{P}_c})^{-1}(\mathcal{B}_c(Y)) \subset \mathcal{B}_c(X).
\]
Therefore,
\[
\Low([S_1],[S_2]) \subset \mathcal{B}_c(X).
\]
Hence $S_1 \adperp S_2$ in $X$.

Assume now that $[\eta]$ is a mono-morphism. In particular, $[\eta]$ is $(\mathcal{B}_c(X),\mathcal{B}_c(Y))$-proper, and the preceding argument has shown that
\[
\eta(S_1) \adperp \eta(S_2) \text{ in } Y \Rightarrow S_1 \adperp S_2 \text{ in } X.
\]
We prove the converse implication. Assume that $S_1 \adperp S_2$, that is,
\[
\Low([S_1],[S_2]) \subset \mathcal{B}_c(X).
\]
Since $[\eta]$ is a mono-morphism, Theorem \ref{theorem:etaPcspimono} implies that $[\eta]^{\mathcal{P}_c}$ is injective. In particular, 
since $[\eta]^{\mathcal{P}_c}$ is open by Theorem \ref{theorem:etaPc}, 
Proposition \ref{proposition:LowOmega} yields
\[
[\eta]^{\mathcal{P}_c}(\Low([S_1],[S_2])) = \Low([\eta(S_1)],[\eta(S_2)]).
\]
It follows that
\[
\Low([\eta(S_1)],[\eta(S_2)]) = [\eta]^{\mathcal{P}_c}(\Low([S_1],[S_2])) \subset [\eta]^{\mathcal{P}_c}(\mathcal{B}_c(X)) \subset \mathcal{B}_c(Y).
\]
Thus, $\eta(S_1) \adperp \eta(S_2)$ in $Y$.
\end{proof}

By combining Theorem \ref{theorem:properAD} with Corollary \ref{corollary:CE_Pcbij}, we obtain the following:

\begin{corollary}\label{corollary:CE_adperp}
Suppose that $\eta$ is a coarse equivalence. Then the isomorphism $[\eta]^{\mathcal{P}_c}$ between the finite-join partially ordered sets $\mathcal{P}_c(X)$ and $\mathcal{P}_c(Y)$ preserves the binary relation $\adperp$.
\end{corollary}

\begin{example}
Let $G$ be a linear reductive Lie group equipped with the coarse structure $\mathcal{E}^{\mathrm{LR}}_G$ defined in Example \ref{example:groupaction_coarse}. By combining Corollary \ref{corollary:CE_adperp} with the arguments in Example \ref{example:Cartan}, one sees that the following conditions on a pair $(S_1,S_2)$ of subsets of $G$ are equivalent:
\begin{itemize}
    \item $S_1 \adperp S_2$ in $(G,\mathcal{E}^{\mathrm{LR}}_G)$,
    \item $\mu(S_1) \adperp \mu(S_2)$ in $(\Omega = K \backslash G/K,\mathcal{E}^{d_{\Omega}})$.
\end{itemize}

In particular, by the arguments in Example \ref{example:Kobayashipitchfork}, if $(S_1,S_2)=(H,L)$ is a pair of closed subgroups of $G$, then the following conditions are equivalent:
\begin{enumerate}
    \item $H \adperp L$ in $(G,\mathcal{E}^{\mathrm{LR}}_G)$.
    \item The $H$-action on the homogeneous space $G/L$ is proper.
    \item The $L$-action on the homogeneous space $G/H$ is proper.
    \item The diagonal $G$-action on $G/L \times G/H$ is proper.
    \item $\mu(S_1) \adperp \mu(S_2)$ in $(\Omega = K \backslash G/K,\mathcal{E}^{d_{\Omega}})$.
\end{enumerate}

Here $\Omega := K \backslash G/K \simeq W(G,K)\backslash \mathfrak{a}$ (see \cite[Section 6]{OgawaOkuda2025Kob60} for details). The equivalence above essentially coincides with the theorem established by Kobayashi \cite{Kobayashi89} when $H$ and $L$ are reductive subgroups, and independently by Kobayashi \cite{Kobayashi96} and Benoist \cite{Benoist96} in the general case. This theorem, commonly known as the ``properness criterion,'' is a fundamental tool in the study of discontinuous groups acting on reductive homogeneous spaces with noncompact isotropy subgroups.

The proof strategy described above is primarily a coarse-geometric reinterpretation of the argument in Kobayashi \cite{Kobayashi96}. A proof of the coarse equivalence property of $\mu$ introduced in Example \ref{example:Cartan} will be given elsewhere.
\end{example}

\section*{Acknowledgements}

This paper is based on a talk given by the fourth author at the 10th China--Japan Geometry Conference. The authors would like to express their sincere gratitude to the organizers of the conference and to the editors of the proceedings volume for their efforts in organizing the conference and preparing this publication.

The authors' interest in coarse geometry was motivated by the observation that several arguments appearing in the study of discontinuous groups on homogeneous spaces, particularly those concerning proper group actions in \cite{Kobayashi96}, admit natural interpretations from the viewpoint of coarse geometry. The authors are especially grateful to Toshiyuki Kobayashi, whose work inspired this line of research. The fourth author was a student of Toshiyuki Kobayashi and benefited greatly from his guidance, not only in the theory of discontinuous groups but also in many other aspects of mathematics.

The authors have benefited greatly from discussions on coarse geometry with many colleagues. In particular, they would like to thank Tomohiro Fukaya, Yutaro Kamigaki, Kazuki Kannaka, Atsushi Matsuo, Takumi Matsuka, Masato Mimura, Ayato Mitsuishi, Shunsuke Miyauchi, Yosuke Morita, Shin-ichi Oguni, Masashi Omoto, Takashi Shioya, Hiroshi Tamaru, Koichi Tojo, Kotaro Mine, Takamitsu Yamauchi, and Ibuki Yonezawa for valuable discussions and insightful comments.

\appendix

\section{Proofs of Some Propositions in Section \ref{subsection:Mslice}}\label{section:app:proofMslice}

In this appendix, we provide proofs of several propositions concerning the slice categories and the orthogonal factorization system introduced in Section \ref{subsection:Mslice}.

\subsection{Cocompleteness of Slice Categories}

Let $I$ be a $\mathcal{V}$-small category. We shall write the endofunctor $\Cat(I,-)$ of $\Cat$ (see \cite[Chapter~II, Section~5]{MacLane1971Working}), given by forming the functor category from $I$, simply in the form
\begin{align*}
    \mathcal{C} &\mapsto \mathcal{C}^I \\
    (F : \mathcal{C} \rightarrow_\Cat \mathcal{D}) &\mapsto (F^I : \mathcal{C}^I \rightarrow_\Cat \mathcal{D}^I).
\end{align*}
This endofunctor naturally induces an endofunctor of $\hCat$.

Henceforth, fix $\mathcal{C} \in \Ob(\Cat)$. We write the diagonal functor $\Delta : \mathcal{C} \rightarrow \mathcal{C}^I$ (cf.~\cite[Definition 3.1.1]{Riehl2016category}) using notation such as
\begin{align*}
    X &\mapsto \Delta_X \\
    (f : X \rightarrow_\mathcal{C} Y) &\mapsto (\Delta_f : \Delta_X \rightarrow_{\mathcal{C}^I} \Delta_Y).
\end{align*}
We say that $\mathcal{C}$ is $I$-cocomplete if the diagonal functor $\Delta$ admits a left adjoint $\lambda^\mathcal{C}$ (cf.~\cite[Proposition 4.6.1]{Riehl2016category}). 
For $I$-cocomplete categories $\mathcal{C}$ and $\mathcal{D}$, 
we say that a functor $F : \mathcal{C} \rightarrow_\Cat \mathcal{D}$ preserves $I$-colimits if the following equality holds in $\hCat$:
\[
[F] \circ [\lambda^\mathcal{C}] = [\lambda^\mathcal{D}] \circ [F^I].
\]

For each object $X$ of $\mathcal{C}$, let $U_X$ denote the forgetful functor from the slice category $\mathcal{C}/X$ to $\mathcal{C}$. Thus, for $h \in \Ob(\mathcal{C}/X)$, we set $U_X(h) = s(h)$, and for $\phi : h_1 \rightarrow_{\mathcal{C}/X} h_2$, we set
\[
U_X(\phi) = (\phi : s(h_1) \rightarrow_{\mathcal{C}} s(h_2)).
\]
Here, for a morphism $h$ in $\mathcal{C}$, $s(h)$ denotes the source of $h$.

The following proposition identifies $(\mathcal{C}/X)^I$ with $\mathcal{C}^I/\Delta_X$:

\begin{proposition}\label{proposition:CXI}
Fix an object $X$ of $\mathcal{C}$.
\begin{enumerate}
    \item Let $D_X : I \rightarrow_{\Cat} (\mathcal{C}/X)$. Put $\widetilde{D_X} := U_X \circ D_X \in \Ob(\mathcal{C}^I)$. Then the assignment
    \[
    \Xi(D_X) : \Ob(I) \rightarrow \Mor(\mathcal{C}), \quad i \mapsto D_X(i)
    \]
    defines a natural transformation from $\widetilde{D_X}$ to $\Delta_X$. In particular, $\Xi(D_X) \in \Ob(\mathcal{C}^I/{\Delta_X})$.
    \item Fix a natural transformation $\Phi : D_X \rightarrow_{(\mathcal{C}/X)^I} D'_X$. Then the horizontal composition
    \[
    \Xi(\Phi) := \id_{U_X} * \Phi : \widetilde{D_X} = U_X \circ D_X \rightarrow_{\mathcal{C}^I} \widetilde{D'_X} = U_X \circ D'_X
    \]
    satisfies
    \[
    \Xi(D_X) = \Xi(D'_X) \circ \Xi(\Phi)
    \]
    in $\mathcal{C}^I$. In particular, $\Xi(\Phi) : \Xi(D_X) \rightarrow_{\mathcal{C}^I/\Delta_X} \Xi(D'_X)$.
    \item The assignments $D_X \mapsto \Xi(D_X)$ and $\Phi \mapsto \Xi(\Phi)$ define an isomorphism of categories $\Xi_X$ from $(\mathcal{C}/X)^I$ to $\mathcal{C}^I/\Delta_X$.
    \item Via the isomorphism $\Xi_X$, the diagonal functor $\Delta^{X} : \mathcal{C}/X \rightarrow (\mathcal{C}/X)^I \simeq \mathcal{C}^I/\Delta_X$ is described as follows:
    \begin{align*}
    \Ob(\mathcal{C}/X) &\rightarrow \Ob(\mathcal{C}^I/\Delta_X), \quad h \mapsto \Delta^{X}_h : \Ob(I) \rightarrow \Mor(\mathcal{C}), \quad i \mapsto h, \\
    \Mor(\mathcal{C}/X) &\rightarrow \Mor(\mathcal{C}^I/\Delta_X), \quad \phi \mapsto \Delta^{X}_\phi : \Ob(I) \rightarrow \Mor(\mathcal{C}), \quad i \mapsto \phi.
    \end{align*}
\end{enumerate}
\end{proposition}

\begin{proposition}\label{proposition:fsharpI}
Let $f : X \rightarrow_{\mathcal{C}} Y$. Under the isomorphisms $\Xi_X : (\mathcal{C}/X)^I \simeq \mathcal{C}^I/\Delta_X$ and $\Xi_Y : (\mathcal{C}/Y)^I \simeq \mathcal{C}^I/\Delta_Y$, the functor $(f_\sharp)^I : (\mathcal{C}/X)^I \rightarrow (\mathcal{C}/Y)^I$ corresponds to $(\Delta_f)_\sharp : \mathcal{C}^I/{\Delta_X} \rightarrow \mathcal{C}^I/{\Delta_Y}$; that is,
\[
\Xi_Y \circ (f_\sharp)^I = (\Delta_f)_\sharp \circ \Xi_X
\]
holds in $\Cat$.
\end{proposition}

Propositions \ref{proposition:CXI} and \ref{proposition:fsharpI} follow directly from the definitions.

Proposition \ref{proposition:CXFCC} follows from the following theorem:

\begin{theorem}\label{theorem:Icocompletenessofslice}
Suppose that $\mathcal{C}$ is $I$-cocomplete. Fix a left adjoint $\lambda$ to the diagonal functor $\Delta : \mathcal{C} \rightarrow \mathcal{C}^I$, together with the natural isomorphism associated with this adjunction
\[
\Theta : \mathcal{C}^I(-,\Delta -) \rightarrow \mathcal{C}(\lambda -,-).
\]
\begin{enumerate}
    \item Fix an object $X$ of $\mathcal{C}$. We define a functor $\lambda^X : \mathcal{C}^I/\Delta_X \rightarrow \mathcal{C}/X$ as follows. For each object $\alpha : D \rightarrow_{\mathcal{C}^I} \Delta_X$ of the slice category $\mathcal{C}^I/{\Delta_X}$, put $\lambda^X(\alpha) := \Theta_{D,X}(\alpha) : \lambda(D) \rightarrow_{\mathcal{C}} X$. For each morphism $\Phi : \alpha_1 \rightarrow_{\mathcal{C}^I/{\Delta_X}} \alpha_2$, put $D_i := s(\alpha_i)$ in $\mathcal{C}^I$. Since $\Phi$ is a morphism in $\mathcal{C}^I$ from $D_1$ to $D_2$, we have $\lambda(\Phi) : \lambda(D_1) \rightarrow_{\mathcal{C}} \lambda(D_2)$. Note that $\lambda^X(\alpha_2) \circ \lambda(\Phi) = \lambda^X(\alpha_1)$. We put
    \[
    \lambda^X(\Phi) := \lambda(\Phi) : \lambda^X(\alpha_1) \rightarrow_{\mathcal{C}/X} \lambda^X(\alpha_2).
    \]
    Then $\lambda^X$ is a left adjoint to $\Delta^X$.
    \item Let $f : X \rightarrow_{\mathcal{C}} Y$. Then the following equality holds in $\Cat$:
    \[
    f_\sharp \circ \lambda^X = \lambda^Y \circ (\Delta_f)_\sharp.
    \]
    In particular, $f_\sharp : \mathcal{C}/X \rightarrow \mathcal{C}/Y$ preserves $I$-colimits.
\end{enumerate}
\end{theorem}

Theorem \ref{theorem:Icocompletenessofslice} also follows directly from the definitions.

\subsection{The Functor Associated with the $\mathcal{M}$-Slice Categories}

Fix a $\mathcal{V}$-small category $\mathcal{C}$ and an orthogonal factorization system $(\mathcal{E},\mathcal{M})$. Put
\[
\mathcal{M} \times_{s,t} \mathcal{E} := \{ (m,e) \in \mathcal{M} \times \mathcal{E} \mid s(m) = t(e) \},
\]
where $s(m)$ and $t(e)$ denote the source of the morphism $m$ and the target of the morphism $e$, respectively. By the definition of an orthogonal factorization system, the map
\[
\mathcal{M} \times_{s,t} \mathcal{E} \rightarrow \Mor(\mathcal{C}), \quad (m,e) \mapsto m \circ e
\]
is surjective.

\begin{setting}\label{setting:section_EMfact}
Fix a section
\[
\Theta : \Mor(\mathcal{C}) \rightarrow \mathcal{M} \times_{s,t} \mathcal{E}, \quad h \mapsto \Theta(h) = (m^{\Theta}(h),e^{\Theta}(h)).
\]
\end{setting}

Assume Setting \ref{setting:section_EMfact}. For each object $X$ of $\mathcal{C}$, we construct a functor $\pi^{\Theta}_X : \mathcal{C}/X \rightarrow \mathcal{M}/X$ as follows. For each object $h$ of $\mathcal{C}/X$, put $\pi^\Theta_X(h) = m^{\Theta}(h)$. For each morphism $\phi : h_1 \rightarrow_{\mathcal{C}/X} h_2$, since $e^{\Theta}(h_1) \perp m^{\Theta}(h_2)$, there exists a unique morphism $\pi^\Theta_X(\phi)$ in $\mathcal{C}$ from $s(m^{\Theta}(h_1))$ to $s(m^{\Theta}(h_2))$ such that the following diagram commutes:
\[
\xymatrix@C=5.5em{
s(h_1) \ar[r]^{e^{\Theta}(h_1)} \ar[d]^{\phi} & s(m^{\Theta}(h_1)) \ar[r]^{m^{\Theta}(h_1)} \ar@{-->}[d]^{\pi^\Theta_X(\phi)} & X \ar[d]^{\id_X} \\
s(h_2) \ar[r]^{e^{\Theta}(h_2)} & s(m^{\Theta}(h_2)) \ar[r]^{m^{\Theta}(h_2)} & X
}
\]

Proposition \ref{proposition:MXCXreflective} follows from the following proposition together with the uniqueness of left adjoints:

\begin{proposition}
For each object $X$ of $\mathcal{C}$, the correspondences
\[
h \mapsto m^{\Theta}(h),
\qquad
(\phi : h_1 \rightarrow_{\mathcal{C}/X} h_2) \mapsto \bigl(\pi^\Theta_X(\phi) : m^{\Theta}(h_1) \rightarrow_{\mathcal{M}/X} m^{\Theta}(h_2)\bigr)
\]
define a functor $\pi^\Theta_X : \mathcal{C}/X \rightarrow \mathcal{M}/X$, which is left adjoint to the inclusion functor $\iota_X : \mathcal{M}/X \rightarrow \mathcal{C}/X$.
\end{proposition}

\begin{proof}
The fact that the above assignments define a functor follows from the uniqueness of lifts. We show that $\pi^\Theta_X$ is left adjoint to the inclusion functor. Let $h$ be an object of $\mathcal{C}/X$, and let $m$ be an object of $\mathcal{M}/X$. For each $d : m^{\Theta}(h) \rightarrow_{\mathcal{M}/X} m$, the composite $d \circ e^{\Theta}(h)$ is a morphism $h \rightarrow_{\mathcal{C}/X} m$. Since $e^{\Theta}(h) \perp m$, the existence and uniqueness of lifts imply that $d \mapsto d \circ e^{\Theta}(h)$ defines a bijection from $\mathcal{M}/X(m^{\Theta}(h),m)$ to $\mathcal{C}/X(h,m)$. The naturality of this bijection follows from the definition of $\pi^\Theta_X$.
\end{proof}

The following proposition is a key ingredient in the proof of Proposition \ref{proposition:functfM}:

\begin{proposition}\label{proposition:fM}
Fix a morphism $f : X \rightarrow_{\mathcal{C}} Y$. Put $f^\mathcal{M}_\sharp := \pi_Y^\Theta \circ f_\sharp \circ \iota_X$. Then the following equality holds in $\hCat$:
\[
[f^\mathcal{M}_\sharp] \circ [\pi^\Theta_X] = [\pi^\Theta_Y] \circ [f_\sharp].
\]
\end{proposition}

\begin{proof}[Proof of Proposition \ref{proposition:fM}]
We construct a natural isomorphism from $\pi^\Theta_Y \circ f_\sharp$ to $f^\mathcal{M}_\sharp \circ \pi^\Theta_X$. First, for each object $h$ of $\mathcal{C}/X$, we have $(\pi^\Theta_Y \circ f_\sharp)(h) = m^\Theta(f \circ h)$ and $(f^\mathcal{M}_\sharp \circ \pi^\Theta_X)(h)= m^\Theta(f \circ m^\Theta(h))$. For simplicity, write $(e,m) := (e^\Theta(h), m^\Theta(h))$, $(e',m') := (e^\Theta(f \circ h), m^\Theta(f \circ h))$, and $(e'',m'') = (e^\Theta(f \circ m),m^\Theta(f \circ m))$. Then
\[
m' \circ e' = f \circ h = m'' \circ (e'' \circ e).
\]
Therefore, by Proposition \ref{proposition:uniqueness_of_factrization}, there exists a unique isomorphism $d_h$ such that $d_h \circ e' = e'' \circ e$ and $m'' \circ d_h = m'$.

It remains to show that the assignment $\Phi : h \mapsto d_h$ defines a natural transformation from $\pi^\Theta_Y \circ f_\sharp$ to $f^\mathcal{M}_\sharp \circ \pi^\Theta_X$. Let $\phi : h_1 \rightarrow_{\mathcal{C}/X} h_2$ be a morphism. We prove that
\[
d_{h_2} \circ ((\pi_Y^\Theta \circ f_\sharp)(\phi)) = ((f^\mathcal{M}_\sharp \circ \pi^\Theta_X)(\phi)) \circ d_{h_1}.
\]
For simplicity, write $(e_i,m_i) := (e^\Theta(h_i), m^\Theta(h_i))$, $(e'_i,m'_i) := (e^\Theta(f \circ h_i), m^\Theta(f \circ h_i))$, and $(e''_i,m''_i) = (e^\Theta(f \circ m_i),m^\Theta(f \circ m_i))$ for $i=1,2$. Recall that $(\pi_Y^\Theta \circ f_\sharp)(\phi)$ is the unique $\mathcal{C}$-morphism satisfying
\[
(\pi_Y^\Theta \circ f_\sharp)(\phi) \circ e'_1 = e'_2 \circ \phi.
\]
It is therefore enough to prove that
\[
(d_{h_2}^{-1} \circ ((f^\mathcal{M}_\sharp \circ \pi^\Theta_X)(\phi)) \circ d_{h_1}) \circ e'_1 = e'_2 \circ \phi.
\]
This follows from a diagram chase. Indeed,
\begin{align*}
(d_{h_2}^{-1} &\circ ((f^\mathcal{M}_\sharp \circ \pi^\Theta_X)(\phi)) \circ d_{h_1}) \circ e'_1 \\
    &= d_{h_2}^{-1} \circ ((f^\mathcal{M}_\sharp \circ \pi^\Theta_X)(\phi)) \circ (d_{h_1} \circ e'_1) \\
    &= d_{h_2}^{-1} \circ ((\pi_Y^\Theta \circ f_\sharp \circ \iota_X)(\pi^\Theta_X(\phi))) \circ e''_1 \circ e_1 \quad (\text{by the definition of } d_{h_1}) \\
    &= d_{h_2}^{-1} \circ ((\pi_Y^\Theta \circ f_\sharp)(\pi^\Theta_X(\phi))) \circ e''_1 \circ e_1 \quad (\text{by the definition of } \iota_X) \\
    &= d_{h_2}^{-1} \circ (\pi_Y^\Theta(\pi^\Theta_X(\phi)) \circ e''_1) \circ e_1 \quad (\text{by the definition of } f_\sharp) \\
    &= d_{h_2}^{-1} \circ e''_2 \circ (\pi^\Theta_X(\phi) \circ e_1) \quad (\text{by the definition of } \pi^\Theta_Y) \\
    &= (d_{h_2}^{-1} \circ (e''_2 \circ e_2)) \circ \phi \quad (\text{by the definition of } \pi^\Theta_X) \\
    &= e'_2 \circ \phi \quad (\text{by the definition of } d_{h_2}).
\end{align*}
\end{proof}

Proposition \ref{proposition:functfM} will now be proved:

\begin{proof}[Proof of Proposition \ref{proposition:functfM}]
It suffices to show that the assignment in Proposition \ref{proposition:functfM} preserves identity morphisms and compositions. In Setting \ref{setting:section_EMfact}, choose the section so that, for every $h \in \mathcal{M}$, one has $(e^\Theta(h),m^\Theta(h)) = (\id_{s(h)},h)$.

First, consider the case where $X = Y$ and $f = \mathrm{id}_X$. Put $f^{\mathcal{M}}_\sharp = \pi_Y^\Theta \circ f_\sharp \circ \iota_X$. We show that $f^\mathcal{M}_\sharp$ is the identity functor on $\mathcal{M}/X$. For each $h \in \Ob(\mathcal{M}/X)$, the equalities $f \circ h = h$ and $m^{\Theta}(f \circ h) = m^{\Theta}(h) = h$ imply that $\pi^\Theta_Y(f \circ h) = h$. In particular, $f^{\mathcal{M}}_\sharp(h) = h$. Next, for each $\phi : h_1 \rightarrow_{\mathcal{M}/X} h_2$, we have
\[
(f_\sharp \circ \iota_X)(\phi) = \phi : h_1 \rightarrow_{\mathcal{C}/X} h_2.
\]
By the definition of $\pi^\Theta_Y$, it follows that $\pi^\Theta_Y(\phi) = \phi$. Hence $f^\mathcal{M}_\sharp(\phi) = \phi$. Thus, when $f = \id_X$, the functor $f^{\mathcal{M}}_\sharp$ is the identity functor. This shows that the assignment in Proposition \ref{proposition:functfM} preserves identity morphisms.

Next, fix morphisms $f : X \rightarrow_{\mathcal{C}} Y$ and $g : Y \rightarrow_{\mathcal{C}} Z$. Put $g^\mathcal{M}_\sharp = \pi^\Theta_Z \circ g_\sharp \circ \iota_Y$, $f^\mathcal{M}_\sharp = \pi^\Theta_Y \circ f_\sharp \circ \iota_X$, and $(g \circ f)^{\mathcal{M}}_\sharp = \pi^\Theta_Z \circ (g \circ f)_\sharp \circ \iota_X$. We show that
\[
[(g \circ f)^{\mathcal{M}}_\sharp] = [g^\mathcal{M}_\sharp] \circ [f^\mathcal{M}_\sharp].
\]
By Proposition \ref{proposition:fM}, we have
\[
[g^{\mathcal{M}}_\sharp] \circ [\pi^\Theta_Y] = [\pi^\Theta_Z] \circ [g_\sharp].
\]
Therefore,
\begin{align*}
[g^{\mathcal{M}}_\sharp] \circ [f^{\mathcal{M}}_\sharp] 
    &= [g^{\mathcal{M}}_\sharp] \circ [\pi^\Theta_Y] \circ [f_\sharp] \circ [\iota_X] \\
    &= [\pi^\Theta_Z] \circ [g_\sharp] \circ [f_\sharp] \circ [\iota_X] \\
    &= [\pi^\Theta_Z] \circ [(g \circ f)_\sharp] \circ [\iota_X] \\
    &= [(g \circ f)^\mathcal{M}_\sharp].
\end{align*}
This proves that the assignment in Proposition \ref{proposition:functfM} also preserves compositions.
\end{proof}

Finally, we prove Theorem \ref{theorem:fMFCC}. We first record the following proposition, which follows, for example, from \cite[Proposition 4.5.15 (ii)]{Riehl2016category}:

\begin{proposition}\label{proposition:FCC_reflect}
Let $\iota : \mathcal{D}\hookrightarrow\mathcal{C}$ be a reflective subcategory, and let $\pi : \mathcal{C} \rightarrow \mathcal{D}$ be a left adjoint to the inclusion functor $\iota$. Let $\mathcal{I}$ be a category, and suppose that $\mathcal{C}$ is $\mathcal{I}$-cocomplete; that is, suppose that the diagonal functor $\Delta^\mathcal{C} : \mathcal{C} \rightarrow \mathcal{C}^\mathcal{I}$ admits a left adjoint $\lambda^\mathcal{C} : \mathcal{C}^\mathcal{I} \rightarrow \mathcal{C}$. Then $\mathcal{D}$ is $\mathcal{I}$-cocomplete, and $\pi \circ \lambda^{\mathcal{C}} \circ \iota^\mathcal{I}$ is a left adjoint to the diagonal functor $\Delta^{\mathcal{D}}$.
\end{proposition}

\begin{proof}[Proof of Theorem \ref{theorem:fMFCC}]
We first prove the assertion \eqref{item:fMFCC:fMsharpFCCpreserve}. The finite cocompleteness of $\mathcal{M}/X$ follows from Propositions \ref{proposition:CXFCC}, \ref{proposition:MXCXreflective}, and \ref{proposition:FCC_reflect}.

We show that $[f^\mathcal{M}_\sharp]$ preserves finite colimits. Let $\mathcal{I}$ be a finite category, and, for $A=X,Y$, let $\Delta^{\mathcal{M}}_A : \mathcal{M}/{A} \rightarrow (\mathcal{M}/{A})^\mathcal{I}$ denote the diagonal functor, with a fixed left adjoint $\lambda^\mathcal{M}_A$. It suffices to prove the following equality in $\hCat$:
\[
[\lambda^\mathcal{M}_Y] \circ [(f^\mathcal{M}_\sharp)^\mathcal{I}] = [f^\mathcal{M}_\sharp] \circ [\lambda^\mathcal{M}_X].
\]
Since $\pi_Y$ preserves colimits, being a left adjoint, since $f_\sharp$ preserves finite colimits by Proposition \ref{proposition:CXFCC}, and since
\[
[\lambda_X^\mathcal{M}] = [\pi_X] \circ [\lambda_X^\mathcal{C}] \circ [\iota_X^I]
\]
and
\[
[f^\mathcal{M}_\sharp] \circ [\pi_X] = [\pi_Y] \circ [f_\sharp]
\]
hold by Propositions \ref{proposition:fM} and \ref{proposition:FCC_reflect}, we obtain
\begin{align*}
[f^\mathcal{M}_\sharp] \circ [\lambda^\mathcal{M}_X]
    &= [f^\mathcal{M}_\sharp] \circ [\pi_X] \circ [\lambda^\mathcal{C}_X] \circ [\iota^\mathcal{I}_X] \\
    &= [\pi_Y] \circ [f_\sharp] \circ [\lambda^\mathcal{C}_X] \circ [\iota^\mathcal{I}_X] \\
    &= [\pi_Y] \circ [\lambda^\mathcal{C}_Y] \circ [f_\sharp^\mathcal{I}] \circ [\iota^\mathcal{I}_X] \\
    &= [\lambda^\mathcal{M}_Y] \circ [\pi_Y^{\mathcal{I}}] \circ [f_\sharp^\mathcal{I}] \circ [\iota^\mathcal{I}_X] \\
    &= [\lambda^\mathcal{M}_Y] \circ [(f^\mathcal{M}_\sharp)^\mathcal{I}].
\end{align*}
Thus, $f^\mathcal{M}_\sharp$ preserves finite colimits.

We next prove \eqref{item:fMFCC:Mono}. Since every morphism in $\mathcal{M}$ is a mono-morphism, its slice categories are thin. The remaining assertion follows from Proposition \ref{proposition:MXCXreflective} and the definition of $[f^\mathcal{M}_\sharp]$.
\end{proof}

\section{Epi--mono factorizations in $\Coarse_{\CtrTotRel}/{\close}$}\label{section:EpiMono_factorization}

As explained in Theorem \ref{theorem:EM_coarse}, if $\mathcal{E}$ and
$\mathcal{M}$ denote the classes of all epi-morphisms and all mono-morphisms
in the category of coarse spaces
$\Coarse_{\CtrMap}/{\close}$, then
$(\mathcal{E},\mathcal{M})$
forms an orthogonal factorization system.

Since
$\Coarse_{\CtrMap}/{\close}$
and
$\Coarse_{\CtrTotRel}/{\close}$
are isomorphic categories by
Theorem \ref{theorem:CtrMapClo_isom_CtrTotRelClo},
it follows that the pair consisting of all epi-morphisms
$\mathcal{E}$
and all mono-morphisms
$\mathcal{M}$
also forms an orthogonal factorization system on
$\Coarse_{\CtrTotRel}/{\close}$.

In particular, every morphism in
$\Coarse_{\CtrTotRel}/{\close}$
admits an epi--mono factorization, and every pair consisting of an
epi-morphism and a mono-morphism satisfies the lifting property.
In this subsection, we describe explicit constructions of epi--mono
factorizations and diagonal fillers in
$\Coarse_{\CtrTotRel}/{\close}$.

The following proposition gives an explicit epi--mono factorization of an
arbitrary morphism in
$\Coarse_{\CtrTotRel}/{\close}$.

\begin{proposition}\label{proposition:epimonofucttame}
Let $(X,\mathcal{E}_X)$ and $(Y,\mathcal{E}_Y)$ be coarse spaces, and fix
$\eta \in \CtrTotRel(X,Y)$.

Define a coarse structure on $X$ by 
$\mathcal{E}_X^\eta
:=
\langle
\Ad_{\eta^\dagger}(\mathcal{E}_Y)
\rangle$,
and put 
$X^\eta
:=
(X,\mathcal{E}_X^\eta)$.

Then the following hold:
\begin{enumerate}
\item
The identity map
$\id_X = 1_X^{\Rel}$
on the underlying set $X$ defines a surjective controlled map
from
$X=(X,\mathcal{E}_X)$
to
$X^\eta=(X,\mathcal{E}_X^\eta)$.

\item
The relations
$\eta \in \Rel(X,Y)$
and
$\eta^\dagger \in \Rel(Y,X)$
are both controlled with respect to
$(\mathcal{E}_X^\eta,\mathcal{E}_Y)$
and
$(\mathcal{E}_Y,\mathcal{E}_X^\eta)$,
respectively.
\end{enumerate}

In particular,
$([\id_X]_{X,X^\eta},[\eta]_{X^\eta,Y})$
gives an epi--mono factorization of the morphism
$[\eta]_{X,Y}$
in
$\Coarse_{\CtrTotRel}/{\close}$.
\end{proposition}

\begin{proof}
The final assertion follows from
Theorems \ref{thm:epi-characterization}
and
\ref{thm:mono-characterization}.
It therefore suffices to show that
\[
\id_X \in \CtrRel(X,X^\eta),
\qquad
\eta \in \CtrRel(X^\eta,Y),
\qquad
\eta^\dagger \in \CtrRel(Y,X^\eta).
\]

We first prove that
$\id_X \in \CtrRel(X,X^\eta)$,
or equivalently,
$\mathcal{E}_X \subset \mathcal{E}_X^\eta$.
Take any
$E\in\mathcal{E}_X$.
Since
$\eta\in\CtrRel(X,Y)$,
we have
$\Ad_\eta(E)\in\mathcal{E}_Y$.
By the definition of
$\mathcal{E}_X^\eta$,
$\Ad_{\eta^\dagger}(\Ad_\eta E)
=
\Ad_{\eta^\dagger\circ\eta}E
\in
\mathcal{E}_X^\eta$.
Thus it suffices to prove that
$E\subset \Ad_{\eta^\dagger\circ\eta}E$.
Since
$\eta$
is total,
Proposition \ref{proposition:total_condition}
implies that
$\eta^\dagger\circ\eta
\supset
1_X^{\Rel}$.
Hence
\[
\Ad_{\eta^\dagger\circ\eta}E
=
(\eta^\dagger\circ\eta)
\circ
E
\circ
(\eta^\dagger\circ\eta)
\supset
E.
\]
This proves that
$\mathcal{E}_X\subset\mathcal{E}_X^\eta$.

Next we show that
$\eta\in\CtrRel(X^\eta,Y)$.
By Proposition \ref{proposition:generator_controlled},
it suffices to prove that
$\Ad_\eta(1_X^{\Rel})\in\mathcal{E}_Y$
and
$\Ad_\eta(\Ad_{\eta^\dagger}(\mathcal{E}_Y))
\subset
\mathcal{E}_Y$.
Since
$\eta\in\CtrRel(X,Y)$
and
$1_X^{\Rel}\in\mathcal{E}_X$,
we have
$\Ad_\eta(1_X^{\Rel})
=
\eta\circ\eta^\dagger
\in
\mathcal{E}_Y$.
Take any
$F\in\mathcal{E}_Y$.
Then
\[
\Ad_\eta(\Ad_{\eta^\dagger}F)
=
\Ad_{\eta\circ\eta^\dagger}F
=
(\eta\circ\eta^\dagger)
\circ
F
\circ
(\eta\circ\eta^\dagger).
\]
Since
$\eta\circ\eta^\dagger
=
\Ad_\eta(1_X^{\Rel})
\in
\mathcal{E}_Y$,
the right-hand side belongs to
$\mathcal{E}_Y$.
Therefore
$\Ad_\eta(\Ad_{\eta^\dagger}F)\in\mathcal{E}_Y$.

Finally, we show that
$\eta^\dagger\in\CtrRel(Y,X^\eta)$.
By the definition of
$\mathcal{E}_X^\eta$,
$\Ad_{\eta^\dagger}(\mathcal{E}_Y)
\subset
\mathcal{E}_X^\eta$,
which is precisely the statement that
$\eta^\dagger\in\CtrRel(Y,X^\eta)$.
\end{proof}

We now describe an explicit construction of a lift for a pair consisting of an
epi-morphism and a mono-morphism.
Let
$A=(A,\mathcal{E}_A)$,
$B=(B,\mathcal{E}_B)$,
$X=(X,\mathcal{E}_X)$,
and
$Y=(Y,\mathcal{E}_Y)$
be coarse spaces.
Fix
$\pi\in\CtrTotRel(A,B)$,
$\iota\in\CtrTotRel(X,Y)$,
$\xi\in\CtrTotRel(A,X)$,
and
$\eta\in\CtrTotRel(B,Y)$
such that
$[\pi]$ is an epi-morphism,
$[\iota]$ is a mono-morphism, and
$[\iota]\circ[\xi]=[\eta]\circ[\pi]$.

By Theorem \ref{thm:epi-characterization}, one can choose
$F_0\in\mathcal{E}_B$ such that $F_0^\dagger = F_0$ and 
$F_0(\pi(A))=B$.
Put $F:=F_0\cup 1_B^{\Rel}$. 
Then $F^\dagger = F$, 
$F\circ\pi\in\CtrTotRel(A,B)$,
$[F\circ\pi]=[\pi]$,
and
$\pi^\dagger\circ F$ is total.
By Proposition \ref{proposition:close_on_totalrel}, there exist
$E_1,E_2\in\mathcal{E}_Y$ such that
\[
E_1\circ\eta\circ F\circ\pi \supset \iota\circ\xi,
\qquad
\eta\circ\pi \subset E_2\circ\iota\circ\xi.
\]
Put $E:= E_1 \cup E_2^\dagger \in\mathcal{E}_Y$.

\begin{theorem}\label{theorem:explicit_lift_epi_mono}
Under the above setting, define
\[
d:=\iota^\dagger\circ E \circ\eta\circ F \in \Rel(B,X).
\]
Then $d\in\CtrTotRel(B,X)$, and $[d]$ gives a diagonal filler of the
following commutative square:
\[
\xymatrix{
A
\ar[r]^{[\xi]}
\ar[d]_{[\pi]}
&
X
\ar[d]^{[\iota]}
\\
B
\ar[r]_{[\eta]}
\ar@{-->}[ur]^{[d]}
&
Y,
}
\]
that is,
$[d]\circ[\pi]=[\xi]$ and $[\iota]\circ[d]=[\eta]$.
\end{theorem}

\begin{proof}
We first show that $d$ is total.
By Proposition \ref{proposition:total_basic}, it suffices to prove that
$\xi^\dagger\circ d$ is total.
Since $\eta$ is total, Proposition \ref{proposition:total_condition} implies
that $1_B^{\Rel}\subset \eta^\dagger\circ\eta$.
Therefore
\begin{align*}
\xi^\dagger \circ d
    &= \xi^\dagger \circ \iota^\dagger \circ E \circ \eta \circ F \\
    &\supset (E_2 \circ \iota \circ \xi)^\dagger \circ \eta \circ F \\
    &\supset (\eta \circ \pi)^\dagger \circ \eta \circ F \\
    &= \pi^\dagger \circ (\eta^\dagger \circ \eta) \circ F \\
    &\supset \pi^\dagger \circ F.
\end{align*}
Since $\pi^\dagger\circ F$ is total,
$\xi^\dagger\circ d$ is also total by Proposition
\ref{proposition:total_basic}.

Next we show that $d$ is controlled.
Indeed,
$\iota^\dagger\in\CtrRel(Y,X)$ by Theorem
\ref{thm:mono-characterization},
$E \in\mathcal{E}_Y\subset\CtrRel(Y)$,
$\eta\in\CtrRel(B,Y)$, and
$F \in\mathcal{E}_B\subset\CtrRel(B)$.
Hence their composite
$d=\iota^\dagger\circ E \circ\eta\circ F$
belongs to $\CtrRel(B,X)$.
Thus $d\in\CtrTotRel(B,X)$.

We next prove that $[d]\circ[\pi]=[\xi]$.
By Proposition \ref{proposition:close_on_totalrel}, it is enough to show
that $d\circ\pi\supset\xi$.
Since $\iota$ is total, Proposition \ref{proposition:total_condition} gives
$1_X^{\Rel}\subset\iota^\dagger\circ\iota$.
Thus
\begin{align*}
d\circ\pi
    &= \iota^\dagger\circ E \circ\eta\circ F \circ\pi \\
    &\supset \iota^\dagger\circ E_1 \circ\eta\circ F\circ\pi \\
    &\supset \iota^\dagger\circ\iota\circ\xi \\
    &\supset \xi .
\end{align*}
Hence $[d]\circ[\pi]=[\xi]$.

Finally, we prove that $[\iota]\circ[d]=[\eta]$.
Since $[\pi]$ is an epi-morphism, it suffices to show that
$([\iota]\circ[d])\circ[\pi]=[\eta]\circ[\pi]$.
Using the equality $[d]\circ[\pi]=[\xi]$ proved above and the assumed
commutativity $[\iota]\circ[\xi]=[\eta]\circ[\pi]$, we obtain
\[
([\iota]\circ[d])\circ[\pi]
=
[\iota]\circ[\xi]
=
[\eta]\circ[\pi].
\]
Therefore $[\iota]\circ[d]=[\eta]$.
\end{proof}

\providecommand{\bysame}{\leavevmode\hbox to3em{\hrulefill}\thinspace}
\providecommand{\MR}{\relax\ifhmode\unskip\space\fi MR }
\providecommand{\MRhref}[2]{%
  \href{http://www.ams.org/mathscinet-getitem?mr=#1}{#2}
}
\providecommand{\href}[2]{#2}

\end{document}